\documentclass{amsart}

\usepackage{amssymb}
\usepackage{amsmath}
\usepackage{amsfonts}
\usepackage{geometry}
\usepackage{bbm}
\usepackage{hyperref}
\usepackage{tikz}
\usepackage{mathrsfs}
\usetikzlibrary{matrix,arrows,decorations.pathmorphing}

\newcounter{cprop}[section]

\newtheorem{theorem}[cprop]{Theorem}
\newtheorem*{theorem*}{Theorem}

\theoremstyle{plain}

\newtheorem{corollary}[cprop]{Corollary}

\newtheorem{lemma}[cprop]{Lemma}
\newtheorem{proposition}[cprop]{Proposition}

\numberwithin{equation}{section}

\theoremstyle{definition}
\newtheorem{definition}[cprop]{Definition}
\newtheorem{example}[cprop]{Example}

\theoremstyle{remark}
\newtheorem{remark}[cprop]{Remark}

\newcommand{\E}{\mathbb{E}}
\renewcommand{\P}{\mathbb{P}}

\newcommand{\R}{\mathbb{R}}
\newcommand{\Q}{\mathbb{Q}}
\newcommand{\N}{\mathbb{N}}
\newcommand{\Z}{\mathbb{Z}}

\newcommand{\vertiii}[1]{{\left\vert\kern-0.25ex\left\vert\kern-0.25ex\left\vert #1 
    \right\vert\kern-0.25ex\right\vert\kern-0.25ex\right\vert}}

\begin{document}
\title[Dynamical theory for SDDE]{A dynamical theory for singular stochastic delay differential equations I: Linear equations and a Multiplicative Ergodic Theorem on fields of Banach spaces}

\author{M. Ghani Varzaneh}
\address{Mazyar Ghani Varzaneh\\
Institut f\"ur Mathematik, Technische Universit\"at Berlin, Germany and Department of Mathematical Sciences, Sharif University of Technology, Tehran, Iran}
\email{mazyarghani69@gmail.com}

\author{S. Riedel}
\address{Sebastian Riedel \\
Institut f\"ur Mathematik, Technische Universit\"at Berlin, Germany}
\email{riedel@math.tu-berlin.de}

\author{M. Scheutzow}
\address{Michael Scheutzow \\
Institut f\"ur Mathematik, Technische Universit\"at Berlin, Germany}
\email{ms@math.tu-berlin.de}



\begin{abstract}
  We show that singular stochastic delay differential equations (SDDEs) induce cocycle maps on a field of Banach spaces. A general Multiplicative Ergodic Theorem on fields of Banach spaces is proved and applied to linear SDDEs. In Part II of this article, we use our results to prove a stable manifold theorem for non-linear singular SDDEs.
\end{abstract}

\maketitle

\section*{Introduction}

Stochastic delay differential equations (SDDEs) describe stochastic processes for which the dynamics do not only depend on the present state, but may depend on the whole past of the process. In its simplest formulation, an SDDE takes the form 
\begin{align}\label{eqn:SDDE_intro}
 dy_t = b(y_t,y_{t-r})\, dt + \sigma(y_t,y_{t-r})\, dB_t(\omega)
\end{align}
for some delay $r > 0$ where $B$ is a Brownian motion, $b$ is the drift and $\sigma$ the diffusion coefficient, both depending on the present and a delayed state of the system. In this case, we speak of a (single) \textit{discrete time delay}. SDDE appear frequently in practice. For instance, they can be used to model cell population growth and neural control mechanisms, cf.~\cite{Buc00} and the references therein, they are applied in financial modeling \cite{Sto05} and for climate models \cite{BTR07}. To be able to solve \eqref{eqn:SDDE_intro} uniquely, an initial condition has to be given which is a path or, more generally, a stochastic process. This means that we are led to solve an equation on an infinite dimensional (path) space. Popular choices for spaces of initial conditions are continuous paths or $L^2$ paths. However, standard It\=o theory can be applied without too much effort to solve \eqref{eqn:SDDE_intro} for such initial conditions, cf. \cite{Mao08, Moh84}. 

To analyse the qualitative behaviour of solutions to \eqref{eqn:SDDE_intro}, in particular its long-time behaviour, it is natural to use a dynamical systems approach. Maybe the most popular concept, which was successfully applied to stochastic differential equations (SDEs) in both finite and infinite dimensions, was developed by L.~Arnold and is called the theory of \textit{Random Dynamical Systems} (RDS), cf. \cite{Arn98} for an exposition. Examples for which the language and theory of RDS are used include random attractors \cite{Sch92,CF94,CDF97}, random stable and unstable manifolds \cite{MS99, MS04} and different concepts of stochastic bifurcation \cite[Chapter 9]{Arn98}. The crucial result on which Arnold's theory is built is a \textit{Multiplicative Ergodic Theorem} (MET), originally proved by Oseledec \cite{Ose68}. This theorem has attracted much attention by different researchers and has been proven with different techniques and increasing generality, cf. \cite{Rag79,Rue79,Rue82,Man83,Thi87,LL10,GTQ15,Blu16}. Under certain conditions, the MET shows that linear and linearised Random Dynamical Systems possess a \textit{Lyapunov spectrum} which can be seen as an analogue to the spectrum of eigenvalues of a matrix. Studying the behaviour of the possibly complex RDS can often be reduced to study its Lyapunov spectrum, which is a huge simplification. For a long time, it was believed that the RDS approach can not be used to study SDDE of the form \eqref{eqn:SDDE_intro}. This article claims, that indeed, it is possible.

Let us explain why it was believed that RDS are not applicable for general SDDEs. The idea is to show that certain equations do not generate a continuous \textit{stochastic semi-flow} which is a necessary condition for generating an RDS and to apply the MET. Recall that given a probability space $(\Omega,\mathcal{F},\P)$, a continuous stochastic semi-flow on a topological space $E$ is a measurable map 
\begin{align*}
  \phi \colon \{(s,t) \in [0,\infty)^2\, | \, s\leq t \} \times \Omega \times E \to E
\end{align*}
 such that on a set of full measure $\tilde{\Omega}$, we have $\phi(t,t,\omega,x) = x$ and $\phi(s,u,\omega,x) = \phi(t,u,\omega,\phi(s,t,\omega,x))$ for every $s,t,u \in [0,\infty)$, $s \leq t \leq u$, every $x \in E$ and every $\omega \in \tilde{\Omega}$ and $x \mapsto \phi(s,t,\omega,x)$ is assumed to be continuous for every choice of $s,t \in [0,\infty)$, $s\leq t$, and every $\omega \in \tilde{\Omega}$. Consider the linear equation
\begin{align}\label{eqn:counterexample_intro}
  \begin{split}
 dy_t &= y_{t-1}\, dB_t(\omega); \quad t \geq 0 \\
 y_t &= \xi_t; \quad t \in [-1,0]
  \end{split}
\end{align}
interpreted as an It\=o integral equation. It is clear that the solution on the time interval $[0,1]$ should be given by
\begin{align*}
 y_t = \xi_0 + \int_0^t \xi_{s-1}\, dB_s(\omega)
\end{align*}
whenever the stochastic integral makes sense. However, Mohammed proved in \cite{Moh86} that there is no modification of the process $y$ which depends continuously on $\xi$ in the supremum norm. This rules out the choice of $E = C([-1,0],\R)$ on which a possible semi-flow $\phi$ induced by \eqref{eqn:counterexample_intro} could be defined. At this stage, one might still hope that another choice of $E$ could be a possible state space for our semi-flow. However, we will prove now that there is in fact no such choice. Inspired by \cite[Section 1.5.1]{LCL07}, we make the following definition:

\begin{definition}
  Let $E$ be a Banach space of functions mapping from $[-1,0]$ to $\R$. We say that $E$ \emph{carries the Wiener measure} if the functions $t \mapsto \sin[(n-1/2)\pi t]$ are contained in $E$ for every $n \geq 1$ and if the series
  \begin{align*}
   \sum_{n = 1}^{\infty} Z_n(\omega) \frac{\sin[(n-1/2)\pi t]}{(n-1/2)\pi}, \quad t \in [-1,0]
  \end{align*}
  converges in $E$ almost surely for every sequence $(Z_n)$ of independent, $\mathcal{N}(0,1)$-distributed random variables.
\end{definition}
Note that \emph{carrying the Wiener measure} is indeed a minimum requirement for the state space $E$ of a possible semi-flow induced by \eqref{eqn:counterexample_intro}, otherwise we would not even be able to choose constant paths as initial conditions. However, this assumption already rules out the possibility of the existence of a continuous semi-flow, as the following theorem shows.

\begin{theorem}
   There is no space $E$ carrying the Wiener measure for which the equation \eqref{eqn:counterexample_intro} induces a continuous mapping $I \colon E \to \R$, $I(\xi) = y_1$, on a set of full measure, which extends the pathwise defined mapping for smooth initial conditions.

 \end{theorem}

 \begin{proof}
  Let $(Z_n)$ be a sequence of independent standard normal random variables. Set
  \begin{align*}
   B^N_t(\omega) = \sum_{n = 1}^{N} Z_n(\omega) \frac{\sin[(n-1/2)\pi t]}{(n-1/2)\pi}.
  \end{align*}
  Then $B^N \to B$ as $N \to \infty$ in $\alpha$-H\"older norm, $\alpha < 1/2$, on a set of full measure $\Omega_1$, cf. \eqref{eqn:def_hoelder_norm} where we recall the definition of the H\"older norm and \cite[3.5.1. Theorem]{Bog98} for a general result about Gaussian sequences from which the convergence above follows. Assume that $E$ carries the Wiener measure. Then there is a set of full measure $\Omega_2$ such that the limit
  \begin{align*}
    \sum_{n = 1}^{\infty} \tilde{Z}_n(\omega) \frac{\sin[(n-1/2)\pi t]}{(n-1/2)\pi} =: \lim_{N \to \infty} \tilde{B}^N_t(\omega) =: \tilde{B}_t(\omega)
  \end{align*}
  exists in $E$ for every $\omega \in \Omega_2$ where $\tilde{Z}_n := (-1)^n Z_n$. The theory of Young integration \cite{You36} implies that
  \begin{align*}
   \int_0^1 \tilde{B}^N_t(\omega)\, dB_t^M(\omega) \to \int_0^1 \tilde{B}^N_t(\omega)\, dB_t(\omega)
  \end{align*}
  as $M \to \infty$ for every $\omega \in \Omega_1 \cap \Omega_2$. Noting that $\tilde{Z}_n \sin[(n-1/2)\pi t] = Z_n \cos[(n-1/2)\pi(1+t)]$, we obtain that
  \begin{align*}
   \int_0^1 \tilde{B}^N_t(\omega)\, dB_t^M(\omega) = \sum_{n=1}^N \frac{Z^2_n(\omega)}{(n-1/2)\pi}
  \end{align*}
  for all $M \geq N$. Therefore,
  \begin{align*}
   \int_0^1 \tilde{B}^N_t(\omega)\, dB_t(\omega) = \sum_{n=1}^N \frac{Z^2_n(\omega)}{(n-1/2)\pi} \to \infty
  \end{align*}
  as $N \to \infty$ on a set of full measure $\Omega_3 \subset \Omega_1 \cap \Omega_2$.   
  Now we can argue by contradiction. Assume that there is a set of full measure $\Omega_4$ such that for every $\omega \in \Omega_4$, the map 
  \begin{align*}
   E \ni \xi \mapsto \xi_0 + \int_0^1 \xi_{t-1}\, dB_t(\omega)
  \end{align*}
  is continuous. Since $\Omega_3 \cap \Omega_4$ has full measure, the set is nonempty and we can choose $\omega \in \Omega_3 \cap \Omega_4$. Set $\xi_n := \tilde{B}^n(\omega)$ and $\xi := \tilde{B}(\omega)$. Then we have $\xi_n \to \xi$ in $E$ as $n \to \infty$, but $\int_0^1 \xi^n_{t-1}\, dB_t(\omega)$ diverges as $n \to \infty$ which leads to a contradiction.

 \end{proof}
 
 This theorem shows that \emph{there is no reasonable space of functions on which the SDDE \eqref{eqn:counterexample_intro} induces a continuous semi-flow}, and using RDS to study such equations seems indeed hopeless. Let us however mention here that only a delay in the diffusion part causes the trouble, a delay in a possible drift part would be harmless. For this reason, we will discard the drift in our article and study equations of the form \eqref{eqn:SDDE_intro} with $b=0$ only. We also remark that studying delay equations where the diffusion coefficient may depend on a whole path segment of the solution, so-called \emph{continuous delay}, can lead to easier equations since in that case, the diffusion coefficient might have a smoothing effect. Such equations are called \emph{regular} stochastic delay differential equations, and they can indeed be studied using RDS, cf. \cite{MS96} and \cite{MS97}. The equation \eqref{eqn:counterexample_intro} is an example of \emph{singular} stochastic delay differential equation. 
 
   Let us now explain the idea of the present article. We have seen that there is no space of paths $E$ on which $E \ni \xi \mapsto \int_0^1 \xi_s\, dB_s(\omega)$ is a continuous map on a set of full measure. However, in rough path theory, one knows that there is a family of Banach spaces $\{E_{\omega}\}_{\omega \in \Omega}$ and a set of full measure $\tilde{\Omega}$ such that the maps
  \begin{align*}
   E_{\omega} \ni \xi \mapsto \int_0^1 \xi_s\, d\mathbf{B}_s(\omega)
  \end{align*}
  are continuous for every $\omega \in \tilde{\Omega}$ where the integral has to be interpreted as a \emph{rough paths integral}. Indeed, the spaces $E_{\omega}$ are nothing but the usual spaces of \emph{controlled paths} introduced by Gubinelli in \cite{Gub04} for which we will recall the definition below. Therefore, we can hope to establish a semi-flow property for solutions to \eqref{eqn:counterexample_intro} (and even more general equations) if we allow the state spaces to be random and by interpreting the equation as a delay differential equation driven by a random rough path. Fortunately, Neuenkirch, Nourdin and Tindel already studied delay equations driven by rough paths in \cite{NNT08}, and we can build on their results. Having established such a semi-flow property, the corresponding RDS will involve random spaces as well. This seems hopelessly complicated and maybe unnatural at first sight, but we argue that it is not. It turns out that the structure of such RDS is similar to that which appears when studying the linearisation of an RDS which is induced by an SDE defined on a Riemannian manifold, cf. \cite[Chapter 4]{Arn98}. These RDS act on measurable bundles and are therefore called \emph{bundle RDS}, cf. \cite[Section 1.9]{Arn98}. In a sense, we will see that SDDE induce bundle RDS with the fibres being (infinite dimensional!) spaces of controlled paths. However, it turns out that defining a bundle structure is not necessary since we are only interested in the fibres. Therefore, instead of studying RDS defined on an infinite dimensional bundle, we will study RDS which are defined on \emph{measurable fields of Banach spaces}. After having defined such a structure, the crucial point to ask is whether an MET holds on it. Fortunately, this is indeed the case, and we provide a full proof of such a theorem in the present work. With the MET at hand, we can indeed deduce the existence of a Lyapunov spectrum for linear 
  SDDE. Our main result, which is a combination of Theorem \ref{thm:main_linear} 
  and Corollary \ref{cor:main_thm_bm} to be found in Section \ref{sec:lyapunovspectrum}, can loosely be formulated as follows:
  
  \begin{theorem}
   Linear stochastic delay differential equations of the form
   \begin{align}
    dy_t =  \sigma(y_t,y_{t-r})\, dB_t(\omega)
   \end{align}
   induce linear RDS on measurable fields of Banach spaces given by the spaces of controlled paths defined by $B(\omega)$. Furthermore, an MET applies and provides the existence of a Lyapunov spectrum for the linear RDS. 
  \end{theorem}
  
  In Part II of our paper, we will show that also non-linear equations linearized around equilibrium points induce linear RDS, and prove a stable manifold theorem for such equations. 
  
  Let us finally remark that stochastic differential equations on infinite dimensional spaces frequently lack the semi-flow property. For instance, this is often the case for stochastic partial differential equations (SPDEs), too, cf. e.g. \cite{Fla95} and the references therein. We believe that the approach we present here can be applied also in the context of SPDEs to provide a dynamical systems approach to equations for which the semi-flow property is known not to hold.
  
  \vspace{10pt}

  The article is structured as follows. In Section \ref{sec:basics_rough_delay}, we introduce the techniques to study delay equations driven by rough paths and prove some basic properties. The content of Section \ref{sec:delayed_area_BM} is to show that the Brownian motion can drive rough delay equations and to prove a Wong-Zakai theorem, also in the non-linear case, which might be of independent interest. In Section \ref{sec:RDS_and_delay}, we establish the connection to Arnold's theory and define RDS on measurable fields of Banach spaces. Section \ref{sec:MET_Banach_fields} provides the formulation and the proof of an MET on a field of Banach spaces. The main results of the present paper and a discussion of them are contained in Section \ref{sec:lyapunovspectrum}. Finally, we come back to the example \eqref{eqn:counterexample_intro} and discuss it in more detail in Section \ref{sec:concrete_example}.

\subsection*{Preliminaries and notation}

In this section we collect some notations which will be used throughout the paper.\\ 
\begin{itemize}
 \item If not stated differently, $U$, $V$, $W$ and $\bar{W}$ will always denote finite-dimensional, normed vector spaces over the real numbers, with norm denoted by $|\cdot |$. By $ L(U,W) $ we mean 
the set of linear and continuous functions from $ U $ to $ W $ equipped with usual operator norm.
  \item Let $I$ be an interval in $\R$. A map $ m:I\rightarrow U $ will also be called a \emph{path}. For a path $m$, we denote its increment by $ m_{s,t}=m_{t}-m_{s} $ where by $ m_{t} $ we mean $ m(t) $. We set
\begin{align*}
\Vert m\Vert_{\infty;I}:=\sup_{s\in I}\vert m_{s}\vert
\end{align*}
and define the $ \gamma$-H\"older seminorm, $\gamma \in (0,1]$, by
\begin{align*}
\Vert m\Vert_{\gamma;I} := \sup_{s,t \in I; s \neq t} \frac{\vert m_{s,t}\vert}{\vert t-s\vert^{\gamma}}.
\end{align*}
For a general $2$-parameter function $m^{\#} \colon I \times I \to U$, the same notation is used. We will sometimes omit $I$ as subindex if the domain is clear from the context. The space $C^0(I,U)$ consists of all continuous paths $m \colon I \to U$ equipped with the uniform norm, $C^{\gamma}(I,U)$ denotes the space of all $\gamma$-H\"older continuous functions equipped with the norm
\begin{align}\label{eqn:def_hoelder_norm}
 \Vert \cdot \Vert_{C^{\gamma};I} := \Vert \cdot \Vert_{\infty;I} + \Vert \cdot \Vert_{\gamma;I}.
\end{align}
$C^{\infty}(I,U)$ is the space of all arbitrarily often differentiable functions. If $0 \in I$, using $0$ as subindex such as for $C^{\gamma}_0(I,U)$ denotes the subspace of functions for which $x_0 = 0$. An upper index such as $C^{0,\gamma}(I,U)$ means taking the closure of smooth functions in the corresponding norms.
\end{itemize}

Next, we introduce some basic objects from rough paths theory needed in this article. We refer the reader to \cite{FH14} for a general overview.
\begin{itemize}
 \item  Let  $X \colon \R \to U$ be a locally $\gamma$-H\"older path, $\gamma \in (0,1]$. A \emph{L\'evy area} for $X$ is a continuous function
 \begin{align*}
  \mathbb{X} \colon \R \times \R \to U \otimes U
 \end{align*}
 for which the algebraic identity
 \begin{align*}
   \mathbb{X}_{s,t} = \mathbb{X}_{s,u} + \mathbb{X}_{u,t} + X_{s,u} \otimes X_{u,t} 
 \end{align*}
 is true for every $s,u,t \in \R$ and for which $\| \mathbb{X} \|_{2\gamma ; I} < \infty$
 holds on every compact interval $I \subset \R$. If $\gamma \in (1/3,1/2]$ and $X$ admits L\'evy area $\mathbb{X}$, we call $ \mathbf{X}= \big{(}X, \mathbb{X}\big{)}$ a \emph{$\gamma$-rough path}. If $\mathbf{X}$ and ${\mathbf{Y}}$ are $\gamma$-rough paths, one defines
 \begin{align*}
  \varrho_{\gamma;I}(\mathbf{X},\mathbf{Y}) := \sup_{s,t \in I; s \neq t} \frac{|X_{s,t} - Y_{s,t}|}{|t-s|^{\gamma}} + \sup_{s,t \in I; s \neq t} \frac{|\mathbb{X}_{s,t} - \mathbb{Y}_{s,t}|}{|t-s|^{2 \gamma}}.
 \end{align*}
 
 \item   Let $I = [a,b]$ be a compact interval. A path $m \colon I \to \bar{W}$ is a \emph{controlled path} based on $X$ on the interval $I$ if there exists a $\gamma$-H\"older path $m' \colon I \to L(U,\bar{W})$ such that
\begin{align*}
m_{s,t} = m'_s X_{s,t} + m_{s,t}^{\#}
\end{align*}
for all $s,t \in I$ where $m^{\#} \colon I \times I \to \bar{W}$ satisfies $\| m^{\#} \|_{2\gamma ; I} < \infty$.
The path $m'$ is called a \emph{Gubinelli derivative} of $m$. We use $\mathscr{D}_{X}^{\gamma}(I,\bar{W})$ to denote the space of controlled paths based on $X$ on the interval $I$. It can be shown that this space is a Banach space with norm 
\begin{align*}
 \|m\|_{\mathscr{D}_{X}^{\gamma}} := \|(m,m')\|_{\mathscr{D}_{X}^{\gamma}} := |m_a| + |m'_a| + \|m'\|_{\gamma;I} + \| m^{\#} \|_{2\gamma; I}.
\end{align*}
If $X$ and $\tilde{X}$ are $\gamma$-H\"older paths, $(m,m') \in \mathscr{D}_{X}^{\gamma}(I,\bar{W})$ and $(\tilde{m},\tilde{m}') \in \mathscr{D}_{\tilde{X}}^{\gamma}(I,\bar{W})$, we set 
\begin{align*}
 d_{2\gamma ; I}((m,m'),(\tilde{m},\tilde{m}')) := \|m' - \tilde{m}'\|_{\gamma;I} + \| m^{\#} - \tilde{m}^{\#} \|_{2\gamma; I}.
\end{align*}
If $\bar{W} = \R$, we will also use the noation $\mathscr{D}_{X}^{\gamma}(I)$ instead of $\mathscr{D}_{X}^{\gamma}(I,\R)$.

\end{itemize}

We finally recall the definition of a random dynamical system introduced by L.~ Arnold \cite{Arn98}. 

\begin{itemize}
 \item Let $(\Omega,\mathcal{F})$ and $(X,\mathcal{B})$ be measurable spaces. Let $\mathbb{T}$ be either $\R$ or $\Z$, equipped with a $\sigma$-algebra $\mathcal{I}$ given by the Borel $\sigma$-algebra $\mathcal{B}(\R)$ in the case of $\mathbb{T} = \R$ and by $\mathcal{P}(\Z)$ in the case of $\mathbb{T} = \Z$.
A family $\theta=(\theta_t)_{t \in \mathbb{T}}$ of maps from $\Omega$ to itself is called a \emph{measurable dynamical system} if
\begin{itemize}
   \item[(i)] $(\omega,t) \mapsto \theta_t\omega$ is $\mathcal{F} \otimes \mathcal{I} / \mathcal{F}$-measurable,   \vspace{0.07cm}

   \item[(ii)] $\theta_0 = \operatorname{Id}$,   \vspace{0.07cm}
  
   \item[(iii)] $\theta_{s + t} = \theta_s \circ \theta_t$, for all $s,t \in \mathbb{T}$.   \vspace{0.1cm}
\end{itemize}
If $\mathbb{T} = \mathbb{Z}$, we will also use the notation $\theta := \theta_1$, $\theta^n := \theta_n$ and $\theta^{-n} := \theta_{-n}$ for $n \geq 1$.
If $\P$ is furthermore a probability on $(\Omega,\mathcal{F})$ that is invariant under any of the elements of $\theta$,
$$
\P \circ \theta_t^{-1} = \P
$$ 
for every $t \in \mathbb{T}$, we call the tuple $\big(\Omega, \mathcal{F},\P,\theta\big)$ a \emph{measurable metric dynamical system}. The system is called \emph{ergodic} if every $\theta$-invariant set has probability $0$ or $1$.

\item Let $\mathbb{T}^+ := \{t \in \mathbb{T}\, :\, t \geq 0\}$, equipped with the trace $\sigma$-algebra. An  \emph{(ergodic) measurable random dynamical system} on $(X,\mathcal{B})$ is an (ergodic) measurable metric dynamical system $\big(\Omega, \mathcal{F},\P,\theta\big)$ with a measurable map 
   $$
   \varphi \colon \mathbb{T}^+ \times \Omega \times X \to X
   $$ 
   that enjoys the \emph{cocycle property}, i.e. $\varphi(0,\omega,\cdot) = \operatorname{Id}_X$, for all $\omega \in \Omega$, and
  \begin{align*}
   \varphi(t+s,\omega,\cdot) = \varphi(t,\theta_s\omega,\cdot) \circ \varphi(s,\omega,\cdot)
  \end{align*}
  for all $s,t \in \mathbb{T}^+$ and $\omega \in \Omega$. The map $\varphi$ is called \emph{cocycle}. If $X$ is a topological space with $\mathcal{B}$ being the Borel $\sigma$-algebra and the map $\varphi_{\cdot}(\omega,\cdot) \colon \mathbb{T}^+ \times X \to X$ is continuous for every $\omega \in \Omega$, it is called a \emph{continuous (ergodic) random dynamical system}. In general, we say that \emph{$\varphi$ has property $P$} if and only if $\varphi(t,\omega,\cdot) \colon X \to X$ has property $P$ for every $t \in \mathbb{T}^+$ and $\omega \in \Omega$ whenever the latter statement makes sense.
\end{itemize}


\section{Basic properties of rough delay equations}\label{sec:basics_rough_delay}

In this section, we show how to solve rough delay differential equations and present some basic properties of the solution. 

\subsection{Basic objects, existence, uniqueness and stability}

This section basically summarizes the concepts and results from \cite{NNT08}. We start by introducing ``delayed'' versions of rough paths and controlled paths. Note that, as already mentioned in the introduction, we restrict ourselves to the case of one time delay only. We refer to \cite{NNT08} for corresponding definitions for a finite number of delays.

 \begin{definition}
 Let  $X \colon \R \to U$ be a locally $\gamma$-H\"older path and $r > 0$. A \emph{delayed L\'evy area} for $X$ is a continuous function
 \begin{align*}
   \mathbb{X}(-r) \colon \R \times \R \to U \otimes U
 \end{align*}
 for which the algebraic identity
\begin{align*}
   \mathbb{X}_{s,t}(-r)=\mathbb{X}_{s,u}(-r) + \mathbb{X}_{u,t}(-r)+X_{s-r,u-r}\otimes X_{u,t} 
 \end{align*}
 is true for every $s,u,t \in \R$ and for which $\| \mathbb{X}(-r) \|_{2\gamma ; I} < \infty$
 holds on every compact interval $I \subset \R$. If $\gamma \in (1/3,1/2]$ and $X$ admits L\'evy- and delayed L\'evy area $\mathbb{X}$ and $\mathbb{X}(-r)$, we call $ \mathbf{X}= \big{(}X, \mathbb{X}, \mathbb{X}(-r)\big{)}$ a \emph{delayed $\gamma$-rough path with delay $r > 0$}. If $\mathbf{X}$ and ${\mathbf{Y}}$ are delayed $\gamma$-rough paths, we set 
 \begin{align*}
  \varrho_{\gamma;I}(\mathbf{X},\mathbf{Y}) := \sup_{s,t \in I; s \neq t} \frac{|X_{s,t} - Y_{s,t}|}{|t-s|^{\gamma}} + \sup_{s,t \in I; s \neq t} \frac{|\mathbb{X}_{s,t} - \mathbb{Y}_{s,t}|}{|t-s|^{2 \gamma}} + \sup_{s,t \in I; s \neq t} \frac{|\mathbb{X}(-r)_{s,t} - \mathbb{Y}(-r)_{s,t}|}{|t-s|^{2 \gamma}}.
 \end{align*}

\end{definition}

\begin{remark}\label{eqn:delayed_rp_as_usual_rp}
  For $X$ as in the former definition, set
 \begin{align*}
   Z := (X, X_{\cdot - r}) \in U \oplus U.
 \end{align*}
  If $X$ admits a L\'evy- and delayed L\'evy area, also $Z$ admits a L\'evy area ${\mathbb{Z}}$ given by
  \begin{align*}
    \mathbb{Z} = \begin{pmatrix}
                      \mathbb{X} & \bar{\mathbb{X}}(-r) \\
                      \mathbb{X}(-r) & \mathbb{X}_{\cdot - r, \cdot - r}
                     \end{pmatrix}
  \end{align*}
  where $\bar{\mathbb{X}}^{ij}(-r) := X^i_{s,t} X^{j}_{s-r,t-r} - \mathbb{X}^{ji}_{s,t}(-r)$. Conversely, if $Z$ admits a L\'evy area, the path $X$ admits both L\'evy- and delayed L\'evy area. The delayed L\'evy area can therefore be understood as the usual L\'evy area of a path enriched with its delayed path.

\end{remark}

Next, we recall what is a delayed controlled path.

\begin{definition}
  Let $I = [a,b]$ be a compact interval. A path $m \colon I \to \bar{W}$ is a \emph{delayed controlled path} based on $X$ on the interval $I$ if there exist $\gamma$-H\"older paths $\zeta^0, \zeta^1 \colon I \to L(U,\bar{W})$ such that
\begin{align}\label{eqn:delay_controlled}
m_{s,t} = \zeta^0_s X_{s,t} + \zeta^1_s X_{s-r,t-r} + m^{\#}_{s,t}
\end{align}
for all $s,t \in I$ where $m^{\#} \colon I \times I \to \bar{W}$ satisfies $\| m^{\#} \|_{2\gamma ; I} < \infty$.
The path $(\zeta^0,\zeta^1)$ will again be called \emph{Gubinelli derivative} of $m$. We use $\mathcal{D}_{X}^\gamma(I,\bar{W})$ to denote the space of delayed controlled paths based on $X$ on the interval $I$. A norm on this space can be defined by
\begin{align}\label{eqn:norm_delayed_controlled_path}
 \| m \|_{\mathcal{D}_{X}^{\gamma}} := \|(m,\zeta^0,\zeta^1)\|_{\mathcal{D}_{X}^{\gamma}} := |m_a| + |\zeta^0_a| + |\zeta^1_a| + \|\zeta^0\|_{\gamma;I} + \|\zeta^1\|_{\gamma;I} + \|m^{\#}\|_{2\gamma; I}.
\end{align}
\end{definition}

\begin{remark}\label{remark:delayed_contr_usual_contr}
 Note that any controlled path is also a delayed controlled path (by the choice $\zeta^1 = 0$), but the converse might not be true. However, considering again the enhanced path
  \begin{align*}
    Z = (X,X_{\cdot - r}) \in U \oplus U,
 \end{align*}
 the identity \eqref{eqn:delay_controlled} shows that $m$ is a usual $\bar{W}$-valued controlled path based on $Z$ with Gubinelli derivative $\bar{\zeta} \colon I \to L(U \oplus U, \bar{W})$ given by $\bar{\zeta}_t(v,w) := \zeta^0_t v + \zeta^1_t w$.
\end{remark}

With these objects, we can define an integral as follows.

\begin{theorem}\label{thm:delayed_rough_integral}
 Let $\mathbf{X}= \big{(}X, \mathbb{X}, \mathbb{X}(-r)\big{)}$ be a delayed $\gamma$-rough path and $m$ an $L(U,W)$-valued delayed controlled path based on $X$ with decomposition as in \eqref{eqn:delay_controlled} on the interval $[a,b]$. Then the limit
 \begin{align}\label{dfn}
    \int_{a}^{b} m_{s}\, d\mathbf{X}_s := \lim_{|\Pi| \to 0}  \sum_{t_j \in \Pi} m_{t_{j}} X_{t_{j,t_{j+1}}} + \zeta^{0}_{t_{j}} \mathbb{X}_{t_{j},t_{j+1}} + \zeta^{1}_{t_{j}} \mathbb{X}_{t_{j},t_{j+1}}(-r)
\end{align}
exists where $\Pi$ denotes a partition of $[a,b]$. Moreover, there is a constant $C$ depending on $\gamma$ and $(b-a)$ only such that for all $s < t \in [a,b]$, the estimate
\begin{align*}
 &\left| \int_s^t m_{u}\, d\mathbf{X}_u - m_s X_{s,t} - \zeta^{0}_s \mathbb{X}_{s,t} - \zeta^{1}_{s} \mathbb{X}_{s,t}(-r) \right| \\
 &\quad \leq C\left( \|m^{\#} \|_{2\gamma} \|X\|_{\gamma}  + \|\zeta^0\|_{\gamma} \|\mathbb{X}\|_{2\gamma}  +  \|\zeta^1\|_{\gamma} \|\mathbb{X}(-r)\|_{2\gamma} \right)|t-s|^{3 \gamma}
\end{align*}
holds. In particular,
\begin{align*}
 t \mapsto \int_s^t m_{u}\, d\mathbf{X}_u
\end{align*}
is controlled by $X$ with Gubinelli derivative $m$.

\end{theorem}

\begin{proof}
  This is just an application of the Sewing lemma, cf. e.g. \cite[Lemma 4.2]{FH14}, applied to
  \begin{align*}
    \Xi_{s,t} = m_s X_{s,t} + \zeta^{0}_s \mathbb{X}_{s,t} + \zeta^{1}_{s} \mathbb{X}_{s,t}(-r).
  \end{align*}

\end{proof}

\begin{example}\label{ex:rough_delay_linear}
 Let $U = W = \R$ and $\mathbf{X}= \big{(}X, \mathbb{X}, \mathbb{X}(-1)\big{)}$ be a delayed $\gamma$-rough path. We aim to solve the equation
 \begin{align}\label{eqn:delayed_linear}
  \begin{split}
  dy_t &= y_{t-1}\, d\mathbf{X}_t;\quad t \geq 0 \\
  y_t &= \xi_t; \quad t \in [-1,0].
  \end{split}
 \end{align}
 If $\xi \in \mathscr{D}^{\gamma}_X([-1,0])$, the path $[0,1] \ni t \mapsto \xi_{t-1}$ is a delayed controlled path, thus the integral
 \begin{align*}
  [0,1] \ni t \mapsto \int_0^t \xi_{s - 1}\, d \mathbf{X}_s
 \end{align*}
 exists. Therefore, the path
 \begin{align*}
  y_t := \begin{cases}
  \xi_t &\text{ if } t\in [-1,0] \\
  \int_0^t \xi_{s - 1}\, d \mathbf{X}_s + \xi_0 &\text{ if } t\in [0,1] 
  \end{cases}
 \end{align*}
  is the unique continuous solution to \eqref{eqn:delayed_linear} on $[-1,1]$. Since the integral is again an element in $\mathscr{D}^{\gamma}_X([0,1])$, we can iterate the procedure to solve \eqref{eqn:delayed_linear} on the whole positive real line.

\end{example}

We will need the following class of vector fields:

\begin{definition}
 By $C^3_b(W^2, L(U,W))$, we denote the space of bounded functions $\sigma \colon W \oplus W \to L(U,W)$ possessing 3 bounded derivatives.
\end{definition}

We can now state the first existence and uniqueness result for rough delay equations.

\begin{theorem}[Neuenkirch, Nourdin, Tindel]\label{thm:delay_existence}
 For $r > 0$, let $\mathbf{X}$ be a delayed $\gamma$-rough path for $\gamma \in (1/3,1/2]$, $\sigma \in C^3_b(W^2, L(U,W))$ and $(\xi,{\xi}') \in \mathscr{D}_{X}^\beta([-r , 0],W)$ for some $\beta \in (1/3,\gamma)$. Then the equation
 \begin{align}\label{eqn:rough_delay}
  \begin{split}
  y_t &= \xi_0 + \int_0^t \sigma (y_s, y_{s-r})\, d \mathbf{X}_s; \quad t \in [0,r] \\
  y_t &= \xi_{t}; \quad t \in [-r, 0] 
  \end{split}
\end{align}
has a unique solution $(y,y') \in \mathscr{D}_{X}^\beta([0,T],W)$ with Gubinelli derivative given by $y'_t = \sigma(y_t,y_{t-r})$.
\end{theorem}

\begin{proof}
  The theorem was proved in \cite[Theorem 4.2]{NNT08}, we quickly sketch the idea here: First, it can be shown that for an element $\zeta \in \mathscr{D}_{X}^\beta([0 , r],W)$, the path $\sigma(\zeta_{\cdot}, \xi_{\cdot - r})$ is a delayed controlled path. Therefore, one can consider the map
\begin{align*}
  \zeta \mapsto \xi_0 + \int_0^{\cdot} \sigma(\zeta_u,\xi_{u-r})\, d\mathbf{X}_u
\end{align*}
and prove that it has a fixed point in the space $\mathscr{D}_{X}^\beta([0 , r],W)$ to obtain a solution on $[0, r]$. The claimed Gubinelli derivative can be deduced using the estimate provided in Theorem \ref{thm:delayed_rough_integral}.
\end{proof}

We proceed with a theorem which shows that the solution map induced by \eqref{eqn:rough_delay} is continuous. Unfortunately, the corresponding result stated in \cite[Theorem 4.2]{NNT08} is not correct, therefore we can not cite it directly. We will first formulate the correct statement and then discuss the difference compared to \cite[Theorem 4.2]{NNT08}.

\begin{theorem}\label{thm:delay_stability}
 Let $\mathbf{X}$ and $\tilde{\mathbf{X}}$ be a delayed $\gamma$-rough paths with $\gamma \in (1/3,1/2]$, $\sigma\in C^{3}_{b}(W^{2},L(U,W))$ and choose $(\xi,{\xi}') \in \mathscr{D}_{X}^\beta([-r , 0],W)$ and $(\tilde{\xi},\tilde{\xi}') \in \mathscr{D}_{\tilde{X}}^\beta([-r , 0],W)$ for some $\beta \in (1/3,\gamma)$. Consider the solutions $(y,y')$ and $(\tilde{y},\tilde{y}')$ to
 \begin{align*}
 dy_t &=\sigma (y_t, y_{t-r})\, d \mathbf{X}; \quad t \in [0,r] \\
  y_t &= \xi_{t}; \quad t \in [-r, 0] 
\end{align*}
resp.
\begin{align*}
 d\tilde{y}_t &=\sigma (\tilde{y}_t, \tilde{y}_{t-r})\, d \tilde{\mathbf{X}}; \quad t \in [0,r]\\
 \tilde{y}_t &= \tilde{\xi}_{t}; \quad t \in [-r, 0].
\end{align*}
  Then 
\begin{align}
  \begin{split}\label{eqn:loc_lipschitz_delay}
 &d_{2\beta;[0,r]}((y,y'),(\tilde{y},\tilde{y}')) \\
 \leq\ &C\left( |\xi_{-r} - \tilde{\xi}_{-r}| + |\xi'_{-r} - \tilde{\xi}'_{-r}| + d_{2\beta;[-r,0]}((\xi,\xi'),(\tilde{\xi},\tilde{\xi}')) + \varrho_{\gamma;[0,r]}(\mathbf{X},\tilde{\mathbf{X}}) \right)
 \end{split}
\end{align}
holds for some constant $C>0$ depending on $r$, $\gamma$, $\beta$ and $M$, where $M$ is chosen such that
\begin{align*}
 M \geq &\|\xi\|_{\mathscr{D}^{\beta}_{X}} + \|\tilde{\xi}\|_{\mathscr{D}^{\beta}_{\tilde{X}}} + \|X\|_{\gamma} + \|\mathbb{X}\|_{2\gamma} + \|\mathbb{X}(-r)\|_{2\gamma} \\
 &\quad + \|\tilde{X}\|_{\gamma} + \|\tilde{\mathbb{X}}\|_{2\gamma} + \|\tilde{\mathbb{X}}(-r)\|_{2\gamma}.
\end{align*}

\end{theorem}

\begin{remark}
 In \cite[Theorem 4.2]{NNT08}, the authors state that the estimate
 \begin{align}\label{eqn:wrong_estimate}
  \|y - \tilde{y}\|_{\beta;[0,r]} \leq C(|\xi_{-r} - \tilde{\xi}_{-r}| + \|\xi - \tilde{\xi}\|_{\beta;[-r,0]} + \rho_{\gamma}(\mathbf{X},\tilde{\mathbf{X}}))
 \end{align}
 holds for the usual H\"older norm. However, this estimate can not be true in general. To see this, assume $\mathbf{X}^1 = \mathbf{X}^2 =: \mathbf{X}$ and consider the equation in Example \ref{ex:rough_delay_linear}. If \eqref{eqn:wrong_estimate} was true, the map 
 \begin{align*}
  \xi \mapsto \int\xi \, d\mathbf{X}
 \end{align*}
 would be continuous in the $\beta$-H\"older norm, which is clearly not the case for a genuine rough path $\mathbf{X}$.

\end{remark}

The proof of Theorem \ref{thm:delay_stability} is a bit lengthy, but mostly straightforward. We sketch it in the appendix, cf. page \pageref{proof:stability}.

\subsection{Linear equations}

In this section, we consider the case where $\sigma$ is linear, i.e. $ \sigma\in L\big{(}W^{2},L(U,W)\big{)} $. Note that in this case, there are  $ \sigma_{1},\sigma_{2}\in L\big{(}W,L(U,W)\big{)} $ such that $ \sigma(y_{1},y_{2})=\sigma_{1}(y_{1})+\sigma_{2}(y_{2}) $ for all $y_1, y_2 \in W$. Since linear vector fields are unbounded, we cannot directly apply Theorem \ref{thm:delay_existence}. However, we can prove an a priori bound for any solution of the equation and then deduce existence, uniqueness and stability for linear equations from Theorem \ref{thm:delay_existence} and \ref{thm:delay_stability} by truncating the vector field $\sigma$.

\begin{theorem}\label{thm:delay_linear}
 Let $\mathbf{X}$ be a delayed $\gamma$-rough path over $X$ with $\gamma \in (1/3,1/2]$ and $\sigma\in L(W^{2},L(U,W))$. Then any solution $y \colon [0,r] \to W$ of 
 \begin{align}\label{AXCV}
 \begin{split}
 dy_t &=\sigma (y_t, y_{t-r})\, d \mathbf{X}; \quad t \in [0, r] \\
  y_t &= \xi_{t}; \quad t \in [-r, 0]
\end{split}
\end{align}
satisfies, for $(y,y') = (y, \sigma(y,\xi_{\cdot - r}))$, the bound
 \begin{align}\label{eqn:a_priori_bound_linear}
  \begin{split}
   &\| y \|_{\mathscr{D}_{X}^{\beta}([0,r],W)}\leqslant\\  C&\big{(}1+r^{\gamma -\beta}\Vert X\Vert_{\gamma;[0,r]}\big{)} \| \xi \|_{\mathscr{D}_{X}^{\beta}([-r,0],W)} \exp\left\{ C(\Vert X\Vert_{\gamma ;[0,r]}+\Vert\mathbb{X}\Vert_{2\gamma ;[0,r]}+\Vert\mathbb{X}(-r)\Vert_{2\gamma ;[0,r]})^{\frac{1}{\gamma - \beta}} \right\}
   \end{split}
 \end{align}
  where $C$ depends on $r$, $\| \sigma\|$, $\gamma$ and $\beta$.
\end{theorem}

\begin{proof}
 For $ s,t\in [0,r] $, we have
\begin{align*}
y_{s,t}= y'_s X_{s,t}+y^{\#}_{s,t}
\end{align*}
where
\begin{align}\label{ddc}
y^{\prime}_s &=\sigma(y_{s},\xi_{s-r})
\end{align}
and 
\begin{align*}
 y^{\#}_{s,t} &= \int_s^t \sigma (y_u, y_{u-r})\, d \mathbf{X}_{u} - \sigma(y_{s},\xi_{s-r})X_{s,t} \\
&= \tilde{\rho}_{s,t} + \sigma_{1}y^{\prime}_{s}\mathbb{X}_{s,t} + \sigma_{2}\xi^{\prime}_{s-r}\mathbb{X}_{s,t}(-r)
\end{align*}
with $\tilde{\rho}$ given by
\begin{align*}
 \tilde{\rho}_{s,t} = \int_s^t \sigma (y_u, y_{u-r})\, d \mathbf{X}_{u} - \sigma(y_{s},\xi_{s-r})X_{s,t} - \sigma_{1}y^{\prime}_{s}\mathbb{X}_{s,t} - \sigma_{2}\xi^{\prime}_{s-r}\mathbb{X}_{s,t}(-r).
\end{align*}
Note that $u \mapsto \sigma (y_u, \xi_{u-r})$ is a delayed controlled path with Gubinelli derivative $u \mapsto (\sigma_1 y'_u, \sigma_2 \xi'_{u-r})$. Therefore, we can use the estimate provided in Theorem \ref{thm:delayed_rough_integral} to see that for a constant $ M = M(\beta,r) $ and $ I=[a,b]\subset [0,r] $ :
\begin{align}\label{DCCC}
\begin{split}
\Vert y^{\#}\Vert_{2\beta ;I}&\leqslant\Vert \sigma\Vert\big{(}\Vert y^{\prime}\Vert_{\infty ;I}\Vert\mathbb{X}\Vert_{2\gamma ;I}+\Vert\xi^{\prime}\Vert_{\infty;[-r,0]}\Vert\mathbb{X}(-r)\Vert_{2\gamma ;I}\big{)}(b-a)^{2\gamma -2\beta}\\&+M\Vert \sigma\Vert\big{(}\Vert y^{\#}\Vert_{2\beta ;I}\Vert X\Vert_{\gamma ;I}+\Vert\xi^{\#}\Vert_{2\beta ;[-r,0]}\Vert X\Vert_{\gamma ;I}\big{)}(b-a)^{\gamma}
\\&+M\Vert\sigma\Vert\big{(}\Vert y^{\prime}\Vert_{\beta ;I}\Vert\mathbb{X}\Vert_{2\gamma ;I}+\Vert\xi^{\prime}\Vert_{\beta ;[-r,0]}\Vert\mathbb{X}(-r)\Vert_{2\gamma ;I}\big{)}(b-a)^{2\gamma -\beta} 
\end{split}
\end{align}
and by relation (\ref{ddc}) :
\begin{align*}
&\Vert y\Vert_{\beta ;I}\leqslant \Vert\sigma\Vert\big{(}\Vert y\Vert_{\infty ,I}+\Vert \xi\Vert_{\infty ,[-r,0]}\big{)}\Vert X\Vert_{\gamma ;I}(b-a)^{\gamma -\beta}+\Vert y^{\#}\Vert _{2\beta ;I}(b-a)^{\beta} \quad \text{and}\\
&\Vert y^{\prime}\Vert_{\beta ;I}\leqslant \Vert\sigma\Vert\big{(}\Vert y\Vert_{\beta ;I}+\Vert\xi\Vert_{\beta ;[-r,0]}\big{)}.
\end{align*}
Now assume that $ b-a=\theta<1\wedge r$ for a given $\theta$ and set 
\begin{align*}
 A := 1+\Vert X\Vert_{\gamma ;[0,r]}+\Vert\mathbb{X}\Vert_{2\gamma ;[0,r]}+\Vert\mathbb{X}(-r)\Vert_{2\gamma ;[0,r]}.
\end{align*}
Our former estimates imply that there are constants $ \tilde{M},\tilde{N} $ depending on $ \Vert\sigma\Vert $ such that
\begin{align}\label{A12}
&\Vert y\Vert_{\beta ;I}+\Vert y\Vert_{\infty ;I}+\Vert y^{\prime}\Vert_{\beta ;I}+\Vert y^{\#}\Vert_{2\beta ;I}+\Vert y^{\prime}\Vert_{\infty ;I}\leqslant\nonumber\\& \tilde{M}A\theta^{\gamma -\beta}\big{(}\Vert y\Vert_{\beta ;I}+\Vert y\Vert_{\infty ;I}+\Vert y^{\prime}\Vert_{\beta ;I}+\Vert y^{\#}\Vert_{2\beta ;I}+\Vert y^{\prime}\Vert_{\infty ;I}\big{)}+\\&\tilde{N}A\big{(}\Vert \xi\Vert_{\beta ;[-r,0]}+\Vert \xi\Vert_{\infty ;[-r,0]}+\Vert \xi^{\prime}\Vert_{\infty ;[-r,0]}+\Vert \xi^{\#}\Vert_{2\beta ;[-r,0]}\big{)}+\big{(}1+\Vert\sigma\Vert\big{)}\Vert y\Vert_{\infty ;I}\nonumber.
\end{align}
Choose $\theta$ small enough such that
\begin{align}\label{A13}
\tilde{M}A\theta^{\gamma -\beta}\leqslant\frac{1}{4} \quad \text{and} \quad\theta^{\beta}(1+\Vert\sigma\Vert)\leqslant\frac{1}{4}.
\end{align}
For $ n \geqslant 1 $ and  $ n\theta\leqslant r $, set $I_{n} := [(n-1)\theta ,n\theta] $ and
\begin{align*}
&B_{n}=\Vert y\Vert_{\beta ;I_{n}}+\Vert y\Vert_{\infty ;I_{n}}+\Vert y^{\prime}\Vert_{\beta ;I_{n}}+\Vert y^{\#}\Vert_{2\beta ;I_{n}}+\Vert y^{\prime}\Vert_{\infty ;I_{n}} \\
&B_{0}=\Vert \xi\Vert_{\beta ;[-r,0]}+\Vert \xi\Vert_{\infty ;[-r,0]}+\Vert \xi^{\prime}\Vert_{\infty ;[-r,0]}+\Vert \xi^{\#}\Vert_{2\beta ;[-r,0]}.
\end{align*}
Note that $ \Vert y\Vert_{\infty ;I_{n}}\leqslant B_{n-1}+\theta^{\beta}B_{n} $. By (\ref{A12}) and (\ref{A13}),
\begin{align*}
B_{n}\leqslant 2\tilde{N}AB_{0}+2\big{(}1+\Vert\sigma\Vert\big{)}B_{n-1}.
\end{align*}
Set $C=2\tilde{N}A $ and $\tilde{C}=2\big{(}1+\Vert\sigma\Vert\big{)}$. By a simple induction argument, it is not hard to verify that for $ k \leq n $,
\begin{align*}
B_{n}\leqslant C(1+\tilde{C}+\tilde{C}^{2}+...+\tilde{C}^{k-1})B_{0}+\tilde{C}^{k}B_{n-k} 
\end{align*}
which implies
\begin{align*}
 B_{n}\leqslant \tilde{C}^{n}(1+C)B_{0}
\end{align*}
for $k = n$. Note that since $ y^{\#}_{s,t}= y^{\#}_{s,u}+ y^{\#}_{u,t} +y^{\prime}_{s,u}X_{u,t}$,
\begin{align}\label{sharp}
\Vert y^{\#}\Vert_{2\beta;[0,r]}\leqslant\sum_{1\leqslant n\leqslant m}\Vert y^{\#}\Vert_{2\beta;I_{n}}+r^{\gamma -\beta}\Vert X\Vert_{\gamma;[0,r]}\sum_{1\leqslant n\leqslant m}\Vert y^{\prime}\Vert_{\beta;I_{n}}
\end{align}
Now set $ m=[\frac{r}{\theta}]+1 $. By (\ref{sharp}) and subadditivity of the H\"older norm,
\begin{align*}
\Vert y\Vert_{\beta ;[0,r]}&+\Vert y\Vert_{\infty ;[0,r]}+\Vert y^{\prime}\Vert_{\beta ;[0,r]}+\Vert y^{\#}\Vert_{2\beta ;[0,r]}+\Vert y^{\prime}\Vert_{\infty ;[0,r]}\\
&\leqslant \big{(}1+r^{\gamma -\beta}\Vert X\Vert_{\gamma;[0,r]}\big{)}\sum_{1\leqslant n\leqslant m}B_{n}\leqslant \big{(}1+r^{\gamma -\beta}\Vert X\Vert_{\gamma;[0,r]}\big{)}\tilde{C}^{m+1}(1+C)B_{0}.
\end{align*}
Note that an appropriate choice for $ \theta $ is 
\begin{align}\label{bound1}
\theta =\frac{1}{\big{(}4\tilde{M}A\big{)}^{\frac{1}{\gamma -\beta}}+\big{(}4(1+\Vert\sigma\Vert)\big{)}^{\frac{1}{\beta}}+1+\frac{1}{r}}
\end{align}
which implies the claimed bound.
\end{proof}

From Theorem \ref{thm:delay_linear}, it follows that in the case of linear vector fields $\sigma$, the solution map induced by \eqref{AXCV} is a bounded linear map. We now prove that it is even compact.

\begin{proposition}\label{prop:compact_map}
 Under the same assumptions as in Theorem \ref{thm:delay_linear}, the solution map induced by \eqref{AXCV} is a compact linear map for every $1/3 < \beta < \gamma$.
\end{proposition}

 \begin{proof}
 Fix $\beta < \gamma$. Let $ \lbrace\xi^{(n)}\rbrace_{n\geqslant 1} $ be a bounded sequence in $ \mathscr{D}_{X}^{\beta}([-r,0],W)$, i.e.
\begin{align*}
\xi^{(n)}_{u,v}=(\xi^{(n)})^{\prime}_{u}X_{u,v} + (\xi^{(n)})^{\#}_{u,v}
\end{align*}
with uniformly bounded $\beta$-H\"older norm of $\xi^{(n)}$ and $(\xi^{(n)})^{\prime}$ and uniformly bounded $2\beta$-H\"older norm of $(\xi^{(n)})^{\#}$. From the Arzel\`a-Ascoli theorem, there are continuous functions $\xi$ and $\xi'$ such that
\begin{align*}
 (\xi^{(n)}, (\xi^{(n)})^{\prime}) \to (\xi,\xi')
\end{align*}
uniformly along a subsequence, which we will henceforth denote by $(\xi^{(n)},(\xi^{(n)})')_n$ itself. It follows that $(\xi^{(n)}, (\xi^{(n)})^{\prime}) \to (\xi,\xi')$ in $\delta$-H\"older norm for every $\delta < \beta$. Define $ \xi^{\#}_{u,v}:=\xi_{u,v} - \xi_{u}^{\prime}X_{u,v}$. Clearly, $(\xi^{(n)})^{\#} \to \xi^{\#}$ uniformly, and since
\begin{align*}
 |\xi^{\#}_{u,v}| \leq \sup_n \|  (\xi^{(n)})^{\#} \|_{2\beta;[-r,0]} |v - u|^{2\beta}
\end{align*}
for every $-r \leq u \leq v \leq 0$, it follows that $(\xi^{(n)})^{\#} \to \xi^{\#}$ in $2\delta$-H\"older norm for every $\delta < \beta$. This implies that $(\xi^{(n)}, (\xi^{(n)})') \to (\xi,\xi')$ in the space $ \mathscr{D}_{X}^{\delta}([-r,0],W)$ for every $\delta < \beta$. Let $(y^n,(y^n)')$ denotes the solutions to \eqref{AXCV} for the initial conditions $(\xi^n, (\xi^n)')$. Fix some $1/3 < \delta < \beta$. From continuity, the solutions $(y^n,(y^n)')$ converge in the space $ \mathscr{D}_{X}^{\delta}([0,r],W)$, too. Choose $\beta < \beta' < \gamma$. Using a similar estimate as \eqref{DCCC} in Theorem \ref{thm:delay_linear} where we apply the estimate in Theorem \ref{thm:delayed_rough_integral} for $\delta$ shows that we can bound  $\| (y^n,(y^n)') \|_{\mathscr{D}_{X}^{\beta'}([0,r],W)}$ uniformly over $n$ where the bound depends, in particular, on $\sup_n \| (y^n,(y^n)') \|_{\mathscr{D}_{X}^{\delta}([0,r],W)}$. This implies convergence also in the space $\mathscr{D}_{X}^{\beta}([0,r],W)$ and therefore proves compactness.
\end{proof}

\subsection{A semi-flow property}

  In this section,
we discuss the flow property induced by a rough delay equation. Recall that a \emph{flow} on some set $M$ is a mapping 
\begin{align*}
  \phi \colon [0,\infty) \times [0,\infty) \times M \to M
\end{align*}
 such that $\phi(t,t,\xi) = \xi$ and
\begin{align}\label{eqn:flow_prop}
 \phi(s,t,\xi) = \phi(u,t, \phi(s,u,\xi))
\end{align}
hold for every $\xi \in M$ and $s,t,u \in [0,\infty)$. Our prime example of a flow is a differential equation in which case $\xi \in M$ denotes an initial condition at time point $s$ and $\phi(s,t,\xi)$ denotes the solution at time $t$. In the setting of a delay equation, we can only expect to solve the equation forward in time, i.e. $\phi(s,t,\xi)$ will only be defined for $s \leq t$. If \eqref{eqn:flow_prop} is assumed to hold only for $s \leq u \leq t$, we will speak of a \emph{semi-flow}. In case of a rough delay equation, we will give up the idea of choosing a common set of admissible initial conditions $M$ which will work for all time instances. Instead, our semi-flow will actually consist of a family of maps 
\begin{align*}
 \phi(s,t,\cdot) \colon M_s \to M_t
\end{align*}
where $(M_t)_{t \geq 0}$ are sets (later: spaces) indexed by time. Note that the semi-flow property \eqref{eqn:flow_prop} still makes perfect sense in this setting, and this is what we are going to prove. Note also that the phenomenon of time-varying spaces is already visible in Example \ref{ex:rough_delay_linear}: admissible initial conditions are controlled paths defined on intervals depending on the time when we start to solve the equation.

\begin{theorem}\label{thm:flow_prop}
Let $\mathbf{X}$ be a delayed $\gamma$-rough path over $X$ with $\gamma \in (1/3,1/2]$ and $\sigma\in C^{3}_{b}(W^{2},L(U,W))$. Consider the equation
\begin{align}\label{equ}
\begin{split}
 dy_t &=\sigma (y_t, y_{t-r})\, d \mathbf{X}; \quad t \in [s, \infty) \\
  y_t &= \xi_{t}; \quad t \in [s-r, s]
\end{split}
\end{align}
for $s \in \R$. Let $\beta \in (1/3, \gamma)$. If $\xi \in \mathscr{D}_{X}^\beta([s-r , s],W)$, the equation \eqref{equ} has a unique solution  $y \colon [s,\infty) \to W$ and for $t \geq s$, we denote by $\phi(s,t,\xi)$ the solution path segment
\begin{align*}
\phi(s,t,\xi) = (y_u)_{t - r \leq u \leq t}.
\end{align*}
If $ r \leqslant t - s$, we have that $\phi(s,t,\xi) \in \mathscr{D}_{X}^\beta([t-r, t],W)$ with Gubinelli derivative $\phi'(s,t,\xi) = (\sigma(y_u,y_{u-r}))_{t-r \leq u \leq t}$ and
\begin{align}\label{continuity}
\phi(s,t,\cdot) \colon \mathscr{D}_{X}^\beta([s-r, s],W) &\to \mathscr{D}_{X}^\beta([t-r, t],W) \\
 \xi &\mapsto \phi(s,t, \xi)\nonumber
\end{align}
is a continuous map. In case that $ \xi _{s}^{\prime}= \sigma\big{(}\xi_{s},\xi_{s-r}\big{)} $, we have $\phi(s,t,\xi)\in \mathscr{D}_{X}^\beta([-r + t, t],W)$ for all $s \leq t$ with Gubinelli derivative given by
\begin{align*}
 \phi(s,t,\xi)(u) = \begin{cases}
                     \xi'_u &\text{ for } t-r \leq u \leq s \\
                     \sigma(y_u,y_{u-r}) &\text{ for } s \leq u \leq t
                    \end{cases}
\end{align*}
for $r > t - s$. For $s \leq u \leq t$ and $r \leqslant u - s $, we have the semi-flow property
\begin{align}\label{Flow}
\phi(s,s,\cdot) &= \operatorname{Id}_{\mathscr{D}_{X}^p([-r + s, s],W)} \qquad \text{and} \nonumber\\
\phi(u,t, \cdot)& \circ \phi(s,u, \xi) = \phi(s,t, \xi).
\end{align}
Again, if $ \xi _{s}^{\prime}= \sigma\big{(}\xi_{s},\xi_{-r+s}\big{)} $, \eqref{Flow} is true for all $s \leq u \leq t$.
\end{theorem}

\begin{proof}
As in Theorem \ref{thm:delay_existence}, we can first solve \eqref{equ} on the time interval $[s,s+r]$. This can now be iterated to obtain a solution on $[s,\infty)$. The claimed Gubinelli derivatives on every interval $[s + kr, s + (k+1)r]$, $k \in \N_0$, are a consequence of Theorem \ref{thm:delayed_rough_integral}. Since the derivatives agree on the boundary points of the intervals, we can ``glue them together'' to obtain a controlled path on arbitrary intervals $[u,v] \subset [s,\infty)$. If the assumption $ \xi _{s}^{\prime}= \sigma\big{(}\xi_{s},\xi_{s-r}\big{)} $ holds, this can even be done for every interval $[u,v] \subset [s-r,\infty)$. Continuity of the map (\ref{continuity}) is a consequence of Theorem \ref{thm:delay_stability}.
The semi-flow property (\ref{Flow}) is a consequence of existence and uniqueness of solutions:  Let $ y_{\tau}^{s,\xi}$ be the solution of (\ref{equ}) for $\tau \geq s - r$ where $ y_{\tau}^{s,\xi}=\xi_{\tau} $ for $ s-r\leqslant\tau\leqslant s $. Let $s \leq u \leq t$ and assume either $r \leqslant u - s $ or $ \xi _{s}^{\prime}= \sigma\big{(}\xi_{s},\xi_{s-r}\big{)} $. For $ \tau<u $, it is not hard to verify that $ y_{\tau}^{s,\xi}=y_{\tau}^{u,\phi(s,u,\xi)} $. If $ u\leqslant\tau $ by definition:
\begin{align*}
y_{\tau}^{s,\xi}&=\xi_s + \int _{s}^{\tau}\sigma (y_{z}^{s,\xi},y_{z-r}^{s,\xi})d \mathbf{X}_{z}= y_{u}^{s,\xi}+\int _{u}^{\tau}\sigma (y_{z}^{s,\xi},y_{z-r}^{s,\xi})d \mathbf{X}_{z} \quad \text{and} \\
 y_{\tau}^{u,\phi(s,u,\xi)}&=y_{u}^{s,\xi}+\int _{u}^{\tau}\sigma (y_{z}^{u,\phi(s,u,\xi)},y_{z-r}^{u,\phi(s,u,\xi)})d \mathbf{X}_{z}.
\end{align*}
Given the uniqueness of the solution, $ y_{\tau}^{s,\xi}=y_{\tau}^{u,\phi(s,u,\xi)} $ which indeed implies (\ref{Flow}).
\end{proof}

\section{Existence of delayed L\'evy areas for the Brownian motion and a Wong-Zakai theorem}\label{sec:delayed_area_BM}

In order to apply the results from Section \ref{sec:basics_rough_delay} to stochastic delay differential equations, we need to make sure that the Brownian motion can be ''lifted'' to a process taking values in the space of delayed rough paths. In this section, $B = (B^1, \ldots, B^d) \colon \R \to \R^d$ will always denote an $\R^d$-valued two-sided Brownian motion defined on a probability space $(\Omega,\mathcal{F},\P)$ adapted to some two-parameter filtration $(\mathcal{F}^t_s)_{s \leq t}$, i.e. $(B_{t+s} - B_s)_{t \geq 0}$ is a usual $(\mathcal{F}_s^{t+s})_{t \geq 0}$-Brownian motion for every $s \in \R$ and $B_0 = 0$ almost surely (cf. \cite[Section 2.3.2]{Arn98} for a more detailed discussion about two-sided stochastic processes).

\begin{definition}
 For $r > 0$, set 
 \begin{align*}
  \mathbf{B}_{s,t}^{\text{It\=o}} := \left( B_{s,t}, \mathbb{B}_{s,t}^{\text{It\=o}}, \mathbb{B}_{s,t}^{\text{It\=o}}(-r) \right) := \left(B_t - B_s, \int_s^t (B_u - B_s)\, \otimes dB_u, \int_s^t (B_{u - r} - B_{s - r})\, \otimes d B_u \right)
 \end{align*}
 for $s \leq t \in \R$ where the stochastic integrals are understood in It\=o-sense. We furthermore define 
 \begin{align*}
  \mathbf{B}_{s,t}^{\text{Strat}} :=  \left( B_{s,t}, \mathbb{B}_{s,t}^{\text{It\=o}} + \frac{1}{2}(t-s) I_d, \mathbb{B}_{s,t}^{\text{It\=o}}(-r) \right)
 \end{align*}
 where $I_d$ denotes the identity matrix in $\R^d$.

\end{definition}

\begin{proposition}\label{prop:existence_bm_lift}
 Both processes $\mathbf{B}^{\text{It\=o}}$ and $\mathbf{B}^{\text{Strat}}$ have modifications, henceforth denoted with the same symbols, with sample paths being delayed $\gamma$-rough paths for every $\gamma < 1/2$ almost surely. Moreover, the $\gamma$-H\"older norms of both processes have finite $p$-th moment for every $p > 0$ on any compact interval.
\end{proposition}

\begin{proof}
 The assertion follows by considering the usual It\=o- and Stratonovich lifts of the enhanced process $(B,B_{\cdot - r})$ as in \cite[Section 3.2 and 3.]{FH14}, using the Kolmogorov criterion for rough paths stated in \cite[Theorem 3.1]{FH14} (cf. also Remark \ref{eqn:delayed_rp_as_usual_rp}).
\end{proof}

The next proposition justifies the names of the processes defined above.

\begin{proposition}
 Let $(m(\omega),\zeta^0(\omega),\zeta^1(\omega)) \in \mathcal{D}_{B(\omega)}^{\gamma}$ almost surely. Furthermore, assume that the process $(m_t,\zeta^0_t,\zeta^1_t)_{t \geq 0}$ is $(\mathcal{F}^t_0)_{t \geq 0}$-adapted. Then
 \begin{align*}
  \int_0^T m_s\, dB_s = \int_0^T m_s\, d\mathbf{B}^{\text{It\=o}}_s \quad \text{and} \quad \int_0^T m_s\,\circ dB_s = \int_0^T m_s\, d\mathbf{B}^{\text{Strat}}
 \end{align*}
 almost surely for every $T > 0$.

\end{proposition}

\begin{proof}
We will first consider the It\=o-case which is similar to \cite[Proposition 5.1]{FH14}. Set $\mathcal{F}_t := \mathcal{F}_0^t$. To simplify notation, assume $W = \R$. Let $(\tau_j)$ be a partition of $[0,T]$. We first prove that 
\begin{align}\label{eqn:off_diag_zero}
\E\bigg{[}\big{(}\zeta^0_{\tau_{j}}\mathbb{B}_{\tau_{j},\tau_{j+1}}+\zeta^1_{\tau_{j}}\mathbb{B}_{\tau_{j},\tau_{j+1}}(-r)\big{)}\big{(}\zeta^0_{\tau_{k}}\mathbb{B}_{\tau_{k},\tau_{k+1}}+\zeta^1_{\tau_{k}}\mathbb{B}_{\tau_{k},\tau_{k+1}}(-r)\big{)}\bigg{]} =0
\end{align}
  for $j < k$. To see this, note that
\begin{align*}
&\E\big{[}\big{(}\zeta^1_{\tau_{j}}\mathbb{B}_{\tau_{j},\tau_{j+1}}(-r)\big{)}\big{(}\zeta^1_{\tau_{k}}\mathbb{B}_{\tau_{k},\tau_{k+1}}(-r)\big{)}\big{]} = \E\bigg{[} \E\big{[}\big{(}\zeta^1_{\tau_{j}}\mathbb{B}_{\tau_{j},\tau_{j+1}}(-r)\big{)}\big{(}\zeta^1_{\tau_{k}}\mathbb{B}_{\tau_{k},\tau_{k+1}}(-r)\big{)}\big{|}\mathcal{F}_{\tau_{k}}\big{]}\bigg{]}=\\
&\E\bigg{[}\zeta^1_{\tau_{j}} \mathbb{B}_{\tau_{j},\tau_{j+1}}(-r)  \zeta^1_{\tau_{k}} \E \big{[}\mathbb{B}_{\tau_{k},\tau_{k+1}}(-r)\big{|}\mathcal{F}_{\tau_{k}}\big{]}\bigg{]}.
\end{align*}
We show that $\E\big{[}\mathbb{B}_{s,u}(-r)\big{|}\mathcal{F}_{s}\big{]} =0 $ for $s \leq u$. By definition,
\begin{align*}
\mathbb{B}_{s,u}(-r) = \lim_{|\Pi| \to 0}\sum_{t_{k}\in\Pi}B_{s-r,t_{k}-r}\otimes B_{t_{k},t_{k+1}}
\end{align*}
where $ \Pi $ is a partition for $[s,u] $ and the limit is understood in $L^2(\Omega)$-sense. Consequently,
\begin{align*}
  \E\big{[}\mathbb{B}_{s,u}(-r)\big{|}\mathcal{F}_{s}\big{]} = \lim_{|\Pi| \to 0}\sum_{t_{k}\in\Pi} \E\big{[}B_{s-r,t_{k}-r}\otimes B_{t_{k},t_{k+1}}\big{|}\mathcal{F}_{s}\big{]}
\end{align*}
again in $L^2$.
Note that
\begin{align*}
  \E\big{[}B_{s-r,t_{k}-r}&\otimes B_{t_{k},t_{k+1}}\big{|}\mathcal{F}_{s}\big{]}=\\
&\begin{cases}
B_{s-r,t_{k}-r}\otimes \E\big{[} B_{t_{k},t_{k+1}}\big{|}\mathcal{F}_{s}\big{]} = 0 , \ \ \text{if} \ t_{k}-r\leqslant s\\
B_{s-r,s}\otimes \E\big{[} B_{t_{k},t_{k+1}}\big{|}\mathcal{F}_{s}\big{]} + \E\big{[}B_{s,t_{k}-r}\otimes B_{t_{k},t_{k+1}}\big{|}\mathcal{F}_{s}\big{]} =0  , \ \ \text{if}\  s<  t_{k}-r.
\end{cases}
\end{align*}
Other cases are similar and \eqref{eqn:off_diag_zero} can be deduced. Using a stopping argument, we may assume that there is a deterministic $M > 0$ such that
\begin{align*}
 \sup_{t \in [0,T]} \|\zeta^0_t(\omega) \| \vee \|\zeta^1_t(\omega) \| \leq M
\end{align*}
almost surely. Then,
\begin{align*}
\E\bigg{[}\big{(}\sum_{j} \zeta^{0}_{t_{j}} \mathbb{B}_{\tau_{j},\tau_{j+1}} + \zeta^{1}_{\tau_{j}} \mathbb{B}_{\tau_{j},\tau_{j+1}}(-r)\big{)}^{2}\bigg{]}\leqslant M \sum_{j}(\tau_{j+1}-\tau_{j})^{2} \leq MT \max_j|\tau_{j+1} - \tau_j|
\end{align*}
which converges to $0$ when the mesh size of the partition gets small. The claim now follows using the definition of the It\=o integral as a limit of Riemann sums. 
  The proof for the Stratonovich integral is similar to  \cite[Corollary 5.2]{FH14}.
\end{proof}

The following corollary is immediate.

\begin{corollary}\label{cor:rough_ito_coincides}
 The solution to the It\=o equation
 \begin{align*}
  dY_t = \sigma(Y_t,Y_{t-r})\, dB_t
 \end{align*}
 is almost surely equal to the solution to the random rough delay equation
 \begin{align*}
  dY_t = \sigma(Y_t,Y_{t-r})\, d\mathbf{B}^{\text{It\=o}}_t
 \end{align*}
 if the initial condition is $\mathcal{F}_{-1}^0$-measurable and almost surely controlled by $B$. The same statement holds in the Stratonovich case.

\end{corollary}

Next, we prove an approximation result. 
\begin{definition}
Let $ \rho :\mathbb{R}\rightarrow [0,2] $ be a smooth function such that $\operatorname{supp}(\rho)\subset [0,1] $ and which integrates to $1$. We set
\begin{align*}
B^{\varepsilon}_{t}:=\int_{\mathbb{R}}B_{-\varepsilon z,t-\varepsilon z}\rho(z)d z, \quad \varepsilon \in (0,1].
\end{align*}
\end{definition}

It is not hard to see that
\begin{align}\label{AZAZ}
\E|B^{\varepsilon}_{s,t}|^{2}\leqslant M(t-s) \ \ \text{and} \ \ \ \lim_{\varepsilon \rightarrow 0}\E|B^{\varepsilon}_{s,t}-B_{s,t}|^{2} =0
\end{align}
where $ M $ is independent of $\varepsilon$.

\begin{lemma}\label{representation}
We have the following pathwise identity:
\begin{align}\label{eqn:pathwise}
\int _{s}^{t}B^{\varepsilon}_{s-r,u-r}\otimes d B_{u}^{\varepsilon}=\int _{\mathbb{R}}\rho(z)\int _{s-\varepsilon z}^{t-\varepsilon z}B^{\varepsilon}_{s-r,u+\varepsilon z-r}\otimes dB_{u} dz.
\end{align}
\end{lemma}

\begin{proof}
Note that both integrals in \eqref{eqn:pathwise} are indeed pathwise defined since $B^{\varepsilon}$ is smooth and $B$ is H\"older continuous. Using integration by parts, for $i,j \in \{1,\ldots, d\}$,
\begin{align*}
\int _{s-\varepsilon z}^{t-\varepsilon z}(B^{\varepsilon})^{i}_{s-r,u+\varepsilon z-r}dB^{j}_{u}=(B^{\varepsilon})^{i}_{s-r,t-r}B^{j}_{t-\varepsilon z}-\int_{s-r}^{t-r}B^{j}_{u+r-\varepsilon z}d(B^{\varepsilon})^{i}_{u}.
\end{align*}
Consequently,
\begin{align*}
\int _{\mathbb{R}}\rho(z)\int _{s-\varepsilon z}^{t-\varepsilon z}(B^{\varepsilon})^{i}_{s-r,u+\varepsilon z-r}dB^{j}_{u} dz&=\int_{\mathbb{R}}(B^{\varepsilon})^{i}_{s-r,t-r}\rho(z)B^{j}_{t-\varepsilon z}dz-\int_{\mathbb{R}}\int_{s-r}^{t-r}\rho(z)B^{j}_{u+r-\varepsilon z}d(B^{\varepsilon})^{i}_{u} dz\\
&=(B^{\varepsilon})^{i}_{s-r,t-r}(B^{\varepsilon})^{j}_{t}-\int_{s-r}^{t-r}(B^{\varepsilon})^{j}_{u+r}d(B^{\varepsilon})^{i}_{u}.
\end{align*}
Using integration by parts again, we have
\begin{align*}
(B^{\varepsilon})^{i}_{s-r,t-r}(B^{\varepsilon})^{j}_{t}-\int_{s-r}^{t-r}(B^{\varepsilon})^{j}_{u+r}d(B^{\varepsilon})^{i}_{u}
 =\int _{s}^{t}(B^{\varepsilon})^{i}_{s-r,u-r}d (B^{\varepsilon})_{u}^{j}
\end{align*}
which implies the claim.
\end{proof}

\begin{lemma}
For $ \mathbb{B}_{s,t}(-r) = \int_{s}^{t}B_{s-r,u-r}\otimes dB_{u}$ and $ \mathbb{B}^{\varepsilon}_{s,t}(-r)=\int_{s}^{t}B^{\varepsilon}_{s-r,u-r}\otimes dB_{u}^{\varepsilon} $,
\begin{align}\label{ZAZZ}
 \E\big{|}\mathbb{B}_{s,t}(-r)\big{|}^{2} \leq M (t-s)^{2} \ \text{ and}  \ \ \ \ \  \E\big{|}\mathbb{B}^{\varepsilon}_{s,t}(-r)\big{|}^{2}  \leq M (t-s)^{2}
\end{align}
for a constant $M > 0$ independent of $s,t$ and $\varepsilon$.
\end{lemma}

\begin{proof}
An easy consequence of the Cauchy-Schwarz inequality and Lemma \ref{representation}.
\end{proof}

\begin{lemma}
We have
\begin{align}\label{eqn:conv_delayed_area}
\lim_{\varepsilon \rightarrow 0} \E\big{|}\mathbb{B}_{s,t}(-r)-\mathbb{B}_{s,t}^{\varepsilon}(-r)\big{|}^{2} = 0.
\end{align}
\end{lemma}

\begin{proof}
A direct consequence of Lemma \ref{representation} and \eqref{AZAZ}. 
\end{proof}

\begin{theorem}\label{thm:approx_strat_delay}
 Setting 
 \begin{align*}
  \mathbf{B}_{s,t}^{\varepsilon} := \left( B^{\varepsilon}_{s,t}, \mathbb{B}_{s,t}^{\varepsilon}, \mathbb{B}_{s,t}^{\varepsilon}(-r) \right) := \left(B^{\varepsilon}_{s,t}, \int_s^t B_{s,u}^{\varepsilon}\, \otimes dB^{\varepsilon}_u, \int_s^t B^{\varepsilon}_{s-r,u - r}\, \otimes d B^{\varepsilon}_u \right),
 \end{align*}
 we have
 \begin{align*}
  \lim_{\varepsilon \rightarrow\infty}\sup_{q\geqslant 1}\frac{\left\|d_{\gamma;I}\big{(} \mathbf{B}^{\varepsilon} , \mathbf{B}^{\text{Strat}} \big{)}\right\|_{L^{q}}}{\sqrt{q}} = 0
 \end{align*}
  for every $\gamma < 1/2$ and every compact interval $I \subset \R$ where $d_{\gamma;I}$ denotes the homogeneous metric
   \begin{align*}
  d_{\gamma;I}(\mathbf{X},\mathbf{Y}) = \sup_{s,t \in I; s \neq t} \frac{|X_{s,t} - Y_{s,t}|}{|t-s|^{\gamma}} + \sqrt{ \sup_{s,t \in I; s \neq t} \frac{|\mathbb{X}_{s,t} - \mathbb{Y}_{s,t}|}{|t-s|^{2 \gamma}}} + \sqrt{ \sup_{s,t \in I; s \neq t} \frac{|\mathbb{X}_{s,t}(-r) - \mathbb{Y}_{s,t}(-r)|}{|t-s|^{2 \gamma}}}.
    \end{align*}
\end{theorem}

\begin{proof}
 The strategy of the proof is standard, cf. \cite[Chapter 15]{FV10}, we only sketch the main arguments. First, the uniform bounds \eqref{ZAZZ} and the convergence \eqref{eqn:conv_delayed_area} hold for $\mathbb{B}^{\varepsilon}$ and $\mathbb{B}^{\text{Strat}}$, too, cf. \cite[Theorem 15.33 and Theorem 15.37]{FV10}. Since all objects are elements in the second Wiener chaos, the results even hold in the $L^q$-norm for any $q \geq 1$. We can now argue as in the proof of \cite[Proposition 15.24]{FV10} to conclude.
\end{proof}

As an application, we can prove a Wong-Zakai theorem for stochastic delay equations. 

\begin{theorem}
 Let $\sigma\in C^{3}_{b}(W^{2},L(U,W))$ and $B^{\varepsilon}$ be defined as above. Assume that there is a set of full measure $\tilde{\Omega} \subset \Omega$ such that
 \begin{align}\label{eqn:ass_controlled}
  (\xi(\omega),\xi'(\omega)) \in \mathscr{D}_{B^{\varepsilon}(\omega)}([-r,0],W) \cap \mathscr{D}_{B(\omega)}([-r,0],W)
 \end{align}
 holds for every $\varepsilon \in (0,1]$ and every $\omega \in \tilde{\Omega}$. Then the solutions to random delay ordinary differential equations
 \begin{align*}
  dY^{\varepsilon}_t &= \sigma(Y^{\varepsilon}_t,Y^{\varepsilon}_{t-r})\, dB^{\varepsilon}_t; \quad t \geq 0 \\
  Y_t^{\varepsilon} &= \xi_t; \qquad t \in [-r,0]
 \end{align*}
 converge in probability as $\varepsilon \to 0$ in $\gamma$-H\"older norm on compact sets for every $\gamma < 1/2$ to the solution $Y$ of
 \begin{align*}
  dY_t &= \sigma(Y_t,Y_{t-r})\, d\mathbf{B}^{\text{Strat}}_t; \quad t \geq 0 \\
  Y_t &= \xi_t; \qquad t \in [-r,0].
 \end{align*}
 Moreover, if $(\xi_t,\xi'_t)$ is $\mathcal{F}_{-1}^0$-measurable for every $t \in [-r,0]$, the solution $Y$ coincides almost surely with the solution of the Stratonovich delay equation
 \begin{align*}
  dY_t &= \sigma(Y_t,Y_{t-r})\,\circ dB_t; \quad t \geq 0 \\
  Y_t &= \xi_t; \qquad t \in [-r,0].
 \end{align*}

\end{theorem}

\begin{proof}
 A combination of the stability result in Theorem \ref{thm:delay_stability}, Theorem \ref{thm:approx_strat_delay} and Corollary \ref{cor:rough_ito_coincides}.
\end{proof}

\begin{remark}
 Note that \eqref{eqn:ass_controlled} is satisfied, for instance, if $\xi$ has almost surely differentiable sample paths, in which case we can choose $\xi' \equiv 0$. 
\end{remark}

%
%

\section{Random Dynamical Systems induced by stochastic delay equations}\label{sec:RDS_and_delay}

This section establishes the connection between stochastic delay equations and Arnold's concept of a random dynamical system. 

\subsection{Delayed rough path cocycles}

 We start by describing the object which will drive our equation. The following definition is an analogue of a \emph{rough paths cocycle} defined in \cite{BRS17} for delay equations.

\begin{definition}
Let $(\Omega, \mathcal{F}, \mathbb{P}, (\theta_t)_{t \in \R})$ be a measurable metric dynamical system and $r > 0$. A \emph{delayed $\gamma$-rough path cocycle $\mathbf{X}$ (with delay $r >0$)} is a delayed $\gamma$-rough path valued stochastic process $\mathbf{X}(\omega) = (X(\omega), \mathbb{X}(\omega), \mathbb{X}(-r)(\omega))$ such that
\begin{align}\label{eqn:cocycle_rough_delay}
 \mathbf{X}_{s,s+t}(\omega) = \mathbf{X}_{0,t}(\theta_s \omega)
\end{align}
holds for every $\omega \in \Omega$ and every $s,t \in \R$. 

\end{definition}

Our goal is to prove that Brownian motion together with L\'evy- and delayed L\'evy area can be understood as delayed rough path cocycles.

\begin{definition}
  For a finite-dimensional vector space $U$, set
\begin{align*}
  \tilde{T}^{2}(U):=\big{\lbrace}\big{(}1\oplus(\alpha,\beta)\oplus(\gamma,\theta)\big{)}\ \vert\ \alpha ,\beta\in U \ \text{and} \ \gamma ,\theta\in U\otimes U\big{\rbrace}.
\end{align*}
We define projections $\Pi_i^j$ by
\begin{align*}
 \Pi_i^j\big{(}1\oplus(\alpha,\beta)\oplus(\gamma,\theta)\big{)} := \begin{cases}
                                                                     \alpha &\text{if } i = 1,\, j = 1 \\
                                                                     \beta &\text{if } i = 1,\, j = 2 \\
                                                                     \gamma &\text{if } i = 2,\, j = 1 \\
                                                                     \theta &\text{if } i = 2,\, j = 2.
                                                                    \end{cases}
\end{align*}
Furthermore, we set
\begin{align*}
\big{(}1\oplus(\alpha_{1},\beta_{1})\oplus(\gamma_{1},\theta_{1})\big{)}&\circledast \big{(}1\oplus(\alpha_{2},\beta_{2})\oplus(\gamma_{2},\theta_{2})\big{)}:=\\ &\big{(}1\oplus(\alpha_{1}+\alpha_{2},\beta_{1}+\beta_{2})\oplus(\gamma_{1}+\gamma_{2}+\alpha_{1}\otimes\alpha_{2},\theta_{1}+\theta_{2}+\beta_{1}\otimes\alpha_{2})\big{)}
\end{align*}
and $\mathbf{1} := (1,(0,0),(0,0))$.

\end{definition}
It is not hard to verify that $(\tilde{T}^{2}(U), \circledast)$ is a topological group with identity $\mathbf{1}$. For a continuous $U$-valued path of bounded variation $x$, we can define the following natural lifting map
\begin{align*}
\tilde{S}_{2}(x)_{u,v}:=\bigg{(}1\oplus\big{(}x_{u,v},x_{u-r,v-r}\big{)}\oplus\big{(}\int_{u}^{v}x_{u,\tau}\otimes dx_{\tau},\int_{u}^{v}x_{u-r,\tau-r}\otimes dx_{\tau}\big{)}\bigg{)} \in \tilde{T}^{2}(U).
\end{align*}

\begin{definition}
Assume $I \subset \R$ and $0 \in I$. We define $\mathcal{C}_0^{0,1-\text{var}}(I, U)$ as the closure of the set of arbitrarily often differentiable paths $x$ from $I$ to $U$ with $x_0 = 0$ with respect to the $1$-variation norm. Furthermore, $C_{0}^{0,p-var}(I,\tilde{T}^{2}(U)) $ is defined as the set of continuous maps $\mathbf{x} \colon I\rightarrow \tilde{T}^{2}(U) $ such that $ \mathbf{x}_{0}=\mathbf{1} $ and for which there exists a sequence $ x_{n}\in C_{0}^{0,1-\text{var}}(I,U) $ with
\begin{align*}
d_{p-var}\big{(}\mathbf{x},\tilde{S}_{2}(x_{n})\big{)}:=\sup_{i,j\in \{1,2 \}}\bigg{(}\sup_{ \mathcal{P} \subset I}\sum_{t_{k}\in\mathcal{P} }\big{\vert}\Pi_{i}^{j}\big{(}\mathbf{x}_{t_{k},t_{k+1}}-\tilde{S}_{2}(x_{n})_{t_{k},t_{k+1}}\big{)}\big{\vert}^{\frac{p}{i}}\bigg{)}^{\frac{1}{p}}\longrightarrow 0
\end{align*} 
as $n \to \infty$. We use the notation $\mathbf{x}_{s,t} := \mathbf{x}_s^{-1} \circledast \mathbf{x}_t$ here. The space $C_{0}^{0,p-var}(\R,\tilde{T}^{2}(U)) $ consists of all continuous paths $\mathbf{x} \colon \R \rightarrow \tilde{T}^{2}(U) $ for which $\mathbf{x} \vert_I \in C_{0}^{0,p-var}(I,\tilde{T}^{2}(U))$ for every $I$ as above.
\end{definition}

We can now state the following results:
\begin{theorem}\label{thm:version_delayed_rough_cocycle}
Let $ p\geqslant 1 $ and let $\overline{\mathbf{X}} $ be an $ C_{0}^{0,p-var}(\mathbb{R},\tilde{T}^{2}(U)) $-valued random variable on a probability space $ (\overline{\Omega},\overline{\mathcal{F}},\overline{\mathbb{P}} )$. Assume that $\mathbf{X}$ has stationary increments, i.e. the law of the process $(\overline{\mathbf{X}}_{t_0,t_0 + h})_{h \in \R} $ does not depend on $t_0 \in \R$. Then we can define a metric dynamical system $(\Omega ,\mathcal{F},\mathbb{P},\theta)$ and a $ C_{0}^{0,p-var}(\mathbb{R},\tilde{T}^{2}(U)) $-valued random variable $ \mathbf{X} $ on $\Omega $ with the same law as $ \overline{\mathbf{X}}$ which satisfies the cocycle property \eqref{eqn:cocycle_rough_delay}.
\end{theorem}
\begin{proof}
The proof in all lines is similar to Theorem 5 in \cite{BRS17} by setting $ \Omega =C_{0}^{0,p-var}(\mathbb{R},\tilde{T}^{2}(U)) $, $\mathcal{F}$ being the Borel $\sigma$-algebra, $\mathbb{P}$ the law of $\bar{\mathbb{X}}$ and for $ \omega\in\Omega $, we define
\begin{align*}
(\theta_{s}\omega)(t):=\omega(s)^{-1}\circledast\omega(t+s) \ , \ \ \ \ \ \mathbf{X}_{t}(\omega)=\omega(t).
\end{align*}
\end{proof}


\begin{remark}
Note that the cocycle property  \eqref{eqn:cocycle_rough_delay} is equivalent to $ \mathbf{X}_{t}(\theta_{s}(\omega))=\mathbf{X}_{s}^{-1}(\omega)\circledast\mathbf{X}_{t+s}(\omega) $ for every $s,t \in \R$ and every $\omega \in \Omega$.
\end{remark}

We will also ask for ergodicity of rough cocycles. The following lemma will be useful.

\begin{lemma}\label{lemma:ergodic}
 Let $(\Omega,\mathcal{F},\mathbb{P},(\theta_t)_{t \in \R})$ and $(\tilde{\Omega},\tilde{\mathcal{F}},\tilde{\mathbb{P}},(\tilde{\theta}_t)_{t \in \R})$ be two measurable metric dynamical systems and let $\Phi \colon \Omega \to \tilde{\Omega}$ be a measurable map such that $\tilde{\mathbb{P}} = \P \circ \Phi^{-1}$. Assume that for every $t \in \R$, there is a set of full $\P$-measure $\Omega_t \subset \Omega$ on which $\Phi \circ \theta_t = \tilde{\theta}_t \circ \Phi$ holds. Then, if $\mathbb{P}$ is ergodic, $\tilde{\mathbb{P}}$ is ergodic, too.
\end{lemma}

\begin{proof}
 The reader will have no difficulties to check that the assertion is just a slight generalization of \cite[Lemma 3]{GS11}.
\end{proof}

\begin{theorem}\label{thm:B_delayed_cocycle}
 Consider the processes $\mathbf{B}^{\text{It\=o}}$ and $\mathbf{B}^{\text{Strat}}$ defined in Section \ref{sec:delayed_area_BM}. Then for each process, we can find an ergodic metric dynamical system $(\Omega ,\mathcal{F},\mathbb{P},\theta)$ on which we can define a new process with the same law, satisfying the cocycle property \eqref{eqn:cocycle_rough_delay}, i.e. both processes are delayed $\gamma$-rough path cocycles for every $\gamma \in (1/3,1/2)$. 
\end{theorem}

\begin{proof}
 We will first consider $\mathbf{B}^{\text{Strat}}$. From the approximation result in Theorem \ref{thm:approx_strat_delay}, we see that $\mathbf{B}^{\text{Strat}}$ takes values in $ C_{0}^{0,p-var}(\mathbb{R},\tilde{T}^{2}(\R^d))$ for every $p \in (2,3)$. It is easy to check that the process has stationary increments, therefore we can apply Theorem \ref{thm:version_delayed_rough_cocycle}. It remains to show ergodicity. By construction, $\Omega = C_{0}^{0,p-var}(\mathbb{R},\tilde{T}^{2}(\R^d))$, $\mathcal{F}$ is the Borel $\sigma$-algebra and $\mathbb{P} = \hat{\mathbb{P}} \circ S^{-1} $ where $(\hat{\Omega},\hat{\mathcal{F}},\hat{\mathbb{P}},\hat{\theta})$ is the measurable metric dynamical system given by $\hat{\Omega} = C_0^0(\mathbb{R},\R^d)$, $\hat{\mathcal{F}}$ the corresponding Borel $\sigma$-algebra, $\hat{\mathbb{P}}$ the Wiener measure and $\hat{\theta} = (\hat{\theta}_t)_{t \in \R}$ the Wiener shift. The map $S \colon \hat{\Omega} \to \Omega$ is defined as follows: For $x \in \hat{\Omega}$, set 
\begin{align*}
  S(x) = \bigg{(}1\oplus\big{(}x_{s,t},x_{s-r,t-r}\big{)}\oplus\big{(}\int_{s}^{t}x_{s,\tau}\otimes dx_{\tau},\int_{s}^{t}x_{s-r,\tau-r}\otimes dx_{\tau}\big{)}\bigg{)}_{s \leq t}
\end{align*}
if the integrals exist as limits of Riemann sums, in Statonovich sense, on compact sets for the sequence of partitions given by $\Pi_n = \{ k/2^{n}\, :\, k \in \Z \}$ as $n \to \infty$, and $S(x) = (1,0,0)$ otherwise. It is not hard to see that there is a set of full $\hat{\mathbb{P}}$-measure on which the  limits do exist. It follows that for every $t \in \R$, there is a set of full measure $\hat{\Omega}_t$ such that for every $x \in \hat{\Omega}_t$,
\begin{align*}
 S(\hat{\theta}_t x) = \theta(S(x)).
\end{align*}
Since $\hat{\mathbb{P}}$ is ergodic, ergodicity of $\mathbb{P}$ follows by Lemma \ref{lemma:ergodic} which completes the proof for the Stratonovich case.

 For the It\=o-case, we can argue analogously: First, we define a map 
 \begin{align*}
\hat{S}_{2}(x)_{s,t}:=\bigg{(}1\oplus\big{(}x_{s,t},x_{s-r,t-r}\big{)}\oplus\big{(}\int_{s}^{t}x_{s,\tau}\otimes dx_{\tau} - \frac{1}{2} (t-s) I_d,\int_{s}^{t}x_{s-r,\tau-r}\otimes dx_{\tau}\big{)}\bigg{)} \in \tilde{T}^{2}(U)
\end{align*}
for smooth paths and a corresponding (separable!) space $ \hat{C}_{0}^{0,p-var}(\mathbb{R},\tilde{T}^{2}(\R^d))$ in which, using again the approximation result for the Stratonovich lift, the random variable $\mathbf{B}^{\text{It\=o}}$ takes its values. Then a version of \cite[Theorem 5]{BRS17} applies and shows the claim. Ergodicity is proven analogously to the Stratonvich case.

\end{proof}

%
%

\subsection{Cocycle property of the solution map}

Let $I \subset \R$ be a compact interval and $X \colon I \to U$ a $\gamma$-H\"older continuous path. 
It is easy to see that for $ \alpha \leq \beta \leq \gamma $, 
\begin{align*}
i_{\alpha ,\beta}: \mathscr{D}_{X}^{\beta}(I,W) &\to \mathscr{D}_{X}^{\alpha}(I,W) ,\\
(\xi,\xi') &\mapsto (\xi,\xi')
\end{align*}
is a continuous embedding. We make the following definition:
\begin{definition}
 We define $ \mathscr{D}_{X}^{\alpha ,\beta}(I,W)$ as the closure of $\mathscr{D}_{X}^{\beta}(I,W)$ in the space $\mathscr{D}_{X}^{\alpha}(I,W)$. 
\end{definition}
The reason why we introduce these spaces is their separability, which we will prove in the next lemma.
\begin{lemma}
 For all $\alpha < \beta$, the spaces $ \mathscr{D}_{X}^{\alpha ,\beta}(I,W)$ are separable.
\end{lemma}
\begin{proof}
 The space $ \mathscr{D}_{X}^{\alpha ,\beta}(I,W))$ can be viewed as a subset of
 \begin{align*}
  {C}^{\alpha,\beta}(I,W) \times {C}^{\alpha,\beta}(I,L(U,W)) \times {C}^{2 \alpha,2 \beta}(I,W)
 \end{align*}
 where ${C}^{\alpha,\beta}$ again means taking the closure of $\beta$-H\"older functions in the $\alpha$-H\"older norm. Since all spaces above are separable, the result follows. 

\end{proof}

If the parameters $\alpha < \beta < \gamma$ satisfy a certain condition, we can find a very explicit dense subset. This is the content of the next theorem, which has far reaching consequences, as we will see.
\begin{theorem}\label{structre}
Let $\alpha < \beta < \gamma \leq 1/2$. Assume that there is a $\kappa \in (0, \gamma)$ such that
\begin{align}\label{eqn:alpha_beta_gamma}
  \beta -\alpha>\frac{(1-\alpha)(1-\beta-\gamma+\kappa)}{(1-\beta)(1-2\alpha +\kappa)}
\end{align}
 holds. Then the set 
 \begin{align*}
  \left\{ (\psi,\psi')\ |\ \psi_{s,t} = \int_{s}^{t}f(\tau)\, dX_{\tau} + R_{s,t},\ \psi'_s = f(s) \text{ where }\ f \in C^{\infty}\big{(} I,L(U,W)\big{)} \text{ and } R \in C^{\infty}(I ,W) \right\}
 \end{align*}
 is dense in $\mathscr{D}_X^{\alpha ,\beta}(I,W)$, the integral being understood as a Young-integral here. In particular, $\mathscr{D}_X^{\alpha ,\beta}(I,W)$ does not depend on $\beta$ in this case.
\end{theorem}

\begin{proof}
Let $I = [a,b]$ and $r = b-a$. Take $ (\xi,\xi') \in \mathscr{D}_X^{\alpha ,\beta}(I,W))$, i.e. $ \xi_{s,t} = \xi^{\prime}_{s}X_{s,t}+\xi^{\#}_{s,t} $, and assume that $ \Vert\xi^{\prime}\Vert _{\beta; I} , \Vert\xi^{\#}\Vert_{2\beta; I} <\infty $. Let $ \phi :\mathbb{R}\rightarrow\mathbb{R} $ be a mollifier. Note that we can extend $ \xi^{\prime} $ to a $ \beta $-H\"older function on $ \mathbb{R} $ by setting $ \xi_{t}^{\prime}=\xi'_b$ for $ t\geqslant b$ and $ \xi_{t}^{\prime}=\xi_{a}^{\prime} $ for $ t\leqslant a $ . For given $n \in \N$, we define a smooth function $f: I \rightarrow L(U,W) $ by setting 
\begin{align*}
f_{s} := \int_{\mathbb{R}}\xi_{s -\frac{1}{n}z}^{\prime}\phi(z)\, dz =n\int_{\mathbb{R}}\xi_{z}^{\prime}\phi\big{(}n(s -z)\big{)}\, dz.
\end{align*}
Our goal is to find a smooth function $R$ such that for  
\begin{align}\label{eqn:def_psi}
 \psi_{s,t} := \int_s^t f(\tau)\, dX_{\tau} + R_{s,t} \quad \text{and} \quad \psi'_s := f(s),
\end{align}
we have $\|(\xi,\xi') - (\psi,\psi')\|_{\mathscr{D}_X^{\alpha}(I,W))} < \varepsilon$ for any given $\varepsilon > 0$ when choosing $n$ large enough. Note that
\begin{align*}
 \psi^{\#}_{s,t} = \int_s^t f(\tau) - f(s)\, dX_{\tau} + R_{s,t}
\end{align*}
 is finite $2\beta$-H\"older continuous by standard Young estimates, which implies that $(\psi,\psi')$ is indeed an element in $\mathscr{D}_X^{\alpha ,\beta}(I,W)$.

Note first that for  $ \eta_{s,t}= f_{s,t}-\xi_{s,t}^{\prime} $,
\begin{align}\label{A2}
\begin{split}
\big{\vert}\eta_{s,t}\big{\vert}&=
\big{\vert} f_{s,t}-\xi_{s,t}^{\prime}\big{\vert} \leqslant\int_{\mathbb{R}}\phi(z)\big{\vert}\xi^{\prime}_{s -\frac{1}{n}z,t -\frac{1}{n}z}-\xi^{\prime}_{s,t}\big{\vert}dz\\&=\int_{\mathbb{R}}\phi(z)\big{\vert} \xi^{\prime}_{s -\frac{1}{n}z,t -\frac{1}{n}z}-\xi^{\prime}_{s ,t}\big{\vert}^{\frac{\alpha}{\beta}}\big{\vert} \xi^{\prime}_{s -\frac{1}{n}z,t -\frac{1}{n}z}-\xi^{\prime}_{s,t}\big{\vert}^{1-\frac{\alpha}{\beta}}dz\lesssim (t-s)^{\alpha}(\frac{1}{n})^{\beta -\alpha} 
\end{split}
\end{align}
which implies that $\| \xi' - \psi'\|_{\alpha;I} + | \xi'_a - \psi'_a| \to 0$ as $n \to \infty$. Since 
\begin{align*}
  \frac{d}{d s} f_{s} = n\int_{\mathbb{R}}\xi_{s -\frac{1}{n}z}^{\prime}\frac{d}{d z}\phi(z)dz 
\end{align*}
 and from $\beta $-H\"older continuity of $ \xi^{\prime} $, 
\begin{align*}
 \big{\vert} f_{s,t}\big{\vert}\lesssim (t-s)^{\beta} \quad \text{and} \quad \big{\vert} f_{s,t}\big{\vert}\lesssim n(t-s).
\end{align*}
By polarization, this implies that
\begin{align*}
 \vert f_{s,t}\vert\lesssim n^{\theta}(t-s)^{\theta +\beta(1-\theta)}
\end{align*}
holds for every $ 0\leqslant\theta\leqslant 1 $.
Setting $\theta =\frac{1-\gamma -\beta +\kappa}{1-\beta} $, we obtain
\begin{align}\label{C2}
\big{\vert}f_{s,t}\big{\vert}\lesssim n^{\frac{1-\gamma -\beta +\kappa}{1-\beta}}(t-s)^{1-\gamma +\kappa}.
\end{align}
Let $ \rho_{s,t}:=\int_{s}^{t}f_{\tau}\, dX_{\tau} -f_{s}X_{s,t} $. From the Young inequality and relation (\ref{C2}),
\begin{align}\label{inter2}
\big{\vert}\rho_{s,t}\big{\vert}\lesssim \Vert f\Vert_{1-\gamma +\kappa}\Vert X\Vert_{\gamma}(t-s)^{1+\kappa} \lesssim n^{\frac{1-\gamma -\beta +\kappa}{1-\beta}}(t-s)^{1+\kappa}.
\end{align}
Let $ \epsilon>0 $ be given and set $ \tilde{\rho} := \xi^{\#}_{s,t}-\rho_{s,t} $. Since $ \Vert\tilde{\rho}\Vert_{2\beta}<\infty$, we can choose $ \delta>0 $ small enough such that
\begin{align}\label{A1}
\sup_{s,t\in I,\vert t-s\vert\leqslant\delta} \frac{\vert\tilde{\rho}_{s,t}\vert}{(t-s)^{2\alpha}}\leqslant\Vert\xi^{\#}\Vert_{2\beta}\delta^{2(\beta -\alpha)}+\sup_{s,t\in I,\vert t-s\vert\leqslant\delta}\frac{\vert\rho_{s,t}\vert}{(t-s)^{2\alpha}}\leqslant\epsilon.
\end{align} 
More precisely, we see from (\ref{inter2}) that $\delta$ has to be chosen such that
\begin{align}\label{A3}
\delta  \lesssim\epsilon ^{\frac{1}{2(\beta -\alpha)}}\ \ \ \ \ \text{and} \ \ \ \ \ \ \ \ \ \delta\lesssim\big{(}\frac{\epsilon}{n^{\frac{1-\beta -\gamma +\kappa}{1-\beta}}}\big{)}^{\frac{1}{1-2\alpha +\kappa}}.
\end{align}
We will determine $ n $ and consequently $ \delta $ in the future. Now, it is easy to verify that
\begin{align}\label{remain}
\tilde{\rho}_{s,t}-\tilde{\rho}_{s,u}-\tilde{\rho}_{u,t} =(\xi^{\prime}_{s,u}-f_{s,u})X_{u,t}=\eta_{s,u}X_{u,t}.
\end{align}
Let $ P=\lbrace a = t_{0} < t_{1} < ... < t_{m} = b\rbrace $ be a partition of $I$ such that $ t_{i}-t_{i-1} =\delta $ for $1\leqslant i\leqslant m-1 $ and $ t_{m}-t_{m-1}\leqslant\delta $. We can define a piecewise linear function $\tilde{R} $ by
\begin{align*}
\tilde{R}_{\tau ,\upsilon}:=\tilde{\rho}_{t_{k},t_{k+1}}\frac{\upsilon -\tau}{t_{k+1}-t_{k}}, \ \ \ \ \tau,\upsilon\in [t_{k},t_{k+1}]
\end{align*}
and consequently for $ \tau ,\upsilon\in I$ with $ t_{k}\leqslant \tau\leqslant t_{k+1}\leqslant ...\leqslant t_{j}\leqslant\upsilon\leqslant t_{j+1} $,
\begin{align*}
\tilde{R}_{\tau,\upsilon}=\frac{t_{k+1}-\tau}{t_{k+1}-t_{k}}\tilde{\rho}_{t_{k},t_{k+1}}+\tilde{\rho}_{t_{k+1},t_{k+2}}+....+\frac{\upsilon -t_{j}}{t_{j+1}-t_{j}}\tilde{\rho}_{t_{j},t_{j+1}}.
\end{align*}
From (\ref{remain}),
\begin{align*}
&\tilde{R}_{\tau,\upsilon}-\tilde{\rho}_{\tau,\upsilon}=\big{(}\frac{t_{k+1}-\tau}{t_{k+1}-t_{k}}\tilde{\rho}_{t_{k},t_{k+1}}+....+\frac{\upsilon -t_{j}}{t_{j+1}-t_{j}}\tilde{\rho}_{t_{j},t_{j+1}}\big{)}-\big{(} \tilde{\rho}_{\tau ,t_{k+1}}+...+\tilde{\rho}_{t_{j-1},t_{j}}+\tilde{\rho}_{t_{j},\upsilon} \big{)}\\ &-\big{(}\eta_{\tau ,t_{k+1}}X_{t_{k+1},\upsilon}+\eta_{t_{k+1},t_{k+2}}X_{t_{k+2},\upsilon}+...+\eta_{t_{j-1},t_{j}}X_{t_{j},\upsilon}\big{)}\\&=\big{(}\frac{t_{k+1}-\tau}{t_{k+1}-t_{k}}\tilde{\rho}_{t_{k},t_{k+1}}+\frac{\upsilon -t_{j}}{t_{j+1}-t_{j}}\tilde{\rho}_{t_{j},t_{j+1}}-\tilde{\rho}_{\tau ,t_{k+1}}-\tilde{\rho}_{t_{j},\upsilon} \big{)}-\big{(} \eta_{\tau ,t_{k+1}}X_{t_{k+1},\upsilon}+...+\eta_{t_{j-1},t_{j}}X_{t_{j},\upsilon} \big{)}\\&=: A - B.
\end{align*}
From (\ref{A1}), it is not hard to verify that
\begin{align}
\Vert A \Vert\leqslant 4\epsilon(\upsilon -\tau)^{2\alpha}.
\end{align}
By (\ref{A2}) and our assumptions on $ X $,
\begin{align*}
\Vert B \Vert\lesssim(\frac{1}{n})^{\beta -\alpha}\big{[}(t_{k+1}-\tau)^{\alpha}(\upsilon -t_{k+1})^{\gamma}+...+(t_{j}-t_{j-1})^{\alpha}(\upsilon -t_{j})^{\gamma}\big{]}\leqslant\frac{\delta^{\alpha +\gamma}}{n^{\beta -\alpha}}\big{[}(j-k)^{\gamma}+...+1^{\gamma}\big{]}.
\end{align*}
Since $ (j-k-1)\delta\leqslant \upsilon -\tau $ and $ m\delta\leqslant r $,
\begin{align}
\frac{\Vert B \Vert}{(\upsilon -\tau)^{2\alpha}}\lesssim\frac{\delta^{\gamma -\alpha}}{n^{\beta -\alpha}}m^{\gamma +1-2\alpha} \lesssim\frac{1}{n^{\beta -\alpha}}\frac{1}{\delta^{1-\alpha}}.
\end{align}
Now from (\ref{A3}), if $ \beta -\alpha>\frac{(1-\alpha)(1-\beta-\gamma+\kappa)}{(1-\beta)(1-2\alpha +\kappa)} $, we can find $n$ and $\delta $ such that
\begin{align*}
  \frac{\Vert B \Vert}{(\upsilon -\tau)^{2\alpha}}\leqslant\epsilon
\end{align*}
and therefore
\begin{align*}
 \| \tilde{R} - \tilde{\rho} \|_{ 2\alpha;I} \leq 5 \epsilon.
\end{align*}
Since $ \tilde{R} $ is a piecewise linear function, we can find an $ R\in C^{\infty}(I,W) $ such that
\begin{align*}
\sup_{s,t\in I}\frac{\vert R_{s,t}-\tilde{R}_{s,t}\vert}{(t-s)^{2\alpha}}\leqslant\epsilon.
\end{align*}
Using this $R$ in \eqref{eqn:def_psi}, we obtain
\begin{align*}
  \| \xi^{\#} - \psi^{\#} \|_{2\alpha ; I} \leq \| \tilde{\rho} - \tilde{R}  \|_{ 2\alpha;I} + \| \tilde{R} - R \|_{ 2\alpha;I} \leq 6 \epsilon,
\end{align*}
thus the stated set is indeed dense. 

\end{proof}

\begin{remark}
 In the applications we have in mind, $\gamma < 1/2$ can be chosen arbitrarily close to $1/2$. Let $\alpha < 1/2$ be given. Choose $\alpha < \beta < \gamma < 1/2$ and $0 < \kappa < \gamma$ such that
 \begin{align*}
  2(1 - \beta - \gamma + \kappa) < (1 - 2\alpha + \kappa)(\beta - \alpha).
 \end{align*}
 Note that this can always be achieved by choosing $\beta$ and $\gamma$ close to $1/2$ and $\kappa$ close to $0$. Since $1-\alpha < 1$ and $1 - \beta > 1/2$, this implies that
 \begin{align*}
  \frac{(1 - \alpha)(1 - \beta - \gamma + \kappa)}{(1-\beta)(1 - 2\alpha + \kappa)} < \beta - \alpha,
 \end{align*}
 i.e. in this case, we can always find parameters such that the condition in Theorem \ref{structre} is satisfied.

\end{remark}

\begin{theorem}\label{fiber}
Let $\mathbf{X}$ be a delayed $\gamma$-rough path cocycle for some $\gamma \in (1/3,1/2]$. Under the assumptions of Theorem \ref{thm:flow_prop}, the map
\begin{align}\label{eqn:cocyle_delay}
\varphi(n, \omega, \cdot) := \phi(0,nr,\omega, \cdot)
\end{align}
is a continuous map
\begin{align*}
\varphi(n, \omega, \cdot) \colon \mathscr{D}_{X(\omega)}^\beta([-r , 0],W) \to \mathscr{D}_{X(\theta_{nr} \omega)}^\beta([-r , 0],W)
\end{align*}
and the cocycle property
\begin{align}
\varphi(n+m , \omega, \cdot) = \varphi(n, \theta_{mr} \omega, \cdot) \circ \varphi(m, \omega, \cdot)
\end{align}
holds for every $s ,t \in [0,\infty)$. 
If $\sigma$ is linear, the cocycle is compact linear. Furthermore, all assertions remain true if we replace the spaces $\mathscr{D}^\beta$ by $\mathscr{D}^{\alpha,\beta}$ for $1/3 < \alpha < \beta < \gamma$.
\end{theorem}

\begin{proof}
Note that $\mathscr{D}_{X(\omega)}^\beta([-r+nr , nr],W) \cong \mathscr{D}_{X(\theta_{nr}\omega)}^\beta([-r , 0],W) $ by the natural linear map 
\begin{align*}
\Psi \colon \mathscr{D}_{X(\omega)}^\beta([-r+nr , nr],V) &\longrightarrow \mathscr{D}_{X(\theta_{nr}\omega)}^\beta([-r , 0],V)\\
(\xi_{\tau})_{-r+nr\leqslant\tau\leqslant nr} &\mapsto (\tilde{\xi}_{\tau}=\xi_{\tau +nr})_{-r\leqslant\tau\leqslant 0}.
\end{align*}
Continuity of $\varphi$ is a consequence of Theorem \ref{thm:flow_prop}. Regarding the cocycle property, by the semi-flow property (\ref{Flow}) it is enough to show that
\begin{align*}
\phi(0,nr,\theta_{mr}\omega, \cdot)=\phi\big{(}mr,(m+n)r,\omega,\cdot\big{)}.
\end{align*}
Using again the semi-flow property (\ref{Flow}), it is enough to show the equality for $ n=1 $ only. Finally, by the definition of the integral in (\ref{dfn}) and the cocycle property of a rough cocycle, this can easily be verified. 
The statements about linearity and compactness are a consequence of \ref{thm:delay_linear} and Proposition \ref{prop:compact_map}. The claim that all spaces $\mathscr{D}^\beta$ can be replaced by $\mathscr{D}^{\alpha,\beta}$ follows from the invariance 
\begin{align}\label{separable}
\varphi\big{(}n,\omega,\mathscr{D}_{X(\omega)}^{\alpha,\beta}([-r,0],W)\big{)}\subset \mathscr{D}_{X(\theta_{nr}\omega)}^{\alpha ,\beta}([-r,0],W)
\end{align}
which is a consequence of the continuity of $\varphi$. 
\end{proof}

Note that so far, we worked with delayed rough path cocycles $\mathbf{X}$ which are defined on a continuous-time metric dynamical system $(\Omega,\mathcal{F},\P,(\theta_t)_{t \in \R})$. In Theorem \ref{fiber}, we saw that stochastic delay equations a priori induce discrete-time RDS only. The reason is that we cannot expect that the semi-flow property \eqref{eqn:flow_prop} holds in full generality for all times, cf. Theorem \ref{thm:flow_prop}. Therefore, in what follows, we will continue working with discrete time only. From now on, whenever we consider cocycles induced by delay equations with delay $r > 0$, our underlying discrete-time metric dynamical system is given by  $(\Omega,\mathcal{F},\P,\theta)$ with $\theta := \theta_r$. We also use the notation $\varphi(\omega,\cdot) := \varphi(1,\omega,\cdot)$ for the cocycle $\varphi$ defined in \eqref{eqn:cocyle_delay}. 

Next, we describe a structure which will be useful for us.

\begin{definition}\label{def:meas_field_banach}
 Let $(\Omega,\mathcal{F})$ be a measurable space. A family of Banach spaces $\{E_{\omega}\}_{\omega \in \Omega}$ is called a \emph{measurable field of Banach spaces} if there is a set of sections
 \begin{align*}
  \Delta \subset \prod_{\omega \in \Omega} E_{\omega}
 \end{align*}
 with the following properties:
 \begin{itemize}
  \item[(i)] $\Delta$ is a linear subspace of $\prod_{\omega \in \Omega} E_{\omega}$.
  \item[(ii)] There is a countable subset $\Delta_0 \subset \Delta$ such that for every $\omega \in \Omega$, the set $\{g(\omega)\, :\, g \in \Delta_0\}$ is dense in $E_{\omega}$.
  \item[(iii)] For every $g \in \Delta$, the map $\omega \mapsto \| g(\omega) \|_{E_{\omega}}$ is measurable.
 \end{itemize}

\end{definition}

\begin{remark}
 Let us remark here that the former definition originates from us, we did not encounter a description of a \emph{measurable field of Banach spaces} elsewhere in the literature. In fact, it is a mix of a \emph{measurable field of Hilbert spaces} to be found e.g. in \cite[page 220]{Fol95} and a \emph{continuous field of Banach spaces}, cf. e.g. \cite[page 211]{Dix77}. Since the stated properties in Definition \ref{def:meas_field_banach} are exactly what we need for proving the Multiplicative Ergodic Theorem in the next section, it is also a pragmatic definition. 
\end{remark}

\begin{proposition}
 Let $X \colon \Omega \to C^{\gamma}(I,U)$ be a stochastic process. Assume that there are $\alpha < \beta < \gamma$ and some $\kappa \in (0,\gamma)$ such that \eqref{eqn:alpha_beta_gamma} is satisfied. Then $\{\mathscr{D}^{\alpha,\beta}_{X(\omega)}(I,W)\}_{\omega \in \Omega}$ is a measurable field of Banach spaces.
\end{proposition}

\begin{proof}
 For $s = (v,f,R) \in \R \times C^{\infty}(I,L(U,W)) \times C^{\infty}_0(I,W)$, define
 \begin{align*}
  g_s(\omega) := \left(v + \int_{-r}^{\cdot} f(\tau)\, dX_{\tau}(\omega) + R, f\right) \in \mathscr{D}^{\alpha,\beta}_{X(\omega)}(I,W)
 \end{align*}
 and set 
 \begin{align}\label{eqn:def_Delta}
  \Delta := \{g_s\, :\, s \in \R \times C^{\infty}(I,L(U,W)) \times C^{\infty}_0(I,W) \}.
 \end{align}
  It is clear that (i) holds for $\Delta$. Let $S$ be a countable and dense subset of $\R \times C^{\infty}(I,L(U,W)) \times C^{\infty}_0(I,W)$ and define $\Delta_0 := \{g_s\, :\, s \in S \}$. By definition, $\Delta_0$ is countable, and $\{g_s(\omega)\, :\, s \in S\}$ is dense in $\mathscr{D}^{\alpha,\beta}_{X(\omega)}(I,W)$ for fixed $\omega \in \Omega$ by Theorem \ref{structre}. It remains to prove (iii). Let $I = [a,b]$ and choose $s = (v,f,R)$. Then
 \begin{align*}
\Vert g_s(\omega) \Vert =\vert v\vert &+\vert f(a)\vert +\sup_{s,t\in I \cap \mathbb{Q},s<t}\frac{\vert f(t)-f(s)\vert}{(t-s)^{\alpha}}\\
&+\sup_{s,t\in I \cap \mathbb{Q},s<t}\frac{\vert R_{s,t}+ \int _{s}^{t}f(\tau)\, dX_{\tau} (\omega) -f(s)X_{s,t}(\omega)\vert}{(t-s)^{2\alpha}}.
 \end{align*}
 The integral is measurable since it is a limit of measurable Riemann sums. Measurability of $\omega \mapsto \Vert g_s(\omega) \Vert$ thus follows which finishes the proof.

\end{proof}

\begin{definition}\label{def:RDS_on_Banach_field}
 Let $(\Omega,\mathcal{F},\P,\theta)$ be a measurable metric dynamical system and $(\{E_{\omega}\}_{\omega \in \Omega},\Delta)$ a measurable field of Banach spaces. A \emph{continuous cocycle on $\{E_{\omega}\}_{\omega \in \Omega}$} consists of a family of continuous maps
 \begin{align}\label{eqn:def_cocycle}
  \varphi(\omega, \cdot) \colon E_{\omega} \to E_{\theta \omega}.
 \end{align}
 If $\varphi$ is a continuous cocycle, we define $\varphi(n,\omega,\cdot) \colon E_{\omega} \to E_{\theta^n \omega}$ as 
 \begin{align*}
  \varphi(n,\omega,\cdot) := \varphi(\theta^{n-1}\omega,\cdot) \circ \cdots \circ \varphi(\omega,\cdot).
 \end{align*}
 We say that \emph{$\varphi$ acts on $\{E_{\omega}\}_{\omega \in \Omega}$} if the maps
 \begin{align*}
  \omega \mapsto \| \varphi(n,\omega,g(\omega)) \|_{E_{\theta^n \omega}}, \quad n \in \N
 \end{align*}
 are measurable for every $g \in \Delta$. In this case, we will speak of a \emph{continuous random dynamical system on a field of Banach spaces}. If the map \eqref{eqn:def_cocycle} is bounded linear/compact, we call $\varphi$ a bounded linear/compact cocycle.

\end{definition}

\begin{theorem}\label{eqn:delay_induce_RDS}
 The continuous cocycle
 \begin{align*}
  \varphi(\omega,\cdot) \colon \mathscr{D}_{X(\omega)}^{\alpha,\beta}([-r,0],W) \to \mathscr{D}_{X(\theta_r \omega)}^{\alpha,\beta}([-r,0],W)  
 \end{align*}
 defined in Theorem \ref{fiber} induces a random dynamical system on the field of Banach spaces \\ $\{\mathscr{D}_{X(\omega)}^{\alpha,\beta}([-r,0],W)\}_{\omega \in \Omega}$. 

\end{theorem}

\begin{proof}
Let $\Delta$ be defined as \eqref{eqn:def_Delta} and take $g \in \Delta$. Consider the solution $y$ to 
\begin{align*}
 y_t(\omega) &= g_0(\omega) + \int_0^t \sigma(y_{\tau}(\omega),y_{\tau - r}(\omega))\, d \mathbf{X}_{\tau}(\omega),\quad t \geq 0; \\
 y_t(\omega) &= g_t(\omega), \quad t \in [-r,0].
\end{align*}
To simplify notation, set $\| \cdot \|_{\mathscr{D}_{X(\omega)}([0,r])} :=  \| \cdot \|_{\mathscr{D}_{X(\omega)}^{\alpha,\beta}([0,r],W)}$. We will prove that $\omega \mapsto \| y(\omega) \|_{\mathscr{D}_{X(\omega)}([0,r])}$ is measurable. Define
\begin{align*}
y^{1}_t(\omega) := g_0(\omega) + \int_{0}^{t} \sigma \big{(} g_0(\omega) ,g_{\tau - r}(\omega) \big{)}\, d\mathbf{X}_{\tau}(\omega)
\end{align*}
and recursively for $ n\geqslant 1 $
\begin{align*}
y^{n+1}_t(\omega) := g_0(\omega) + \int_{0}^{t}\sigma \big{(} y^{n}_{\tau}(\omega) ,g_{\tau - r}(\omega) \big{)}\, d\mathbf{X}_{\tau}(\omega).
\end{align*}
By induction, one can show that $\omega \mapsto y^n_t(\omega)$ is measurable for every $t \in [0,r]$ and $n \geq 1$. 
By a similar strategy for proving continuity of the It\=o-Lyons map, one can show that $y^n(\omega) \to y(\omega)$ in the space $\mathscr{D}_{X(\omega)}^{\alpha,\beta}([0,T(A(\omega))],W)$ as $n \to \infty$ where  
\begin{align*}
  A(\omega) = \Vert X(\omega)\Vert_{\gamma ;[0,r]} + \Vert\mathbb{X}(\omega)\Vert_{2\gamma ;[0,r]} + \Vert\mathbb{X}(\omega)(-r)\Vert_{2\gamma ;[0,r]}
\end{align*}
and $T \colon [0,\infty) \to (0,r]$ is a decreasing function. 
Define
\begin{align*}
  \Omega_m := \left\{ \omega \in \Omega \,:\, T(A(\omega)) \leq \frac{r}{m} \right\}.
\end{align*}
Then $ \Omega_{m} $ is a measurable subset and $ \Omega = \bigcup_{m \geqslant 1} \Omega_{m} $. 
Fix $m \in \N$ and choose $ \omega \in \Omega_{m} $. Then $ (y^{n}(\omega))_n$ is a Cauchy sequence in the space $\mathscr{D}_{X(\omega)}^{\alpha,\beta}([0,r/m],W)$ and, consequently, converges to some element $\tilde{y}^{0}(\omega)$ for which we can conclude that $\omega \mapsto \tilde{y}^{0}_t(\omega)$ is measurable for every $t \in [0,r/m]$. Now we can repeat this argument in $[\frac{jr}{m},\frac{(j+1)r}{m}]$ for $j = 0, \ldots, m-1$ and obtain a sequence of elements $\tilde{y}^{j}(\omega) \in \mathscr{D}_{X(\omega)}^{\alpha,\beta}([jr/m,(j+1)r/m],W)$ with the properties that $\omega \mapsto \tilde{y}_t^j(\omega)$ is measurable for every $t \in [jr/m,(j+1)r/m]$ and 
\begin{align*}
  y_t(\omega) = \sum_{j=0}^{m-1} \tilde{y}_t^{j}(\omega) \chi_{[\frac{jr}{m},\frac{(j+1)r}{m})}(t).
\end{align*}
This implies that $\omega \mapsto y_t(\omega)$ is measurable for every $t \in [0,r]$ on the subspace $\Omega_m$. Since $m$ was arbitrary, measurability follows also on the space $\Omega$. Note that $y'_t(\omega) = \sigma(y_t(\omega),g_{t - r}(\omega))$, thus
\begin{align*}
 \| y(\omega) \|_{\mathscr{D}_{X(\omega)}([0,r])} = |y_0(\omega)| &+ |y'_0(\omega)| + \sup_{s < t \in [0,r] \cap \Q} \frac{|y'_{s,t}|}{|t-s|^{\alpha}} \\
 &+ \sup_{s < t \in [0,r] \cap \Q} \frac{\left| \int_s^t \sigma(y_{\tau}(\omega),g_{\tau - r}(\omega)) \, d\mathbf{X}_{\tau}(\omega) - \sigma(y_{s}(\omega),g_{s - r}(\omega)) \right|}{|t-s|^{2\alpha}}
\end{align*}
and measurability of $\omega \mapsto \| y(\omega) \|_{\mathscr{D}_{X(\omega)}([0,r])}$ follows. We can now repeat this argument to see that $\omega \mapsto \| y(\omega) \|_{\mathscr{D}_{X(\omega)}([nr,(n+1)r])}$ is measurable for every $n \geq 0$ which proves the theorem.


\end{proof}
\section{A Multiplicative Ergodic Theorem on a measurable field of Banach spaces}\label{sec:MET_Banach_fields}

In this section, $(\Omega,\mathcal{F},\mathbb{P},\theta)$ will denote a measurable metric dynamical system, $(\{E_{\omega}\}_{\omega \in \Omega},\Delta, \Delta_0)$ will be a measurable field of Banach spaces as in Definition \ref{def:meas_field_banach} and $\varphi$ a bounded linear cocycle acting on it, cf. Definition \ref{def:RDS_on_Banach_field}. Our goal is to prove a Multiplicative Ergodic Theorem (MET) in this abstract setting. The strategy we use is close to the one introduced in two recent works, both proving an MET on a Banach space. The first one is due to Blumenthal \cite{Blu16}, the second was written by Gonz\'{a}lez-Tokman and Quas \cite{GTQ15}. 
Note, however, that none of them gives a proof of the MET for cocycles acting on fields of Banach spaces. For that reason, the measurability assumption in these works is very different from ours, and we have to prove measurability for our objects in a completely different way. Furthermore, we do not assume reflexivity of the Banach spaces as in \cite{GTQ15}.

We start with an easy observation.

\begin{lemma}\label{lemma:oper_meas}
 For every $n \in \N$, the map
 \begin{align*}
  \omega \mapsto \| \varphi(n,\omega,\cdot) \|_{L(E_{\omega},E_{\theta^n \omega})}
 \end{align*}
 is measurable. 

\end{lemma}

\begin{proof}
 Using properties of $\Delta$ and continuity of $\varphi$,
 \begin{align*}
  \| \varphi(n,\omega,\cdot) \|_{L(E_{\omega},E_{\theta^n \omega})} = \sup_{\xi \in E_{\omega} \setminus\{0\}} \frac{\|\varphi(n,\omega,\xi)\|}{\| \xi \|} = \sup_{g \in \Delta_0} \frac{\|\varphi(n,\omega,g(\omega))\|}{\| g(\omega) \|} \chi_{\{ \|g\| > 0\}}(\omega)
 \end{align*}
 with the convention $\infty \cdot 0 = 0$. Since the fraction on the right hand side is a quotient of measurable functions and the supremum runs over a countable set, measurability follows.

\end{proof}

\begin{definition}
 Let $V$ be a vector space. If we can write $V$ as a direct sum $V = F \oplus H$ of vector spaces, we call it an \emph{algebraic splitting}. We also say that \emph{$F$ is a complement of $H$} and vice versa. The projection operator $\pi_{F \Vert H}(v) = f$ with $v = f + h$, $f \in F$, $h \in H$, is called the \emph{projection operator onto $F$ parallel to $H$}. If $V$ is a normed space and $\pi_{F \Vert H}$ is bounded linear, i.e.
\begin{align*}
 \| \pi_{F \Vert H} \| = \sup_{f \in F, e \in H, f + h \neq 0} \frac{\| f \|}{\|f + h\|} < \infty,
\end{align*}
we call $V = F \oplus H$ a \emph{topological splitting}. 
\end{definition}

%

%
%

The next lemma proves a further measurability result. The assumptions will be justified in the sequel.

\begin{lemma}\label{Z2}
 For $\omega \in \Omega$ and $\mu \in \R$, define the subspace
 \begin{align*}
  F_{\mu}(\omega) := \left\{ \xi \in E_{\omega} \,:\, \limsup_{n \to \infty} \frac{1}{n} \log \| \varphi(n,\omega,\xi) \| \leq \mu \right\}.
 \end{align*}
 Assume that there is a strictly decreasing sequence $ (\mu_{j})_{1 \leqslant j\leqslant N} $, $N \leqslant \infty$, and a $\theta$-invariant, measurable set $\Omega_0 \subset \Omega $ of full measure with the following properties:
 \begin{itemize}
    \item[(i)] $F_{\mu_1}(\omega) = E_{\omega}$ for every $\omega \in \Omega_0$.
  \item[(ii)] For every $ j<N $, there is  a number $m_{j} \in \N$ such that $F_{\mu_{j+1}}(\omega)$ is closed and $m_{j}$-codimensional in $ F_{\mu_{j}}(\omega) $ for every $\omega \in \Omega_0$.
   \item[(iii)]  For every $ j< N $,
  \begin{align}\label{eqn:_bound}
  \lim_{n \to \infty} \frac{1}{n} \log \|\varphi(n,\omega,\cdot)|_{F_{\mu{j}}(\omega)}\| = \mu_j
\end{align}
  for every $\omega \in \Omega_0$.
  \item[(iv)] For every $ j<N $, if $H_{\omega}^{j}$ is any complement of $F_{\mu_{j+1}}(\omega)$ in $F_{\mu_{j}}(\omega)$,
 \begin{align}\label{eqn:lower_bound_subspace}
  \lim_{n \to \infty} \frac{1}{n} \log \inf_{h \in H_{\omega}^{j} \setminus \{0\}} \frac{\|\varphi(n,\omega,h)\|}{\|h\|} = \mu_j
  \end{align}
  for every $\omega \in \Omega_0$.
  \item[(v)]
  \begin{align}\label{eqn:upper_bound_restr_subsp}
   \limsup_{n \to \infty} \frac{1}{n} \log \| \varphi(n,\omega,\cdot)\mid_{F_{\mu_N}(\omega)} \| \leq \mu_N
  \end{align}
  for every $\omega \in \Omega_0$.
 \end{itemize}
Then for every $ n \in \N$ and $ j\leqslant N $ , the map
\begin{align}\label{BCA}
\omega \mapsto {\Vert}\varphi(n,\omega,\cdot) \mid_{F_{\mu_j}(\omega)}{\Vert} \chi_{\Omega_0}(\omega)
\end{align}
is measurable.
\end{lemma}

\begin{proof}
First we claim that for every $ g\in \Delta $ and $ j\leqslant N $ the map
\begin{align}
\omega \mapsto d\big{(}g(\omega) , F_{\mu_{j}}(\omega)\big{)}
\end{align}
is measurable. To see this, it suffices to show measurability of the function
\begin{align*}
    d\big{(}g(\omega), S_{F_{\mu_{j}}(\omega)}\big{)}:=\inf_{\substack{ \xi \in F_{\mu_{j}}(\omega)\\[2pt] \Vert\xi \Vert=1}}\Vert g(\omega)-\xi \Vert
\end{align*}
where $ S_{F_{\mu_{j}}(\omega)} $ is the unit sphere in $ F_{\mu_{j}}(\omega) $. We use induction to prove the claim. The statement is clear for $j = 1$, so let $ j\geqslant 2 $. For every $ 1\leqslant i<j $, since $ \dim\big{[}\frac{F_{\mu_{i}}(\omega)}{F_{\mu_{i+1}}(\omega)}\big{]}<\infty $, we can find  a finite-dimensional subspace $ H_{i}(\omega)$ such that for a constant\footnote{The existence of this complement with the given bound for the projection is a classical result and follows e.g. from \cite[III.B.11]{Woj91}, cf. also \cite[Lemma 2.3]{Blu16}.}  $ M $,
\begin{align}\label{complement}
F_{\mu_{i}}(\omega)=H_{i}(\omega)\oplus F_{\mu_{i+1}}(\omega)\ \ \ \text{and}\ \ \ \Vert\pi_{H_{i}(\omega)||F_{\mu_{i+1}}(\omega)}\Vert< M.
\end{align}
For $ \mu_{0}:=\mu_{1} $ and $l,k \geq 1$ set 
\begin{align*}
B_{\omega}^{l,k}(\mu_{j})=\bigg{\lbrace} \xi \in E_{\omega}:\ \| \xi\| = 1, \ &\Vert\varphi(k,\omega,\xi) \Vert< \exp \big{(}  k(\mu_j + \frac{1}{l}) \big{)}\ \text{and}\ \\ &d\big{(}\xi, F_{\mu_{i}}(\omega)\big{)}<\exp\big{(}k(\mu_{j}-\mu_{i-1})\big{)}, 1\leqslant i<j  \bigg{\rbrace}.
\end{align*}
We claim that
\begin{align}\label{mea_dis}
d\big{(}g(\omega), S_{F_{\mu_{j}}(\omega)}\big{)}=\lim_{k\rightarrow\infty}\liminf_{l\rightarrow\infty}d\big{(}g(\omega), B^{l,k}_{\omega}(\mu_{j})\big{)}.
\end{align}
Set the right side equal to $ A $. By definition, it is straightforward  to show that $ d\big{(}g(\omega), S_{F_{\mu_{j}}(\omega)}\big{)}\geqslant A$. For the opposite direction, let $ \epsilon>0 $. For large $ k,l $ we can find $ \xi^{l,k}\in B_{\omega}^{l,k}(\mu_{j}) $ such that $\Vert g(\omega)-\xi^{l,k}\Vert\leqslant A+\epsilon $. By  our assumptions on $ B^{k,l}_{\omega}(\mu_{j})$, we have a decomposition of the form
\begin{align*}
\xi^{l,k}=\sum_{1\leqslant i<j-1}h_{i}^{l,k}+h_{j-1}^{l,k}+f^{l,k}
\end{align*}
such that for $ 1\leqslant i<j $, $ h_{i}^{l,k}\in H_{i}(\omega)$ and $ f^{l,k}\in F_{\mu_{j}}(\omega) $. Moreover, there is a constant $ \tilde{M} $ such that for $ 1\leqslant i <j-1 $,  
\begin{align*}
\Vert h_{i}^{l,k}\Vert <\tilde{M}\, d\big{(}\xi^{l,k},F_{\mu_{i+1}}(\omega)\big{)}\ \ \ \text{and}\ \ \ \ \Vert f^{l,k}\Vert <\tilde{M}.
\end{align*}
From \eqref{eqn:lower_bound_subspace}, choosing $k$ larger if necessary, we obtain that for a given $\delta > 0$,
\begin{align*}
 \exp\big{(}k(\mu_{j-1}-\delta)\big{)}\Vert h^{l,k}_{j-1}\Vert \leqslant \Vert\varphi(k,\omega,h_{j-1}^{l,k})\Vert&\leqslant \Vert\varphi(k,\omega,\xi^{l,k})\Vert+\\&\sum_{1\leqslant i<j-1}\Vert\varphi(k,\omega,h_{i}^{l,k})\Vert+\tilde{M}\, \Vert\varphi(k,\omega,.)|_{F_{\mu_{j}}(\omega)}\Vert.
 \end{align*}
Consequently, from our assumptions on $ B^{l,k}_{\omega}(\mu_{j}) $ and \eqref{eqn:_bound}, we obtain for large $l,k$
 \begin{align*}
 \Vert h^{l,k}_{j-1}\Vert \leqslant \tilde{\tilde{M}}\, \exp\big{(}k(\mu_{j}-\mu_{j-1}+2\delta)\big{)}
 \end{align*}
 for a constant $ \tilde{\tilde{M}} $. Now for large $ l,k $,
\begin{align*}
\Vert \sum_{1\leqslant i<j}h_{i}^{l,k}\Vert<\epsilon\ , \ \ \ \ \ 1-\epsilon\leqslant\Vert f^{l,k}\Vert\leqslant 1+\epsilon.
\end{align*}
Consequently, $ d\big{(}g(\omega),S_{F_{\mu_{j}(\omega)}}(\omega)\big{)}\leqslant A $ and \eqref{mea_dis} is proved. The rest of the proof is straightforward: For $ \tilde{g}\in\Delta $ we set 
\begin{align*}
C^{l,k,j}(\tilde{g}):=\big{\lbrace} \omega\ :\ \frac{\tilde{g}(\omega)}{\Vert \tilde{g}(\omega)\Vert}\in B^{l,k}_{\omega}(\mu_{j})\big{\rbrace}
\end{align*}
From the definition of $B^{l,k}_{\omega}(\mu_j)$ and the induction hypothesis, $ C^{l,k,j}(\tilde{g}) $ is measurable for every $k,l \geq 1$. Note that
\begin{align*}
    d\big{(}g(\omega), S_{F_{\mu_{j}}(\omega)}\big{)} = \inf_{\tilde{g}\in\Delta_{0}} J_{\tilde{g}}(\omega)
\end{align*}
where 
\begin{equation}
J_{\tilde{g}}(\omega) = 
 \begin{cases}
       \infty &\quad \text{if } \omega\notin C^{l,k,j}(\tilde{g}) \\
       \Vert g(\omega)-\frac{\tilde{g}(\omega)}{\Vert \tilde{g}(\omega)\Vert}\Vert &\quad\text{otherwise.} \\ 
     \end{cases}
\end{equation}
Since $J_{\tilde{g}}(\omega)$ is measurable, this proves the claim. Therefore, we have also shown measurability of $C^{l,k,j}(g)$ for every $j,k,l \geq 1$ and $g \in \Delta$. Next, with the same argument as above, we can show that 
\begin{align*}
{\Vert}{(}\varphi(n,\omega,\cdot) \mid_{F_{\mu_j}(\omega)}&{\Vert} \chi_{\Omega_0}(\omega)=\lim _{l \rightarrow\infty}\liminf_{k \rightarrow\infty}\bigg{[}\sup_{\xi\in B_{\omega}^{l,k}(\mu_{j})} \Vert\varphi (n,\omega,\xi)\Vert\bigg{]}\chi_{\Omega_0}(\omega)
\end{align*}
for every $j  \geq 2$. Since
\begin{align*}
\sup_{\xi\in B_{\omega}^{l,k}(\mu_{j})} \Vert\varphi (n,\omega,\xi)\Vert=\sup_{g\in\Delta_{0}}\frac{\Vert\varphi(n,\omega,g(\omega))\Vert}{\Vert g(\omega)\Vert}\chi_{C^{l,k,j}(g)}(\omega),
\end{align*} 
measurability of \eqref{BCA} follows.
\end{proof}

The next lemma is a version of \cite[Lemma 3.7]{Blu16}. Unfortunately, there was a gap in proof which, however, was corrected in a subsequent erratum\footnote{Private communication with A.~Blumenthal.}. We present a full proof here, using the strategy of the above mentioned erratum. 
\begin{lemma}\label{Mea}
Let the same assumptions as in Lemma \ref{Z2} be satisfied. Then there exists a $ \theta$-invariant, measurable set $\Omega_1 \subset \Omega$ of full measure such that for every $\omega \in \Omega_{1}$, if $H_{\omega}$ is a complement of $F_{\mu_2}(\omega)$ in $E_{\omega}$, we have 
\begin{align}\label{PRo}
\lim _{n\rightarrow\infty}\frac{1}{n}\log \Vert\pi_{\varphi(n,\omega,H_{\omega})\parallel F_{\mu_{2}}(\theta^{n}\omega)}\Vert = 0.
\end{align}
\end{lemma}

\begin{proof}
It is enough to show that
\begin{align}\label{lim}
\limsup_{n\rightarrow\infty}\log \Vert\pi_{\varphi(n,\omega,H_{\omega})\parallel F_{\mu_{2}}(\theta^{n}\omega)}\Vert \leqslant 0.
\end{align}
Define
\begin{align*}
&\phi_{1}(\omega)=\sup_{p\geqslant 0} \exp\big{(}{-p(\mu_{1}+\delta)}\big{)}\Vert\varphi(p,\omega,\cdot)\Vert \\
&\phi_{2}(\omega)=\sup_{p\geqslant 0} \exp\big{(}{-p(\mu_{2}+\delta)}\big{)}\Vert\varphi(p,\omega,\cdot)|_{F_{\mu_{2}}(\omega)}\Vert
\end{align*}
From Lemma \ref{Z2}, $\phi_{1}$ and $\phi_{2} $ are measurable functions and bounded on a set of full measure $\Omega_0$. So from \cite[Lemma III.8]{Man83},
there exists a measurable subset $\Omega_1$ of full measure such that for any $\omega \in \Omega_{1}$,
 \begin{align}\label{ASD}
 \lim_{n\rightarrow\infty}\frac{1}{n}\log^+ \phi(\theta^{n} \omega)=0
 \end{align}
 where $ \phi(\omega) = \max\lbrace\phi_{1}(\omega),\phi_{2}(\omega)\rbrace$. Note that we can assume that $\Omega_{1}$ is also $ \theta$-invariant, otherwise we can replace it by $ \bigcap_{j \in \Z}(\theta^{j})^{-1}(\Omega_{1})$. Fix $\omega \in \Omega_1$ and assume that $ H_{\omega} \oplus F_{\mu_{2}}(\omega)= E_{\omega}$. Let $\epsilon > 0$. From (\ref{eqn:_bound}) and (\ref{eqn:lower_bound_subspace}), we can find an $N \in \N $ such that for $n\geqslant N$, 
 \begin{align}\label{AZ}
 \begin{split}
 \Vert\varphi(n,\omega,\cdot)\Vert\leqslant \exp\big{(} &n(\mu_{1}+\delta)\big{)}\ , \ \ \inf_{h \in H_{\omega}\setminus\lbrace 0\rbrace }\frac{\Vert\varphi(n,\omega,h)\Vert}{\Vert h \Vert}\geqslant \exp\big{(}n(\mu_{1}-\delta) \big{)} \\
 &\phi(\theta^{n}\omega)\leqslant \exp(n\epsilon).
 \end{split}
 \end{align}
 We prove \eqref{lim} by contradiction. Assume there is a ${\gamma} >0$ and a sequence $\big{(}n_{k},h_k,f_k \big{)} \in \big{(}\mathbb{N},H_{\omega},F_{\mu_{2}}(\theta^{n_{k}}\omega)\big{)}$ such that
\begin{align}\label{ZA}
\begin{split}
n_{k} \rightarrow \infty \ , \ \Vert\ h_k \Vert =1 \ \text{and} \ \ \frac{\Vert\varphi(n_{k},\omega,h_k)\Vert}{\Vert\varphi(n_{k},\omega,h_k) -f_k \Vert}\geqslant\frac{1}{2}\exp(n_{k}{\gamma}) \text{ for all } k\geq 1.
\end{split}
\end{align}
For $p \geq 0$,
\begin{align}\label{BBN}
\begin{split}
\Vert\varphi(n_{k}+p,\omega,h_k)\Vert &=\Vert\varphi(p,\theta^{n_{k}} \omega, \varphi(n_{k},\omega,h_k)) \Vert\\
&\leqslant \Vert\varphi(p,\theta^{n_{k}}\omega,\cdot)\Vert \Vert\varphi(n_{k},\omega,h_k) - f_k \Vert + \Vert\varphi(p,\theta^{n_{k}}\omega,\cdot)|_{F_{\theta^{n_{k}}\omega}}\Vert \Vert f_k \Vert
\end{split}
\end{align}
From \eqref{ZA}, it follows that $\Vert f_k \Vert\leqslant 3\Vert\varphi(n_{k},\omega ,h_k) \Vert $. Now for large $ n_{k} $, from (\ref{AZ}) and (\ref{BBN}),
\begin{align*}
\exp\big{(}(n_{k}+p)(\mu_{1}-\delta)\big{)}& \leqslant 2\exp\bigg{(}n_{k}\epsilon +p(\mu_{1}+\delta)+n_{k}(\mu_{1}+\delta)-n_{k} {\gamma}\bigg{)}\\
&+3\exp\bigg{(}p(\mu_{2}+\delta)+n_{k}\epsilon+n_{k}(\mu_{1}+\delta)\bigg{)}.
\end{align*}
Choosing $ p=n_{k}$ and $ \delta ,\epsilon $ small, we will have a contradiction. 
\end{proof}

The following definition is taken from \cite{GTQ15}.
\begin{definition}\label{def:volume}
Let $ X ,Y $ be  Banach spaces. For $ x_{1},...,x_{k}\in X $, we define 
\begin{align}\label{VOL}
  \operatorname{Vol}(x_{1},x_{2},...,x_{k}):=\Vert x_{1}\Vert\prod_{i=2}^{k}d(x_{i},\langle x_{j} \rangle_{1\leqslant j<i})
 \end{align}
 where $d$ denotes the usual distance between a point and a subset in $X$. For a given bounded linear function $ T : X\rightarrow Y $ and $k \geq 1$, set 
\begin{align*}
D_{k}(T):=\sup_{\Vert x_{i}\Vert =1 ; i=1,...,k} \operatorname{Vol} \big{(}T(x_{1}),T(x_{2}),...,T(x_{k})\big{)}
\end{align*}
\end{definition}

We summarize some basic properties of $D_k$ in the next lemma.
\begin{lemma}\label{lemma:prop_Dk}
 Let $X,Y,Z$ be Banach spaces and $ T:X\rightarrow Y $,  $S:Y\rightarrow Z $ bounded linear maps.
 \begin{itemize}
  \item[(i)] $D_{1}(T)=\Vert T\Vert$ and $D_{k}(T)\leqslant\Vert T\Vert^{k}$ for $k \geq 1$.
  \item[(ii)] $D_{k}(S\circ T)\leqslant D_{k}(S)D_{k}(T)$ for $k \geq 1$.
 \end{itemize}

\end{lemma}

\begin{proof}
 The proof of (i) is straightforward, (ii) is proven in \cite[Lemma 1]{GTQ15}.
\end{proof}
%
\begin{lemma}\label{estimate2}
Let $T :X\rightarrow Y $ be a bounded linear map between two Banach spaces, $ x\in \langle x_{i} \rangle_{1\leqslant i\leqslant k}$ and  $ \Vert x_{i}\Vert =1$. Then there exists a constant $\alpha_{k}$ which only depends on $k$ such that
\begin{align*}
  \operatorname{Vol}\big{(}T(x_{1}),T(x_{2}),...,T(x_{k})\big{)}\leqslant \alpha_{k}\Vert T\Vert ^{k-1}\frac{\Vert Tx\Vert}{\Vert x\Vert}
\end{align*}
\end{lemma}

\begin{proof}
  Assume $ \frac{x}{\Vert x\Vert}=\sum_{1\leqslant j\leqslant k}\beta_{j}x_{j}$. Consequently, there exists $ 1\leqslant t\leqslant k $ such that $ \beta_{t}\geqslant\frac{1}{k}$. Define $y = (y_1,\ldots,y_k)$ as
\begin{align*}
  y_i =
\begin{cases} 
x_{i} & \text{for } i\neq t,n, \\ 
x_{n} & \text{for } i = t, \\
x_{t} & \text{for } i = n. 
\end{cases} 
 \end{align*}
By definition,
\begin{align}\label{Vol2}
\begin{split}
  \operatorname{Vol}\big{(}T(y_{1}),T(y_{2}),...,T(y_{n})\big{)} &\leqslant\Vert T\Vert ^{k-1}d \big{(}T(y_{n}),\langle T(y_{i}) \rangle_{1\leqslant i\leqslant n-1}\big{)}\\ &\leqslant k\Vert T\Vert^{k-1}\frac{\Vert Tx\Vert}{\Vert x\Vert}.
\end{split}
\end{align}
From \cite[Proposition 2.14]{Blu16}, there is an inner product $(\cdot,\cdot)_{V}$ on $ V = \langle T(x_{i}) \rangle_{1\leqslant i\leqslant k} $ such that 
\begin{align*}
\frac{1}{\sqrt{k}}\leqslant\frac{\Vert T(x)\Vert_{V}}{\Vert T(x)\Vert} \leqslant\sqrt{k}  \ \ \ \ \ \ \ \ \ \forall x\in \langle x_{i} \rangle_{1\leqslant i\leqslant k}.
\end{align*}
It is not hard to see that this implies that
\begin{align*}
\sqrt{k}\leqslant\frac{d_{V}\big{(}T(x_{j}), \langle T(x_{i}) \rangle_{1\leqslant i<j}\big{)}}{d\big{(}T(x_{j}), \langle T(x_{i}) \rangle_{1\leqslant i<j}\big{)}}\leqslant\sqrt{k}
\end{align*}
and, consequently, 
\begin{align}\label{VOL3}
(\frac{1}{\sqrt{k}})^{k}\leqslant\frac{\operatorname{Vol}_{V}\big{(}T(x_{1}),...,T(x_{k})\big{)}
}{\operatorname{Vol}\big{(}T(x_{1}),...,T(x_{k})\big{)}}\leqslant (\sqrt{k})^{k}.
\end{align}
Note that $\operatorname{Vol}_{V}\big{(}T(x_{1}),...,T(x_{k})\big{)} =\operatorname{Vol}_V \big{(}T(y_{1}),T(y_{2}),...,T(y_{k})\big{)} $ so our claim follows from  (\ref{Vol2}) and (\ref{VOL3}).
\end{proof}

\begin{lemma}\label{restr}
Assume that $ X , Y $  are Banach spaces and that $ T: X\rightarrow Y $ is a linear map. Let $ V\subset X $ be a closed subspace of codimension $m$. Then for $ k>m $, there exists a constant $ C $ which only depends on $k$ and $m$ such that
\begin{align}
D_{k}(T)\leqslant C D_{m}(T)D_{k-m}(T|_{V})
\end{align}
\end{lemma}
\begin{proof}
  \cite[Lemma 8]{GTQ15}.
\end{proof}

\begin{proposition}\label{KING}
 Let $\varphi$ be a bounded linear cocycle acting on a measurable field of Banach spaces $(\{E_{\omega}\}_{\omega \in \Omega},\Delta,\Delta_0)$. Then for every $n, k\geqslant 1$, the map 
 \begin{align*}
  \Psi_n^k \colon  \Omega &\to \R \\
  \omega &\mapsto D_k(\varphi(n,\omega,\cdot))
 \end{align*}
is measurable.
\end{proposition}

\begin{proof}
  For $k = 1$, the claim follows from Lemma \ref{lemma:prop_Dk} and Lemma \ref{lemma:oper_meas}. Note that for $\omega \in \Omega$,
  \begin{align*}
   &\Psi_n^k(\omega) = \sup_{g_1,\ldots, g_k \in \Delta_0} \operatorname{Vol} ( \varphi(n,\omega,\tilde{g}_1(\omega)), \ldots, \varphi(n,\omega,\tilde{g}_k(\omega))) \chi_{\{\|g_1\| > 0,\ldots, \|g_k\| > 0\}}(\omega)
  \end{align*}
  where we used the notation $\tilde{g}_i(\omega) = g_i(\omega)/\| g_i(\omega) \|$, $i = 1,\ldots,k$. It is therefore sufficient to prove that for fixed $g_1,\ldots,g_k \in \Delta$,
  \begin{align*}
    \omega \mapsto  \operatorname{Vol} ( \varphi(n,\omega, \tilde{g}_1(\omega)), \ldots, \varphi(n,\omega,\tilde{g}_k(\omega))) \chi_{\{\|g_1\| > 0,\ldots, \|g_k\| > 0\}}(\omega)
  \end{align*}
  is measurable. For $ i\geqslant 2$, we have
\begin{align*}
d\bigg{(}\varphi &\big{(}n,\omega,\tilde{g}_i(\omega)\big{)},\langle \varphi\big{(}n,\omega, \tilde{g}_t(\omega)\big{)} \rangle_{1\leqslant t<i}\bigg{)} \\
= &\inf_{q_{1},...,q_{i-1}\in \mathbb{Q}}\bigg{\Vert} \varphi\big{(}n,\omega, \tilde{g}_i(\omega) \big{)} - \Sigma_{1\leqslant t< i} q_{t} \varphi \big{(}n,\omega,\tilde{g}_t(\omega) \big{)} \bigg{\Vert} \\
= &\frac{1}{\|g_i(\omega)\|} \inf_{q_{1},...,q_{i-1}\in \mathbb{Q}}\bigg{\Vert} \varphi\big{(}n,\omega, {g}_i(\omega) \big{)} - \Sigma_{1\leqslant t< i} q_{t} \varphi \big{(}n,\omega,{g}_t(\omega) \big{)} \bigg{\Vert} \\
= &\frac{1}{\|g_i(\omega)\|} \inf_{q_{1},...,q_{i-1}\in \mathbb{Q}}\bigg{\Vert} \varphi \left( n,\omega, {g}_i(\omega) - \Sigma_{1\leqslant t< i} q_{t} {g}_t(\omega) \right) \bigg{\Vert}.
\end{align*}
The claim follows by definition of $\operatorname{Vol}$.

\end{proof}
\begin{lemma}\label{two}
Under the same setting as in Proposition \ref{KING}, let
 $\chi_{n}^{k}(\omega)=\log(\Psi_{n}^{k}(\omega))$. Assume that
 \begin{align*}
  \log^+ \Vert\varphi(1,\omega,\cdot)\Vert \in L^{1}(\Omega).
 \end{align*}
  Then there exists a measurable forward invariant set $\Omega_{1} \subset \Omega$ of full measure such that the limit
\begin{align}
  \Lambda_k(\omega) := \lim_{n\rightarrow\infty}\frac{\chi_{n}^{k}(\omega)}{n}\in [-\infty ,\infty)
\end{align}
  exists for every $\omega \in \Omega_1$ and $k \geq 1$. Furthermore, $\Lambda_k(\theta \omega) = \Lambda_k(\omega)$ for every $k \geq 1$, $\omega \in \Omega_1$ and $\Lambda_k(\omega)$ is constant on $\Omega_1$ in case the underlying metric dynamical system is ergodic. 
\end{lemma}

\begin{proof}
From Lemma \ref{lemma:prop_Dk} and the cocycle property,
\begin{align}
\chi_{n+m}^{k}(\omega)\leqslant\chi_{n}^{k}(\theta^{m}\omega)+\chi_{m}^{k}(\omega).
\end{align}
 By assumption and Lemma \ref{lemma:prop_Dk}, it follows that $ \chi_{1}^{k;+}\in L^{1}(\Omega) $. Therefore, we can directly apply Kingman's Subadditive Ergodic Theorem \cite[3.3.2 Theorem]{Arn98} to conclude.
\end{proof}

\begin{remark}
\begin{itemize}
 \item[(i)] From Birkhoff's Ergodic Theorem, we can furthermore assume that 
\begin{align}\label{limit}
\lim_{n\rightarrow\infty}\frac{ \log^{+}  \Vert\varphi(1,\theta^{n}\omega,\cdot)\Vert }{n}=0
\end{align}
for all $\omega \in \Omega_1$.
  \item[(ii)] From Lemma \ref{restr}, it follows that
  \begin{align*}
   \Lambda_k \le \Lambda_m + \Lambda_{k-m}
  \end{align*}
  for every $k > m$. In particular, if $\Lambda_m = -\infty$, it follows that $\Lambda_k = -\infty$ for every $k > m$.

\end{itemize}
\end{remark}

\begin{definition}
 If the assumptions of Lemma \ref{two} are satisfied, we define
 \begin{align*}
  \lambda_k(\omega) := \begin{cases}
                        \Lambda_k(\omega) - \Lambda_{k-1}(\omega) &\text{if } \Lambda_k(\omega), \Lambda_k(\omega) \in \R \\
                        -\infty &\text{if } \Lambda_k(\omega) = -\infty
                       \end{cases}
 \end{align*}
 for $k \geq 1$, where we set $\Lambda_0(\omega) := 0$. We call $\lambda_k$ the \emph{$k$-th Lyapunov exponent} of $\varphi$. Note that they are deterministic almost surely in case the underlying system is ergodic. 

\end{definition}

\begin{remark}
 Following the same strategy as in \cite[Theorem 13]{GTQ15}, one can show that $(\lambda_k)_{k \geq 1}$ is a decreasing sequence.
\end{remark}

%

The next lemma shows that the sequence $(\lambda_k)$ does not have real cluster points in case the cocycle is compact.     
\begin{lemma}\label{three}
  Let $\varphi$ be as in Lemma \ref{two}. Furthermore, assume that it is compact. Then there is a measurable forward invariant subset $\tilde{\Omega} \subset \Omega$ with full measure such that for any $\omega \in \tilde{\Omega}$ and $\rho\in\mathbb{R}$, there are only finitely many exponents $ \lambda_{k}(\omega)$ that exceed $\rho$.
\end{lemma}

\begin{proof}
Let $\Omega_1$ be the set provided in Lemma \ref{two}. For $\omega \in \Omega$, let $B_{\omega}$ be the unit ball in $E_{\omega}$. Set
\begin{align}
G(\vartheta,\nu) := \bigg{\lbrace}\omega\in \Omega_{1}\, :\, \varphi(1,\omega,B_{\omega}) \text{ can be covered by  }   e^{\vartheta}  \text{  balls with sizes less than }  e^{\nu}\bigg{\rbrace}.
\end{align}
We claim that $G(\vartheta,\nu)$ is a measurable subset. To see this, define
\begin{align*}
 S(\omega) := \left\{ s \in B_{\omega}\, :\, s = r \frac{g(\omega)}{\|g(\omega)\|} \chi_{\{ \|g\| > 0\}}(\omega), \, g \in \Delta_0,\, r \in \Q \cap [0,1] \right\}.
\end{align*}
One can easily check that $S(\omega)$ is dense in $B_{\omega}$. 
Let $ p = e^{\vartheta} $ and define
\begin{align*}
H(\omega)=\inf_{s_1,\ldots,s_p \in S(\omega)} \bigg{(}\sup_{s \in S(\omega)} \min_{1\leqslant i\leqslant p} \big{(}\Vert\varphi(1,\omega,s) - \varphi(1,\omega,s_i) \Vert\big{)} \bigg{)}.
\end{align*}
It is not hard to see that
\begin{align*}
G(\vartheta ,\nu)=\big{\lbrace}\omega \in \Omega_1\, :\, \ H(\omega)< e^{\nu}\big{\rbrace}
\end{align*}
and consequently $G(\vartheta ,\nu)$ is indeed measurable. Since $\varphi$ is compact, for any $\nu\in\mathbb{R}$,
\begin{align*}
\lim_{\vartheta\rightarrow\infty} \P\big{(}G(\vartheta,\nu)\big{)} = 1. 
\end{align*}
Let $\omega \in \Omega_1$. With the same argument as on \cite[page 247]{GTQ15}, we can say that $\varphi(m,\omega,B_{\omega})$ can be covered by $N_{m}=e^{m\vartheta}$ balls of size $R_{m}^{\vartheta,\nu}=e^{m\gamma _{m}^{\vartheta,\nu}}$ where
\begin{align*}
\gamma _{m}^{\vartheta,\nu}(\omega)=\frac{1}{m}\bigg{[}\nu\sum_{0\leqslant j\leqslant m}\chi_{G(\vartheta,\nu)}(\theta^{j}\omega ) + \sum_{0\leqslant j\leqslant m}\chi_{G(\vartheta,\nu)^{c}}\log^{+} \Vert\varphi(1,\theta^{j}\omega,\cdot)\Vert \bigg{]} =: \nu A_{m}^{\vartheta,\nu}(\omega) +B_{m}^{\vartheta,\nu}(\omega).
\end{align*}
Let $\lambda_k(\omega) > \rho$. For large $m$, 
we must have $ k(\rho -\gamma _{m}^{\vartheta,\nu})\leqslant \vartheta$.
 If we can show that $\rho -\gamma _{m}^{\vartheta,\nu} > 0$ for some $m,\vartheta,\mu$, the proof is finished since in that case, $ k<\frac{\vartheta}{\rho -\gamma _{m}^{\vartheta,\nu}}$.\\
 Let $ \epsilon >0 $ and choose $ \nu< 0 $ such that $\nu<\frac{\rho-\epsilon}{\epsilon}$. From integrability of $\log^{+} \Vert\varphi(1,\omega,\cdot)\Vert$, there exists a $\delta >0$ such that for $\P(E) < \delta$,
 \begin{align}\label{Asghar}
  \int_{E}\log^{+} \Vert(\varphi(1,\omega,\cdot)\Vert \, d\P \leqslant \epsilon^{2}.
 \end{align}
Now we choose $ \vartheta>0 $ such that
\begin{align}\label{akbar}
  \P\big{(}G(\vartheta,\nu)^{c}\big{)}\leqslant \epsilon \wedge \delta.
\end{align}
 Since $ 0\leqslant A_{m}^{\vartheta,\nu}(\omega)\leqslant 1 $,
\begin{align*}
\int _{\Omega}A_{m}^{\vartheta,\nu}\, d\P &\leqslant \P(A_{m}^{\nu ,r}>\epsilon) + \epsilon\quad \text{and} \\
\P(B_{m}^{\vartheta,\nu}>\epsilon) &\leqslant \frac{1}{\epsilon} \int_{\Omega} B_{m}^{\vartheta,\nu}\, d\P.
\end{align*}
Now from (\ref{Asghar}), (\ref{akbar}) and Birkhoff's Ergodic theorem, for large $ m $,
\begin{align*}
  \P(A_{m}^{\vartheta,\nu}>\epsilon)\geqslant 1-3\epsilon \ \ \  \text{and}\ \ \ \P(B_{m}^{\vartheta,\nu}>\epsilon) \leqslant 2\epsilon.
\end{align*}
Set $ A_{1} := \lbrace A_{m}^{\vartheta,\nu}>\epsilon\rbrace $ and $B_{1} := \lbrace B_{m}^{\vartheta,\nu}\leqslant\epsilon\rbrace $ and note that $ \P(A_{1}\cap B_{1})\geqslant 1-5\epsilon $. For $ \omega\in  A_{1}\cap B_{1}$,
\begin{align*}
\gamma_{m}^{\vartheta,\nu} < \rho.
\end{align*}
Since $\epsilon$ is arbitrary, we can find a set $\Omega_{2} \subset \Omega_1$ of full measure with the desired property. Finally we put $\Omega_{3} := \bigcap_{j=0}^{\infty}(\theta^{j})^{-1}{\Omega_{2}}$.
\end{proof}

The following proposition, a trajectory-wise version of the Multiplicative Ergodic Theorem, will play a central role in the proof of our main result. It is a slight reformulation of \cite[Proposition 3.4]{Blu16}. The proof is very similar to Blumenthal's original proof, but because of its importance, we decided to sketch it in the appendix, cf. page \pageref{proof_pathwise_MET}.

\begin{proposition}\label{imp}
  Let $ \lbrace V_{j}\rbrace_{j\geqslant 0} $ be a sequence of Banach spaces and  $ T_{i}:V_{i}\rightarrow V_{i+1} $ a sequence of bounded linear operators. Set $T^{n}=T_{n-1}\circ ...\circ T_{0}$. Assume that: 

\begin{itemize}
  \item[(i)] $\limsup_{n\rightarrow\infty}\frac{1}{n}\log^{+} \Vert T_{n} \Vert =0$.
  \item[(ii)] For any $ k \geq 1 $, the following limits exists:
\begin{align*}
  L_{k}=\lim_{n\rightarrow\infty}\frac{1}{n}\log D_{k}(T^{n}).
\end{align*}
\item[(iii)] Setting $L_{0} := 0$ and $ l_{k} := L_{k} - L_{k-1}$ for $k \geq 1$, assume that there is a number $m < \infty $ for which $\overline{l} := l_{1} = \ldots = l_{m} > l_{m+1} =: \underline{l}$.
 \end{itemize}
 Then the subspace
 \begin{align*}
    F:=\big{\lbrace}v\in V_{0} : \limsup_{n\rightarrow\infty} \frac{1}{n}\log\Vert T^{n}v\Vert\leqslant\underline{l}\big{\rbrace}
 \end{align*}
 is closed and $m$-codimensional. Also, for $ v\in V_{0}\setminus F $,
\begin{align}\label{eqn:conv_overlinel}
  \lim_{n\rightarrow\infty} \frac{1}{n} \log \Vert T^{n}v \Vert = \overline{l}.
\end{align}
  Furthermore, for any complement $ H $ of $ F $,
\begin{align}\label{uniform}
  \lim_{n\rightarrow\infty} \frac{1}{n} \log\inf_{v\in H \setminus \lbrace 0\rbrace} \frac{ \Vert T^{n}v \Vert}{\Vert v\Vert}= \overline{l}.
\end{align}
  Finally, if $h_1,\ldots, h_m \in V_0$ are linearly independent and $H = \langle h_{1},...,h_{m}\rangle$,
\begin{align}\label{det}
  \lim_{n\rightarrow\infty} \frac{1}{n} \log \operatorname{Vol}{(}T^{n}h_{1},T^{n}h_{2},...,T^{n}h_{m}{)}  = m\overline{l}.
\end{align}

\end{proposition}

%

\begin{remark}\label{remark:after_prop_key}
In the proof of the proposition above, we will also see that
\begin{align}\label{uniform2}
\limsup_{n\rightarrow \infty}\frac{1}{n}\log \Vert T^{n}|_{F} \Vert \leqslant\underline{l}
\end{align}
holds.
\end{remark}

We finally state the main result of this section, a Multiplicative Ergodic Theorem for cocycles acting on measurable fields of Banach spaces.
\begin{theorem}\label{thm:MET_Banach_fields}
Let $(\Omega,\mathcal{F},\mathbb{P},\theta)$ be an ergodic measurable metric dynamical system and $\varphi$ be a compact linear cocycle acting on a measurable field of Banach spaces $\{E_{\omega}\}_{\omega \in \Omega}$ in the sense of Definition \ref{def:RDS_on_Banach_field}. For $\lambda \in \R \cup \{-\infty\}$ and $\omega \in {\Omega}$, define
\begin{align*}
 F_{\lambda}(\omega) := \big{\lbrace} x\in E_{\omega}\, :\, \limsup_{n\rightarrow\infty} \frac{1}{n} \log \Vert\varphi(n,\omega,x) \Vert \leqslant\lambda \big{\rbrace}. 
\end{align*}
Assume that
\begin{align*}
 \log^+ \Vert\varphi(1,\omega,\cdot)\Vert \in L^{1}(\Omega).
\end{align*}
Then there is a measurable forward invariant set $\tilde{\Omega} \subset \Omega$ of full measure such that:
 \begin{itemize}
 \item[(i)] For any $\omega\in\tilde{\Omega}$ and $ k\geqslant 1 $, the limit
\begin{align}
 \Lambda_k := \lim_{n\rightarrow\infty}\frac{1}{n} \log D_{k}({\varphi(n,\omega,\cdot)}) \in [-\infty,\infty)
\end{align}
exists and is independent of $\omega$.
\item[(ii)] Setting $\Lambda_0 := 0$ and $ \lambda_{k} := \Lambda_{k} -\Lambda_{k-1}$ with $\lambda_k = -\infty$ if $\Lambda_k = -\infty$, the sequence $(\lambda_k)$ is decreasing. If the number of distinct values of this sequence is infinite, then \mbox{$\lim_{k\rightarrow\infty}\lambda_{k} = -\infty$}. We denote the decreasing subsequence of distinct values by $ (\mu_{j})_{j\geqslant 1}$, which can be a finite or an infinite sequence, and $m_j$ will denote the multiplicity of $\mu_j$ in the sequence $(\lambda_j)$. If $\mu_j \in \R$, $m_j$ is finite.
\item[(iii)] For $\lambda_{i} > \lambda_{i+1}$ and $\omega \in \tilde{\Omega}$,
\begin{align}
x\in F_{\lambda_{i}}(\omega)\setminus F_{\lambda_{i+1}}(\omega) \quad \text{if and only if}\quad \lim_{n\rightarrow\infty}\frac{1}{n}\log \Vert\varphi(n,\omega,x ) \| = \lambda_{i}.
\end{align}
\item[(iv)] For any $\mu_j$, $\operatorname{codim} F_{\mu_j}(\omega) = m_1 + \ldots + m_{j-1}$ for every $\omega \in \tilde{\Omega}$.
\item[(v)] For $ \omega\in\tilde{\Omega}$, if $h^1,\ldots,h^k \in E_{\omega}$ are linearly independent and $H_{\omega} = \langle h^{1},...,h^{k} \rangle$ is a complement subspace for $ F_{\mu_{j}}(\omega) $ in $E_{\omega}$, then
\begin{align}\label{eqn:vol_growth_MET}
\lim_{n\rightarrow\infty}\frac{1}{n} \log \operatorname{Vol} \big{(}\varphi(n,\omega,h^{1}),...,\varphi(n,\omega,h^{k})\big{)} = \sum_{1\leqslant i\leqslant j} m_{i} \mu_{i}.
\end{align}
 \end{itemize}
\end{theorem}

\begin{remark}
 The sequence $(\mu_j)$ is called the \emph{Lyapunov spectrum}, the filtration of spaces
 \begin{align*}
  F_{\mu_1}(\omega) \supset F_{\mu_2}(\omega) \supset \cdots
 \end{align*}
 is called \emph{Oseledets filtration}.
\end{remark}

\begin{proof}
Note that (i) and (ii) are direct consequences of Lemma \ref{two} and Lemma \ref{three}, hence we only have to prove (iii), (iv) and (v). The idea
is to prove the consecutive statements for each Lyapunov exponent by induction, where Proposition \ref{imp} will play a central role. 
We will only give the proof in case that the Lyapunov spectrum is infinite, the case of a finite Lyapunov spectrum is similar.

Let us start to formulate a result for the first Lyapunov exponent $\mu_1$. Consider $\Omega _{1} \subset \Omega $ as in Lemma \ref{two}.
We may assume that \eqref{limit} is also satisfied for every $\omega \in \Omega_1$. Fix some $\omega \in \Omega_1$ and define $ V_{j} := E_{\theta^{j}\omega}$ and $ T_{j} := \varphi(1, \theta^{j}\omega,\cdot)$. Note that, by definition, $\mu_{1}=\lambda_{1}=...=\lambda _{m_1} > \lambda _{m_1 + 1} = \mu_{2} $ and $\mu_1 = \Lambda_1$, therefore $F_{\mu_1}(\omega) = E_{\omega} = V_0$. Proposition \ref{imp} now implies that for $x \in F_{\mu_1}(\omega) \setminus F_{\mu_2}(\omega)$, we have $\lim_{n \to \infty} \frac{1}{n} \log \| \varphi(n,\omega,x) \| = \mu_1$ and that $F_{\mu_2}(\omega)$ is $m_1$-codimensional. Furthermore, if $H_{\omega} = \langle h^1,\ldots,h^{m_1} \rangle$ is a complement for $F_{\mu_2}(\omega)$, 
\begin{align}\label{eqn:vol_growth_MET_mu2}
\lim_{n\rightarrow\infty}\frac{1}{n} \log \operatorname{Vol} \big{(}\varphi(n,\omega,h^{1}),...,\varphi(n,\omega,h^{k})\big{)} = m_{1} \mu_{1}.
\end{align}\eqref{eqn:vol_growth_MET}

For the next step, we set $ V_{j}:=F_{\mu_{2}}(\theta^{j}\omega)$ and $ T_{j}:=\varphi (1, \theta^{j}\omega,\cdot )\mid_{F_{\mu_{2}}(\theta^{j}\omega)} $. Note that from the cocycle property, $ T_{j} \colon V_{j} \to V_{j+1}$. We claim that there is a measurable and $ \theta $-invariant subset $\Omega_{2}\subset\Omega_{1}$ with full measure such that for any $ \omega\in\Omega_{2} $ and $ k\geqslant 1 $,
\begin{align}\label{LLL}
\lim_{n\rightarrow\infty}\frac{1}{n}\log D_{k}\big{[}\varphi(n,\omega,\cdot)\mid_{F_{\mu_{2}}(\omega)}\big{]}=\Lambda_{k+m}-\Lambda_{m}
\end{align}
where we set $m := m_1$ for simplicity. Let $ \Omega_{2}\subset\Omega_{1} $ be a measurable subset with the properties stated in Lemma \ref{Mea}. Fix some $ \omega\in\Omega_{2} $. As a consequence of Lemma \ref{restr},
\begin{align}\label{liminf}
\Lambda_{k+m} \leqslant \Lambda_{m} + \liminf_{n\rightarrow\infty} \frac{1}{n} \log D_{k}\big{[}\varphi(n,\omega,\cdot)\mid_{F_{\mu_{2}}(\omega)}\big{]}.
\end{align}
For $n \in \N$ to be specified later, let $\lbrace f^{i} \rbrace_{1\leqslant i\leqslant k}\subset F_{\mu_{2}}(\omega) $ be chosen such that $ \Vert f^{i} \Vert = 1 $ for every $i$ and
\begin{align}\label{CX}
  \operatorname{Vol} \bigg{(}\varphi(n,\omega,f^{1}),...,\varphi(n,\omega,f^{k}) \bigg{)}\geqslant \frac{1}{2}D_{k}\big{[}\varphi(n,\omega,\cdot)\mid_{F_{\mu_{2}}(\omega)}\big{]}.
\end{align}
 Let $H_{\omega} = \langle h^{1} ,h^{2},...,h^{m} \rangle $ be a complement subspace for $ F_{\mu_{2}}(\omega) $. We can assume that $\Vert h^{i} \Vert =1 $ for all $i$. To ease notation, set $ \varphi_{\omega}^{n}(\cdot) := \varphi(n,\omega,\cdot)$. By definition,
\begin{align}\label{CY}
\begin{split}
  &D_{k+m} (\varphi^{n}_{\omega}(\cdot)) \geqslant \operatorname{Vol} \big{(}\varphi^{n}_{\omega}(h^{1}),...,\varphi^{n}_{\omega}(h^{m}),\varphi^{n}_{\omega}(f^{1}),...,\varphi^{n}_{\omega}(f^{k})\big{)} \\ 
  &= \operatorname{Vol} \big{(}\varphi^{n}_{\omega}(h^{1}),...,\varphi^{n}_{\omega}(h^{m})\big{)} \prod_{j=1}^{m} d\bigg{(}\varphi^{n}_{\omega}(f^{j}_{\omega}),\big{<}\varphi^{n}_{\omega}(h^{1}),...,\varphi^{n}_{\omega}(h^{m}), \varphi^{n}_{\omega}(f^{1}),...,\varphi^{n}_{\omega}(f^{j-1})\big{>}\bigg{)}.
\end{split}
\end{align} 
It is not hard to see that
\begin{align*}
\frac{d\bigg{(}\varphi^{n}_{\omega}(f^{j}),\big{<}\varphi^{n}_{\omega}(f^{1}), ... ,\varphi^{n}_{\omega}(f^{j-1})\big{>}\bigg{)}}{d\bigg{(}\varphi^{n}_{\omega}(f^{j}),\big{<}\varphi^{n}_{\omega}(h^{1}),...,\varphi^{n}_{\omega}(h^{m}),\varphi^{n}_{\omega}(f^{1}),...,\varphi^{n}_{\omega}(f^{j-1})\big{>}\bigg{)}} \leqslant \| \Pi_{F_{\mu_{2}}(\theta^{n}\omega)||\varphi(n,\omega,H_{\omega})} \|.
\end{align*}
Consequently, by (\ref{CX}) and (\ref{CY}),
\begin{align*}
D_{k+m}(\varphi^{n}_{\omega}(\cdot)) &\geqslant \| \Pi_{F_{\mu_2}(\theta^{n}\omega)||\varphi(n,\omega,H_{\omega})} \|^{-m} \operatorname{Vol} \big{(}\varphi^{n}_{\omega}(h^{1}),...,\varphi^n_{\omega} (h^{m}) \big{)} \operatorname{Vol} (\varphi^{n}_{\omega}(f^{1}),...,\varphi^{n}_{\omega}(f^{n})) \\
&\geqslant \frac{1}{2} \| \Pi_{F_{\mu_{2}}(\theta^{n}\omega)||\varphi(n,\omega,H_{\omega})} \|^{-m} \operatorname{Vol} \big{(}\varphi^{n}_{\omega}(h^{1}),...,\varphi^n_{\omega} (h^{m}) \big{)} D_{k}\big{[}\varphi(n,\omega,\cdot)\mid_{F_{\mu_{2}}(\omega)}\big{]}.
\end{align*}
Note that, by definition of the projection operator,
\begin{align*}
   1\leqslant \| \Pi_{F_{\mu_{2}}(\theta^{n}\omega)||\varphi(n,\omega,H_{\omega})} \| \leqslant \| \Pi_{\varphi(n,\omega,H_{\omega})||F_{\mu_{2}}(\theta^{n}\omega)}\| + 1.
\end{align*}
 Choosing $n$  large, using \eqref{eqn:vol_growth_MET_mu2} and Lemma \ref{Mea}, we see that
\begin{align}\label{lisup}
\limsup_{n\rightarrow\infty}\frac{1}{n} \log D_{k}\big{[}\varphi(n,\omega,\cdot)\mid_{F_{\mu_{2}}(\omega)}\big{]}+\Lambda_{m}\leqslant\Lambda_{k+m}
\end{align}
and \eqref{LLL} is shown. We can now use Proposition \ref{imp} again with $\overline{l} = \mu_2$, $\underline{l} = \mu_3$ and $m = m_2$ which proves that for $\omega \in \Omega_2$ and $x \in F_{\mu_2}(\omega) \setminus F_{\mu_3}(\omega)$,
\begin{align*}
 \lim_{n \to \infty} \frac{1}{n} \log \| \varphi(n,\omega,x) \| = \mu_2.
\end{align*}
Moreover, $F_{\mu_3}(\omega)$ is $m_2$-codimensional in $F_{\mu_2}(\omega)$. Using that $F_{\mu_2}(\omega)$ is $m_1$-codimensional in $E_{\omega}$ implies that $F_{\mu_3}(\omega)$ has codimension $m_1 + m_2$ in $E_{\omega}$.

It remains to prove (v). Let $ \langle h^{1},...,h^{m_{1}+m_{2}} \rangle$ be a complement subspace for $ F_{\mu_{3}}(\omega)$. Note that $ \operatorname{Vol}\big{(}\varphi^{n}_{\omega}(h^{1}),...,\varphi^{n}_{\omega}(h^{m_{1}+m_{2}})\big{)} $ is not invariant under permutation, but all permutations are equivalent up to a constant which only depends on $ m_{1}+m_{2} $, cf. the proof of Lemma \ref{estimate2}. We may assume that $ H_{\omega} = \langle h^{1} ,..., h^{m_{1}} \rangle$ is a complement subspace for $ F_{\mu_{2}}(\omega)$ and that for $ m_{1}+1\leqslant j\leqslant m_{1}+m_{2}$, we have $
h^{j} = g^{j-m_{1}} + f^{j-m_{1}}$ where $g^{j-m_{1}} \in F_{\mu_{2}}(\omega)$ and $f^{j-m_{1}} \in H_{\omega}$. It is not hard to see that $G_{\omega} := \langle g^{1},...,g^{m_{2}} \rangle$ is a complement subspace for $ F_{\mu_{3}}(\omega) $ in $ F_{\mu_{2}}(\omega) $. By definition,
\begin{align*}
&\operatorname{Vol} \big{(}\varphi^{n}_{\omega}(g^{1}),...,\varphi^{n}_{\omega}(g^{m_{2}}),\varphi^{n}_{\omega}(h^{1}),...,\varphi^{n}_{\omega}(h^{m_{1}})\big{)} \\
=\ &\operatorname{Vol} \big{(}\varphi^{n}_{\omega}(g^{1}),...,\varphi^{n}_{\omega}(g^{m_{2}}) \big{)} \prod_{j=1}^{m_{1}}d\big{(}\varphi^{n}_{\omega}(h^{j}) , \langle \varphi^{n}_{\omega}(g^{1}),...,\varphi^{n}_{\omega}(g^{m_{2}}),\varphi^{n}_{\omega}(h^{1}),...,\varphi^{n}_{\omega}(h^{j-1}) \rangle\big{)}.
\end{align*}
Note that
\begin{align*}
1\leqslant\frac{d\big{(}\varphi^{n}_{\omega}(h^{j}), \langle \varphi^{n}_{\omega}(h^{1}),...,\varphi^{n}_{\omega}(h^{j-1}) \rangle \big{)}}{d\big{(}\varphi^{n}_{\omega}(h^{j}),\langle \varphi^{n}_{\omega}(g^{1}),...,\varphi^{n}_{\omega}(g^{m_{2}}),\varphi^{n}_{\omega}(h^{1}),..,\varphi^{n}_{\omega}(h^{j-1}) \rangle \big{)}}\leqslant \| \Pi_{\varphi(n,\omega,H_{\omega}) || F_{\mu_{2}}(\theta^{n}\omega)} \|.
\end{align*}
Together with Lemma \ref{Mea} and (\ref{det}) in Proposition \ref{imp}, this implies that
\begin{align}\label{VBN}
\begin{split}
&\lim_{n\rightarrow\infty}\frac{1}{n}\log \operatorname{Vol} \big{(}\varphi^{n}_{\omega}(h^{1}),...,\varphi^{n}_{\omega}(h^{m_{1}}),\varphi^{n}_{\omega}(g^{1}),...,\varphi^{n}_{\omega}(g^{m_{2}})\big{)}  \\
=\ &\lim_{n\rightarrow\infty}\frac{1}{n}\log \operatorname{Vol} \big{(}\varphi^{n}_{\omega}(g^{1}),...,\varphi^{n}_{\omega}(g^{m_{2}}),\varphi^{n}_{\omega}(h^{1}),...,\varphi^{n}_{\omega}(h^{m_{1}})\big{)} = m_{1}\mu_{1}+m_{2}\mu_{2}.\\
\end{split}
\end{align}
 Since $ f^{k} \in H_{\omega}  $ for $1\leqslant j\leqslant m_{1}$,
\begin{align*}
&d\big{(}\varphi^{n}_{\omega}(g^{j}), \langle \varphi^{n}_{\omega}(h^{1}),...,\varphi^{n}_{\omega}(h^{m_{1}}),\varphi^{n}_{\omega}(g^{1}),...,\varphi^{n}_{\omega}(g^{j-1}) \rangle \big{)}  \\
=\ &d\big{(}\varphi^{n}_{\omega}(h^{m_{1}+j}), \langle \varphi^{n}_{\omega}(h^{1}),...,\varphi^{n}_{\omega}(h^{m_{1}}),\varphi^{n}_{\omega}(h^{m_{1}+1}),..,\varphi^{n}_{\omega}(h^{m_{1}+j-1}) \rangle \big{)}.
\end{align*}
Consequently, by (\ref{VBN}),
\begin{align*}
&\lim_{n\rightarrow\infty}\frac{1}{n}\log \operatorname{Vol} \big{(}\varphi^{n}_{\omega}(h^{1}),...,\varphi^{n}_{\omega}(h^{m_{1}}),\varphi^{n}_{\omega}(h^{m_{1}+1}),...,\varphi^{n}_{\omega}(h^{m_{1}+m_{2}} )\big{)}  \\
=\ &\lim_{n\rightarrow\infty}\frac{1}{n}\log \operatorname{Vol} \big{(}\varphi^{n}_{\omega}(h^{m_{1}+1}),...,\varphi^{n}_{\omega}(h^{m_{1}+m_{2}}),\varphi^{n}_{\omega}(h^{1}),...,\varphi^{n}_{\omega}(h^{m_{1}})\big{)}\big{]} = m_{1}\mu_{1}+m_{2}\mu_{2}.
\end{align*}
This finishes step 2. We can now iterate the procedure and the general result follows by induction.

%
%
\end{proof}

\section{The Lyapunov spectrum for linear equations}\label{sec:lyapunovspectrum}

In this section, we formulate the main results of the article. 

\begin{theorem}\label{thm:main_linear}
 Let $(\Omega,\mathcal{F},\P,(\theta)_{t \in \R})$ be an ergodic measurable metric dynamical system and $\mathbf{X}$ a delayed $\gamma$-rough path cocycle for some $\gamma \in (1/3,1/2]$ and some delay $r > 0$. Assume that there are $\alpha < \beta < \gamma$ such that \eqref{eqn:alpha_beta_gamma} holds for some $\kappa \in (0,\gamma)$. In addition, we assume that
  \begin{align}\label{eqn:moment_bound_rough_cocycle} 
& \| X \|_{\gamma;[0,r]} + \|\mathbb{X}\|_{2 \gamma;[0,r]} + \|\mathbb{X}(-r)\|_{2 \gamma;[0,r]} \in L^{\frac{1}{\gamma - \beta}}(\Omega).
 \end{align}
 Let $\sigma \in L(W^2,L(U,W))$. Then we have the following:
 
 \begin{itemize}
  \item[(i)] The equation
 \begin{align}\label{eqn:main_thm_linear}
  \begin{split}
  dy_t &= \sigma(y_t,y_{t-r})\, d\mathbf{X}_t(\omega); \quad t \geq 0 \\
  y_t &= \xi_t; \quad t \in [-r,0]
  \end{split}
 \end{align}
 has a unique solution $y \colon [0,\infty) \to W$ for every initial condition $(\xi,\xi') \in \mathscr{D}^{\alpha,\beta}_{X(\omega)}([-r,0],W)$ with
 \begin{align*}
    (y_{t+n}(\omega), y'_{t+n}(\omega))_{t \in [-r,0]} \in \mathscr{D}_{X(\theta_{nr}\omega)}^{\alpha,\beta}([-r,0],W)
 \end{align*}
 for every $n \geq 0$ where
 \begin{align*}
  y'_t(\omega) = \begin{cases}
                  \sigma(y_t(\omega),y_{t-r}(\omega)) &\text{for } t \geq 0 \\
                  \xi'_t &\text{for } t \in [-r,0].
                 \end{cases}
 \end{align*}
 
 \item[(ii)] Set $\varphi(n,\omega,\xi) := (y_{t+n}(\omega), y'_{t+n}(\omega))_{t \in [-r,0]}$ and $E_{\omega} := \mathscr{D}^{\alpha,\beta}_{X(\omega)}([-r,0],W)$. Then $\varphi$ is a compact linear cocycle defined on the discrete ergodic measurable metric dynamical system $(\Omega,\mathcal{F},\P,\theta_r)$ acting on the measurable field of Banach spaces $\{E\}_{\omega \in \Omega}$ and all statements of the Multiplicative Ergodic Theorem \ref{thm:MET_Banach_fields} hold. In particular, a deterministic Lyapunov spectrum $(\mu_j)_{j \geq 0}$ exists and induces an Oseledets filtration of the space of admissible initial conditions $\mathscr{D}^{\alpha,\beta}_{X(\omega)}([-r,0],W)$ on a set of full measure.
 
 \end{itemize}

\end{theorem}

\begin{proof}
 Theorem \ref{fiber} together with Theorem \ref{eqn:delay_induce_RDS} show that \eqref{eqn:main_thm_linear} induces a cocycle acting on a measurable field of Banach spaces given by the spaces of controlled paths. The estimate in Theorem \ref{thm:delay_linear} together with our assumption \eqref{eqn:moment_bound_rough_cocycle} show that the moment condition of the MET \ref{thm:MET_Banach_fields} is satisfied and the theorem follows.
\end{proof}

Finally, we apply our results for the Brownian motion.

\begin{corollary}\label{cor:main_thm_bm}
 Theorem \ref{thm:main_linear}
 can be applied for $X$ being a two-sided Brownian motion $B$ adapted to a two-paramter filtration $(\mathcal{F}_s^t)$ and $\mathbf{X}$ being either $\mathbf{B}^{\text{It\=o}}$ or $\mathbf{B}^{\text{Strat}}$. In the former case, the solution to \eqref{eqn:main_thm_linear} coincides with the usual It\=o-solution and in the later case it coincides with the Stratonovich solution of a stochastic differential equation almost surely in case the initial condition is $\mathcal{F}_{-1}^0$-measurable. 
\end{corollary}

\begin{proof}
 The fact that $\mathbf{B}^{\text{It\=o}}$ and $\mathbf{B}^{\text{Strat}}$ are delayed $\gamma$-rough path cocycles on an ergodic measurable metric dynamical system for every $\gamma \in (1/3,1/2)$ was shown in Theorem \ref{thm:B_delayed_cocycle}. Choosing $\gamma$ close enough to $1/2$, we can find $\alpha$ and $\beta$ such that \eqref{eqn:alpha_beta_gamma} holds. In Proposition \ref{prop:existence_bm_lift}, we saw that the integrability condition \eqref{eqn:moment_bound_rough_cocycle} is satisfied in the Brownian case, and we can indeed apply Theorem \ref{thm:main_linear}.
 The fact that the solution to \eqref{eqn:main_thm_linear} coincides with the usual It\=o resp. Stratonovich solution was shown in Corollary \ref{cor:rough_ito_coincides}.
\end{proof}

We close this section with a few remarks.

\begin{remark}
 \begin{enumerate}
  \item As already mentioned, it is not hard to prove Theorem \ref{thm:main_linear} 
  for a vector of delays $0 < r_1 < \ldots < r_m$ in which case the equation reads
  \begin{align*}
   dy_t = \sigma(y_t,y_{t-r_1},\ldots,y_{t-r_m})\, d\mathbf{X}_t.
  \end{align*}
  In that case, the largest delay $r_m$ will play the role of $r$. It is also straightforward to include a smooth and bounded drift term in the equation by adding the function $t \mapsto t$ as a smooth component to the process $\mathbf{X}$. Including unbounded drifts is more challenging, cf. \cite{RS17} for a discussion regarding equations without delay. 

  
  \item Theorem \ref{thm:main_linear} 
  is formulated in a generality which opens the possibility to apply the results for a much larger class of driving processes $\mathbf{X}$. For instance, \cite{NNT08} prove that the fractional Brownian motion possess a ``canonical'' delayed L\'evy area using the Russo-Vallois integral \cite{RV93}. However, this approach does not directly show that the fractional Brownian motion has a canonical lift to a delayed rough path cocycle since we used that such lifts are limits of smooth convolutions, cf. the proof of Theorem \ref{thm:B_delayed_cocycle} where we used Theorem \ref{thm:approx_strat_delay}. However, it is possible to show that the delayed L\'evy area for the fractional Brownian motion defined through the Russo-Vallois integral is also a limit of smooth convolutions. This fact even holds for a significantly larger class of Gaussian processes and will be discussed in another future work. Other possible drivers in Theorem \ref{thm:main_linear} 
  are semimartingales with stationary increments and good integrability properties.
  
  
  \item It is possible to use the language of Hairer's Regularity Structures \cite{Hai14} to reformulate our results. In that case, the space of controlled paths has to be replaced by the space of \emph{modelled distributions}. We decided to use the language of rough paths here because less theory is needed and we can directly rely on prior work such as \cite{NNT08}. However, it might be useful to use regularity structures in the future.

 \end{enumerate}

\end{remark}

\section{An example}\label{sec:concrete_example}

In view of our main results obtained in the former section, we now come back to the previous example already discussed in the introduction: we consider the stochastic delay equation
\begin{align}\label{eqn:delayed_linear_chapter6}
  \begin{split}
  dy_t &= y_{t-1}\, d\mathbf{B}^{\text{It\=o}}_t;\quad t \geq 0 \\
  y_t &= \xi_t; \quad t \in [-1,0].
  \end{split}
 \end{align}
 This equation can be considered as the prototype of a singular stochastic delay equation. In its classical It\=o formulation, it was studied by one of us in \cite{Sch13}. In that work, it was shown that there exists a deterministic real number $\Lambda$ such that
 \begin{align}\label{eqn:limit_sch}
  \Lambda = \lim_{t \to \infty} \frac{1}{t} \log \| \varphi(t,\omega,\xi) \|
 \end{align}
 almost surely for any initial condition $\xi \in C([-1,0],\R) \setminus \{0\}$. In \eqref{eqn:limit_sch}, the norm $\| \cdot \|$ may denote the uniform norm or the $M_2$-norm which we will define below. It is a natural question to ask whether $\Lambda$ coincides with the top Lyapunov exponent provided by the Multiplicative Ergodic Theorem \ref{thm:MET_Banach_fields}. We will give an affirmative answer in this section.
 
 Set $ E_{\omega}=\mathscr{D}_{B(\omega)}^{\alpha ,\beta}([-1,0])$ with $\alpha$, $ \beta$ chosen such that  \eqref{eqn:alpha_beta_gamma} holds. Take $ (\xi,\xi') \in E_{\omega} $. On the time interval $[-1,1]$, the unique solution to \eqref{eqn:delayed_linear_chapter6} is given by
  \begin{align}\label{ans}
  (y_t,y'_t) = \begin{cases}
  (\xi_t,\xi'_t) &\text{ if } t\in [-1,0] \\
  \left( \int_0^t \xi_{s - 1}\, d\mathbf{B}^{\text{It\=o}} + \xi_0,\xi_{t-1} \right) &\text{ if } t\in [0,1].
  \end{cases}
 \end{align}
 Note that $C^1([-1,0],\R) \subset E_{\omega}$ for every $\omega \in \Omega$ by the embedding $\eta \mapsto (\eta,0)$. Let us introduce the Hilbert space $M_2 := \R \times L^2([-1,0],\R)$ furnished with the norm
 \begin{align*}
  \| (\nu,\eta) \|_{M_2} := \left( |\nu |^2 + \|\eta\|_{L^2}^2 \right)^{\frac{1}{2}}
 \end{align*}
 for $(\nu,\eta) \in M_2$. Note that $C([-1,0],\R) \subset M_2$ using the embedding $\eta \mapsto (\eta_0,\eta)$. Recall the definition of $\operatorname{Vol}$ given in Definition \ref{def:volume}. Our main result in this section is the following.
 
 \begin{theorem}\label{thm:vol_indep_limit}
  For every $\eta_{1},...,\eta_{k}\in C^{1}([-1,0], \R) \setminus \lbrace 0\rbrace $, the limit
    \begin{align}\label{lim_chap6}
      \lim_{n\rightarrow\infty} \frac{1}{n} \log \operatorname{Vol}\big{(}\varphi(n,\omega,\eta_{1}),...,\varphi(n,\omega,\eta_{k})\big{)}
    \end{align}
  exists almost surely in $[-\infty,\infty)$. Moreover, the limit is independent of the choice of the norm when we take $ \Vert \cdot \Vert_{E_{\theta^{n}\omega}}$, $\Vert \cdot \Vert_{C^{\alpha}}$, $\Vert \cdot \Vert_{\infty}$ or $ \Vert \cdot \Vert_{M_{2}}$ in the definition of $\operatorname{Vol}$. For $k = 1$, if $\| \cdot \|$ denotes any of the norms above, the limit
  \begin{align*}
    \lim_{n\rightarrow\infty} \frac{1}{n} \log  \| \varphi(n,\omega,\eta) \|
  \end{align*}
  is independent of the choice of $\eta \in C^{1}([-1,0], \R) \setminus \lbrace 0\rbrace$ and coincides with the largest Lyapunov exponent provided by the Multiplicative Ergodic Theorem \ref{thm:MET_Banach_fields}.

 \end{theorem}

%
%

Before proving Theorem \ref{thm:vol_indep_limit}, we need two classical inequalities:

\begin{lemma}\label{lem:GRR}
  Let $ \alpha<\frac{1}{2} $, $ p>2 $ and let $ \xi \colon [-1,0]\rightarrow\mathbb{R} $ be an $ \alpha $-H\"older path. Then there is a constant $A_{p}$ such that
  \begin{align}\label{ES1}
  \| \xi \|_{\alpha} = \sup_{-1 \leqslant s<t\leqslant 0}\frac{\vert \xi_{s,t}\vert}{(t-s)^{\alpha}} \leqslant A_{p} \bigg{(}\iint_{[-1,0]^{2}}\frac{| \xi_{u} - \xi_{v}\vert^{p}}{\vert u-v\vert^{p\alpha+2}}\, du\, dv \bigg{)}^{\frac{1}{p}}.
  \end{align} 
  If $X$ is $\alpha$-H\"older and $(\xi, \xi') \in \mathscr{D}_{X}^{\alpha}([-1,0],\R) $,
  \begin{align}\label{ES2}
  \sup_{-1\leqslant s<t\leqslant 0}\frac{\vert\xi^{\#}_{s,t}\vert}{(t-s)^{2\alpha}}\leqslant A_{p}\bigg{[}\bigg{(}\iint_{-1\leqslant u<v\leqslant 0}\frac{\vert\xi^{\#}_{u,v}\vert^{p}}{ |u - v|^{2\alpha p+2}}\, du \, dv\bigg{)}^{\frac{1}{p}} + \|\xi^{\prime}\|_{\alpha} \| X\|_{\alpha}\bigg{]}.
  \end{align}
\end{lemma}
\begin{proof}
Cf. \cite[Corollary 4]{Gub04}.
\end{proof}


\begin{proof}[Proof of Theorem \ref{thm:vol_indep_limit}]
First, we claim that the limit \eqref{lim_chap6} exists for any choice of $\eta_1,\ldots, \eta_k$ for the norm $ \Vert \cdot \Vert_{E_{\theta^{n}\omega}}$. Indeed, if $\eta_1,\ldots, \eta_k$ are linearly dependent, the limit \eqref{lim_chap6} clearly exists and equals $- \infty$. Also if for every $ j\geqslant 1 $ we have $ \langle \eta_{1},...,\eta_{k} \rangle\cap F_{\mu_{j}}(\omega)\neq \lbrace 0\rbrace$, since $ \mu_{j}\rightarrow-\infty $, Lemma \ref{estimate2} implies that \eqref{lim_chap6} exists and equals $- \infty$. So we can assume that for some  $ j\geqslant 1 $, $ \langle \eta_{1},...,\eta_{k} \rangle\cap F_{\mu_{j+1}}(\omega)=\lbrace 0\rbrace$.
 For $ i\leqslant j $ we can find a finite-dimensional subspace $ H_{i}(\omega) $ such that $ H_{i}(\omega)\bigoplus F_{\mu_{i+1}}(\omega)=F_{\mu_{i}}(\omega) $. Furthermore, for each $ i\leqslant j $, there is a subspace $ \tilde{H}_{i}(\omega)\subset H_{i}(\omega)$ with $ \operatorname{dim} \big{[}\tilde{H}_{i}(\omega)\big{]} = n_{i}$ such that
 \begin{align*}
 \frac{\langle \eta_{1},...,\eta_{k} \rangle}{F_{\mu_{j+1}}(\omega)}=\frac{\bigoplus_{1\leqslant i \leqslant j}\tilde{H}_{i}(\omega)}{F_{\mu_{j+1}}(\omega)}.
 \end{align*}
 Now as a consequence of item (v) in the Multiplicative Ergodic Theorem \ref{thm:MET_Banach_fields},
 \begin{align*}
\lim_{n\rightarrow\infty}\frac{1}{n} \log \operatorname{Vol} \big{(}\varphi(n,\omega,\eta_{1}),...,\varphi(n,\omega,\eta_{k})\big{)} = \sum_{1\leqslant i\leqslant j} n_{i} \mu_{i}
 \end{align*}
 which shows the claim. 
 
 The strategy of the proof now is to compare all norms against one another. For $ -1\leqslant t\leqslant 0 $ , $ \xi ,\eta \in C^{1}([-1,0],\R) \setminus \lbrace 0\rbrace  $ and $ n\in\mathbb{N}_0$ set $ \xi^{n}_{t}=y^{\xi}_{n+t} $ and $\eta^{n}_{t}=y^{\eta}_{n+t}$ where $ y^{\xi} $ and $y^{\eta}$ are solutions to \eqref{eqn:delayed_linear_chapter6} starting from $\xi$, $\eta$ respectively. By definition,
\begin{align}\label{EES}
(\xi^{n})^{\prime}_{t}=\xi^{n-1}_{t} , \ \ \ \ \ (\xi^{n})^{\#}_{s,t}=\int_s^t \xi^{n-1}_{s,u}dB_{n+u}
\end{align}
for $-1 \leq s \leq t \leq 0$ and $n \geq 0$ where we define $\xi^{-1} \equiv 0$. Set $\mathcal{F}_t := \mathcal{F}_0^t$. From Lemma \ref{lem:GRR}, for any $C > 0$,
\begin{align*}
&\P \big{(}\| \xi^{n} \|_{\alpha}>C\, |\, \mathcal{F}_{n-1}\big{)}\leqslant \P \left( \iint_{[-1,0]^{2}}\frac{\vert \xi_{v,u}^{n}\vert^{p}}{\vert u-v\vert^{2+p\alpha}}\, du\, dv \geqslant \frac{C^{p}}{(A_{p})^{p}} \big{|} \mathcal{F}_{n-1}\right) = \\ &\P \left( \iint_{[-1,0]^{2}} \frac{\vert\int_{[u,v]}\xi^{n-1}_{\tau } \,dB_{n+\tau}\vert^{p}}{\vert u-v\vert^{p\alpha +2}}\, du\, dv \geqslant \frac{C^{p}}{(A_{p})^{p}}|\mathcal{F}_{n-1} \right)
\end{align*}
almost surely. Similarly,
\begin{align*}
\P  \left( \inf_{\beta\in\mathbb{Q}} \| \eta^{n} - \beta\xi^{n} \|_{\alpha} > C  \big{|} \mathcal{F}_{n-1} \right)  \leqslant 
 \inf_{\beta\in\mathbb{Q}} \P\left( \iint_{[-1,0]^{2}}\frac{\vert\int_{[u,v]}\eta^{n-1}_{\tau} - \beta\xi^{n-1}_{\tau}\, dB_{n+\tau}\vert^{p}}{\vert u-v\vert^{p\alpha +2}}\, du\, dv \geqslant \frac{C^{p}}{(A_{p})^{p}}\big{|}\mathcal{F}_{n-1} \right)
\end{align*}
almost surely. Set $ p=2m $ for $m$ chosen such that $ m(1-2\alpha)>1 $. From the Burkholder-Davis-Gundy inequality, it follows that
\begin{align*}
\E \left( \left| \int_{[u,v]}\xi^{n-1}_{\tau}\, dB_{n+\tau} \right|^{2m} \big{|}\mathcal{F}_{n-1} \right) \leqslant B_{2m} \vert u-v\vert^{m}\Vert\xi^{n-1}\Vert^{2m}_{\infty}
\end{align*}
almost surely for some constant $B_{2m} > 0$. Consequently,
\begin{align}
&\P \left( \| \xi^{n} \|_{\alpha} > C | \mathcal{F}_{n-1} \right) \leqslant \tilde{A}_{2m} \frac{\| \xi^{n-1} \|^{2m}_{\infty}}{C^{2m}} \quad \text{and} \label{CDC} \\ &\P \left( \inf_{\beta\in\mathbb{Q}} \| \eta^{n}-\beta\xi^{n} \|_{\alpha} > C |\mathcal{F}_{n-1} \right) \leqslant \tilde{A}_{2m} \frac{\inf_{\beta\in\mathbb{Q}} \| \eta^{n-1} - \beta\xi^{n-1} \|_{\infty}^{2m}}{C^{2m}}
\end{align}
for a general constant $ \tilde{A}_{2m}$. Now for any $ \varepsilon>0 $, \eqref{CDC} implies that
\begin{align}\label{CCD}
\begin{split}
\P \left( \frac{1}{n} \log \| \xi^{n} \|_{\alpha} \geqslant \varepsilon + \frac{1}{n-1} \log \| \xi^{n-1} \|_{\infty}] \right) & \leqslant \P \left( \| \xi^{n-1} \|_{\alpha} \geqslant \| \xi^{n-1} \|_{\infty} \exp [\varepsilon(n-1)] \right) \\
&\leqslant \frac{\tilde{A}_{2m}}{\exp{[}2m\varepsilon(n-1){]}}\longrightarrow 0
\end{split}
\end{align}
as $n \to \infty$. Similarly,
\begin{align}\label{CCE}
  \P \left( \frac{1}{n} \log \inf_{\beta\in\mathbb{Q}} \| \eta^{n}-\beta\xi^{n} \|_{\alpha} \geqslant \varepsilon + \frac{1}{n-1} \log \inf_{\beta\in\mathbb{Q}} \|\eta^{n-1} - \beta\xi^{n-1} \|_{\infty} \right) \longrightarrow 0
\end{align}
as $n \to \infty$. 
Now from \eqref{ES2} and \eqref{EES},
\begin{align*}
&\P \left( \sup_{-1\leqslant s<t\leqslant 0} \frac{\vert(\xi^{n})^{\#}_{s,t}\vert}{(t-s)^{2\alpha}} > C\, \big{|}\, \mathcal{F}_{n-1} \right) \leqslant \\
&\P \left( A_{p} \left[ \left( \iint_{-1\leqslant u<v\leqslant 0} \frac{\vert \int_{u,v}\xi^{n-1}_{u,\tau}\, dB_{n+\tau} \vert^{p}}{(v-u)^{2p\alpha +2}}\,du \,dv \right)^{\frac{1}{p}} + \| \xi^{n-1} \|_{\alpha} \Vert B^{n} \Vert_{\alpha} \right] > C\, \big{|} \, \mathcal{F}_{n-1} \right) \leqslant \\
&\P \left( \iint_{-1\leqslant u<v\leqslant 0} \frac{\vert \int_{u,v}\xi^{n-1}_{u,\tau} \,dB_{n+\tau} \vert^{p}}{(v-u)^{2p\alpha +2}} \,du \, dv + \|\xi^{n-1} \|_{\alpha}^{p} \| B^{n} \|_{\alpha}^{p} > \frac{C^{p}}{(2A_{p})^{p}}\, \big{|}\, \mathcal{F}_{n-1} \right)
\end{align*}
almost surely. Similarly,
\begin{align*}
&\P \left( \inf_{\beta\in\mathbb{Q}} \left[ \sup_{-1\leqslant s<t\leqslant 0} \frac{\vert(\eta^{n}-\beta\xi^{n})^{\#}_{s,t}\vert}{(t-s)^{2\alpha}} \right] > C\, \big{|}\, \mathcal{F}_{n-1} \right) \leqslant \\
&\inf_{\beta\in\mathbb{Q}} \P \left( \iint_{-1\leqslant u<v\leqslant 0} \frac{\vert \int_{u,v}(\eta^{n}_{u,\tau}-\beta\xi^{n}_{u,\tau})\,dB_{n+\tau}\vert^{p}}{(v-u)^{2p\alpha +2}}\, du\, dv + \| \eta^{n-1}-\beta\xi^{n-1} \|_{\alpha}^{p} \| B^{n} \|_{\alpha}^{p} > \frac{C^{p}}{(2A_{p})^{p}}\, \big{|}\, \mathcal{F}_{n-1} \right)
\end{align*}
almost surely. Set $ p=2m $ such that $ m(1-2\alpha)>1 $. Then
\begin{align*}
  \E \left( \left| \int_{[u,v]}\xi^{n-1}_{u,\tau}\, dB_{n+\tau} \right|^{2m}\, \big{|}\, \mathcal{F}_{n-1} \right) \leqslant B_{2m}( v-u)^{m(2\alpha +1)}\Vert\xi^{n-1}\Vert_{\alpha}^{2m}
\end{align*}
almost surely. Consequently, for general constants $ M $ and $ \tilde{M} $,
\begin{align*}
\begin{split}
\P \left( \sup_{-1\leqslant s<t\leqslant 0}\frac{\vert(\xi^{n})^{\#}_{s,t}\vert}{(t-s)^{2\alpha}} > C\, \big{|}\, \mathcal{F}_{n-1} \right) & \leqslant \P \left( M \|\xi^{n-1}\|_{\alpha}^{2m} (1 + \| B^{n} \|_{\alpha}^{2m}) > C^{2m}\, \big{|}\, \mathcal{F}_{n-1} \right) \\
&\leqslant \frac{\tilde{M}}{C^{2m}} \| \xi^{n-1} \|^{2m}_{\alpha}
\end{split}
\end{align*}
almost surely and
\begin{align*}
  \P \left( \inf_{\beta\in\mathbb{Q}} \sup_{-1\leqslant s<t\leqslant 0} \frac{|(\eta^{n}-\beta\xi^{n})^{\#}_{s,t} |}{(t-s)^{2\alpha}} > C \, \big{|}\, \mathcal{F}_{n-1} \right) \leqslant \frac{\tilde{M}}{C^{2m}} \inf_{\beta\in\mathbb{Q}}\Vert\eta^{n-1} - \beta\xi^{n-1} \Vert^{2m}_{\alpha}
\end{align*}
almost surely. Similarly to \eqref{CCD}, for any $\varepsilon > 0$,
\begin{align}\label{CCF}
\begin{split}
  &\P \left( \frac{1}{n}\log \| (\xi^{n})^{\#}\|_{2\alpha} 
  \geqslant \varepsilon+\frac{1}{n-1}\log \Vert\xi^{n-1}\Vert_{\alpha}] \right) \longrightarrow 0\qquad \text{and} \\
  &\P \left( \frac{1}{n}\log \inf_{\beta\in\mathbb{Q}} 
  \| (\eta^{n}-\beta\xi^{n})^{\#} \|_{2\alpha} \geqslant \varepsilon + \frac{1}{n-1}\log \inf_{\beta\in\mathbb{Q}}\Vert\eta^{n-1}-\beta\xi^{n-1}\Vert_{\alpha}] \right) \longrightarrow 0
\end{split}
\end{align}
as $n \to \infty$. 
Remember $ \Vert\xi^{n}\Vert_{M_{2}}^{2}=\vert\xi_{-1}^{n}\vert^{2}+\int_{-1}^{0}(\xi_{t}^{n})^{2}\,dt$. From Doob's submartingale inequality, for a general constant $ M $,
\begin{align*}
\begin{split}
  \P(\| \xi^{n} \|_{\infty} >C\,|\,\mathcal{F}_{n-1}) &\leqslant \P \left( \vert\xi_{-1}^{n}\vert + \sup_{-1\leqslant t\leqslant 0}\vert\xi_{-1,t}^{n}\vert >C\,|\, \mathcal{F}_{n-1} \right) \\
  &\leqslant \frac{4\vert\xi_{-1}^{n}\vert^{2} + 4 \E \vert\xi_{-1,0}^{n}\vert^{2}}{C^{2}}\leqslant \frac{M}{C^{2}}\Vert\xi^{n-1}\Vert_{M_{2}}^{2}
\end{split}
\end{align*}
almost surely. Also,
\begin{align*}
\P \left( \inf_{\beta\in\mathbb{Q}} \| \eta^{n}-\beta\xi^{n} \|_{\infty} > C\, |\, \mathcal{F}_{n-1} \right) \leqslant \frac{M}{C^{2}} \inf_{\beta\in\mathbb{Q}} \|\eta^{n-1} - \beta\xi^{n-1} \|_{M_{2}}^{2}
\end{align*}
almost surely. Again as in \eqref{CCD}, for any $\varepsilon > 0$,
\begin{align}\label{CCG}
\begin{split}
&\P \left( \frac{1}{n}\log \inf_{\beta\in\mathbb{Q}} \| \xi^{n} \|_{\infty} \geqslant \varepsilon + \frac{1}{n-1} \log \inf_{\beta\in\mathbb{Q}} \| \xi^{n-1} \|_{M_{2}} \right) \longrightarrow 0\qquad \text{and} \\
&\P \left( \frac{1}{n} \log \inf_{\beta\in\mathbb{Q}} \| \eta^{n}-\beta\xi^{n} \|_{\infty} \geqslant \varepsilon + \frac{1}{n-1} \log \inf_{\beta\in\mathbb{Q}} \|\eta^{n-1} - \beta\xi^{n-1}\|_{M_{2}} \right) \longrightarrow 0
\end{split}
\end{align}
as $n \to \infty$. Now from the Multiplicative Ergodic Theorem \ref{thm:MET_Banach_fields}, \eqref{CCD}, \eqref{CCE}, \eqref{CCF} and \eqref{CCG}, the following limits exist
\begin{align}\label{top}
\begin{split}
&\lim_{n\rightarrow\infty}\frac{1}{n}\log\Vert\varphi(n,\omega,\xi) \Vert\\
&\lim_{n\rightarrow\infty}\frac{1}{n} \log \operatorname{Vol} \big{(}\varphi(n,\omega,\xi), \varphi(n,\omega,\eta)\big{)}
\end{split}
\end{align}
as $n \to \infty$ where $\| \cdot \|$ could be any of the proposed norms, used also in the definition of $\operatorname{Vol}$, and the limit is independent of the choice of the norm. From the definition of $\operatorname{Vol}$, the above argument together with a simple induction generalizes to every $k\geqslant 1$ which proves the first claim.

To prove the second claim, let $ \eta\in C^{1}([-1,0],\R) \setminus \lbrace 0\rbrace $. Then the limit $\mu := \lim_{n\rightarrow\infty} \frac{1}{n} \log \Vert\varphi(n,\omega,\eta)\Vert$ is independent from $ \eta $, cf. \cite[Theorem 1.1]{Sch13}. Therefore, from the Multiplicative Ergodic Theorem \ref{thm:MET_Banach_fields}, $ C^{1}([-1,0],\R) \setminus\lbrace 0\rbrace\subset F_{\mu_{j}}(\omega)\setminus F_{\mu_{j+1}}(\omega) $ for some $j \geq 1$. Let $ \xi$, $\eta\in C^{\infty}([-1,0],\R)\setminus\lbrace 0\rbrace $, $ a\in\mathbb{R} $ and set $ \tilde{\xi}_{t} := \int_{-1}^{t} \xi_{\tau}\, dB_{\tau} $. Using \eqref{ans}, we have
 \begin{align}\label{eqn:dense_subset}
  \gamma_{t} := \tilde{\xi}_{t} + \eta_{t} + a = \varphi(1, \theta^{-1} \omega,\xi)[t] + \eta_{t} +a -\xi_{0}
 \end{align}
 and \eqref{top} implies that $ \lim_{n\rightarrow\infty} \frac{1}{n} \log \| \varphi(n,\omega, \gamma) \| \leqslant\mu$. From Theorem \ref{structre}, we know that elements of the form $\gamma$ are dense in $ E_{\omega}$. Choose $ \xi_{\omega}\in F_{\mu_{1}}(\omega)\setminus F_{\mu_{2}}(\omega) $. Since  $ F_{\mu_{2}}(\omega) $ is a closed subspace, we can find a neighborhood $ B(\xi_{\omega} ,\delta)\subset F_{\mu_{1}}(\omega)\setminus F_{\mu_{2}}(\omega) $ and an element $\gamma \in B(\xi_{\omega} ,\delta)$ of the form \eqref{eqn:dense_subset}. Therefore, $\mu_1 \leq \mu$, thus $\mu = \mu_1$.
\end{proof}

\begin{remark}
Taking the Hilbert space norm $ \Vert \cdot \Vert_{M_{2}} $ in the definition of $\operatorname{Vol}$, we actually have
\begin{align*}
  \operatorname{Vol}(X_{1},...,X_{k})= \Vert X_{1}\wedge X_{2}\wedge ...\wedge X_{k}\Vert_{M_{2}}.
\end{align*}
We conjecture that the limit
    \begin{align*}
      \lim_{n\rightarrow\infty} \frac{1}{n} \log \Vert \varphi(n,\omega,\eta_{1}) \wedge \cdots \wedge \varphi(n,\omega,\eta_{k}) \Vert_{M_2}
    \end{align*}
is independent of the choice of $\eta_1,\ldots,\eta_k$ whenever these vectors are linearly independent, and that the limit coincides with $\Lambda_k$ almost surely. This would be in good accordance with the classical definition of Lyapunov exponents in the finite dimensional case, cf. \cite[Chapter 3]{Arn98}.

\end{remark}

%
%
%
%
%
%

\appendix


\section{Stability for rough delay equations}

In the following, we sketch the proof of Theorem \ref{thm:delay_stability}. The strategy is the same as in \cite[Theorem 4.2]{NNT08}.
%
%

 \begin{proof}[Proof of Theorem \ref{thm:delay_stability} (sketch)]\label{proof:stability}
 For simplicity, we assume that $U=W=\mathbb{R}$. By definition,
\begin{align}\label{DED}
y_{s,t} = \int_{s}^{t}\sigma(y_{\tau},\xi_{\tau -r})\, d\mathbf{X}_{\tau} 
=\Lambda_{s,t}+\rho^{2}_{s,t} = \sigma(y_{s},\xi_{s-r})X_{s,t} + \rho^{1}_{s,t} + \rho^{2}_{s,t}
\end{align}
where
\begin{align*}
 \Lambda_{s,t} &= \sigma(y_{s},\xi_{s-r})X_{s,t} + \sigma_{1}(y_{s},\xi_{s-r})y^{\prime}_{s}\mathbb{X}_{s,t} + \sigma_{2}(y_{s},\xi_{s-r})\xi^{\prime}_{s-r}\mathbb{X}_{s,t}(-r) , \\
 \rho^{1}_{s,t} &= \sigma_{1}(y_{s},\xi_{s-r})y^{\prime}_{s}\mathbb{X}_{s,t} + \sigma_{2}(y_{s},\xi_{s-r})\xi^{\prime}_{s-r}\mathbb{X}_{s,t}(-r)  \quad \text{and} \\
 \rho^{2}_{s,t} &= \int_{s}^{t}\sigma(y_{\tau},\xi_{\tau -r})\, d\mathbf{X}_{\tau} - \Lambda_{s,t},
\end{align*}
using the notation $\sigma_1(x,y) = \partial_x \sigma(x,y)$, $\sigma_2(x,y) = \partial_y \sigma(x,y)$. Analogously, one defines $\tilde{\Lambda}$, $\tilde{\rho}^1$ and $\tilde{\rho}^2$ such that
\begin{align*}
  \tilde{y}_{s,t} = \tilde{\Lambda}_{s,t} + \tilde{\rho}^{2}_{s,t} = \sigma(\tilde{y}_{s},\tilde{\xi}_{s-r}) \tilde{X}_{s,t} + \tilde{\rho}^{1}_{s,t} + \tilde{\rho}^{2}_{s,t}.
\end{align*}
Note that $y'_s = \sigma(y_{s},\xi_{s-r})$ and $y^{\#}_{s,t}=\rho^{1}_{s,t}+\rho^{2}_{s,t} $. It is not hard to see  that
\begin{align}\label{ZXZA}
 \begin{split}
  \Lambda_{s,t}&-\Lambda_{s,u} - \Lambda_{u,t}  = [\sigma_{1}(y_{s},\xi_{s-r})y^{\#}_{s,u} + \sigma_{2}(y_{s},\xi_{s-r})\xi^{\#}_{s-r,u-r}] X_{u,t}\\&+[\sigma_{1}(y_{u},\xi_{u-r})y^{\prime}_{u}-\sigma_{1}(y_{s},\xi_{s-r})y^{\prime}_{s}]\mathbb{X}_{u,t}+[\sigma_{2}(y_{u},\xi_{u-r})\xi^{\prime}_{u-r}-\sigma_{2}(y_{s},\xi_{s-r})\xi^{\prime}_{s-r}]\mathbb{X}_{u,t}(-r) \\&+ \int_{0}^{1}(1-\tau)\big{[}\sigma_{1,1}(z^{\tau}_{s,u},\bar{z}^{\tau}_{s,u})(y_{s,u})^{2}+2\sigma_{1,2}(z^{\tau}_{s,u},\bar{z}^{\tau}_{s,u})y_{s,y}\xi_{s-r,u-r}+\sigma_{2,2}(z^{\tau}_{s,u},\bar{z}^{\tau}_{s,u})(\xi_{s-r,u-r})^{2}\big{]}d\tau X_{u,t}
 \end{split}
\end{align}
where $ z^{\tau}_{s,u}=\tau y_{u}+(1-\tau)y_{s} $, $ \bar{z}^{\tau}_{s,u}=\tau\xi_{u-r}+(1-\tau)\xi_{s-r} $ and $\sigma_{1,1}(x,y) = \partial^2_x \sigma(x,y)$, $\sigma_{1,2}(x,y) = \partial_x \partial_y \sigma(x,y)$ and $\sigma_{2,2}(x,y) = \partial^2_y \sigma(x,y)$. Set 
\begin{align*}
 R &:= \Vert X-\tilde{X}\Vert_{\gamma,[0,r] } + \Vert \mathbb{X} - \tilde{\mathbb{X}} \Vert_{2\gamma ,[0,r]} + \Vert \mathbb{X}(-r)-\tilde{\mathbb{X}}(-r)\Vert_{2\gamma ,[0,r]} \\
 &\qquad + \Vert\xi^{\prime}-\tilde{\xi}^{\prime}\Vert_{\beta ,[0,r]}+\Vert\xi^{\#}-\tilde{\xi}^{\#}\Vert_{2\beta ,[0,r]}+\Vert\xi -\tilde{\xi}\Vert_{\beta ,[0,r]}, \\
 C(y) &:= \Vert X\Vert_{\gamma} + \Vert\mathbb{X}\Vert_{2\gamma ,[0,r]} + \Vert\mathbb{X}(-r)\Vert_{2\gamma ,[0,r]} + \Vert y\Vert_{\mathscr{D}_{{X}}^\beta([0 ,r],W)}+\Vert\xi\Vert_{\mathscr{D}_{{X}}^\beta([-r,0],W)} \quad \text{and} \\
  D(X) &:= \Vert X\Vert_{\gamma} + \Vert\mathbb{X} \Vert_{2\gamma} + \Vert\mathbb{X}(-r)\Vert_{2\gamma} + \Vert\xi\Vert_{\mathscr{D}_{{X}}^\beta([0 ,r],W)}
 \end{align*}
with an analogous definition of $C(\tilde{y})$ and $D(\tilde{X})$. It is not hard to see that there is a continuous function  $ g:(0,\infty)^{4}\rightarrow [0,\infty) $, increasing in every of its arguments, such that
\begin{align*}
\Vert \rho^{1} -\tilde{\rho}^{1}\Vert_{2\beta ;[a,b]}&\leqslant (b-a)^{\gamma -\beta}g\big{[}D(X),D(\tilde{X}),C(y),C(\tilde{y}) \big{]}\\ &\big{[}R +\Vert y-\tilde{y}\Vert_{\beta ;[a,b]}+\Vert y^{\prime}-\tilde{y}^{\prime}\Vert _{\beta ;[a,b]}+ \Vert y^{\#}-\tilde{y}^{\#}\Vert_{2\beta ;[a,b]}\big{]}
\end{align*} 
for every $[a,b] \subseteq [0,r]$. From the Sewing lemma \cite[Lemma 4.2]{FH14},

\begin{align*}
\Vert\rho^{2} -  \tilde{\rho}^{2}\Vert_{2\beta ;[a,b]} \leqslant M \sup_{s,u,t\in [a,b]}\frac{\big{\vert}(\Lambda_{s,t}-\tilde{\Lambda}_{s,t})-(\Lambda_{s,u}-\tilde{\Lambda}_{s,u})-(\Lambda_{u,t}-\tilde{\Lambda}_{u,t})\big{\vert}}{(t-s)^{2\beta}}
\end{align*}
for some constant $M > 0$.
Using (\ref{ZXZA}), one can deduce that
\begin{align*}
&\sup_{s,u,t\in [a,b]}\frac{\big{\vert}(\Lambda_{s,t}-\tilde{\Lambda}_{s,t})-(\Lambda_{s,u}-\tilde{\Lambda}_{s,u})-(\Lambda_{u,t}-\tilde{\Lambda}_{u,t})\big{\vert}}{(t-s)^{2\beta}}\\ &\leqslant (b-a)^{\gamma -\beta}g\big{[}D(X),D(\tilde{X}),C(y),C(\tilde{y}) \big{]}\big{[}R +\Vert y-\tilde{y}\Vert_{2\beta ;[a,b]}+\Vert y^{\prime}-\tilde{y}^{\prime}\Vert _{\beta ;[a,b]}+ \Vert y^{\#}-\tilde{y}^{\#}\Vert_{2\beta ;[a,b]}\big{]}.
\end{align*}
Now, along with $ (\ref{DED}) $,
\begin{align*}
&\Vert y-\tilde{y}\Vert_{\beta;[a,b]}+\Vert y^{\prime}-\tilde{y}^{\prime}\Vert_{\beta ;[a,b]}+\Vert y^{\#}-\tilde{y}^{\#}\Vert_{2\beta;[a,b]}\leqslant\\
&(b-a)^{\gamma -\beta}\tilde{g}\big{[}D(X),D(\tilde{X}),C(y),C(\tilde{y}) \big{]}\big{[}R +\Vert y-\tilde{y}\Vert_{\beta ;[a,b]}+\Vert y^{\prime}-\tilde{y}^{\prime}\Vert _{\beta ;[a,b]}+ \Vert y^{\#}-\tilde{y}^{\#}\Vert_{2\beta ;[a,b]}\big{]}
\end{align*}
with $ \tilde{g} $ being a continuous increasing function. Using the bounds for the norm of $y$ and $\tilde{y}$ provided in \cite[Equation (62)]{NNT08}, we can find an increasing continuous function $ H:(0,\infty)^{2}\rightarrow [0,\infty) $ such that
\begin{align*}
&\Vert y-\tilde{y}\Vert_{\beta;[a,b]}+\Vert y^{\prime}-\tilde{y}^{\prime}\Vert_{\beta ;[a,b]}+\Vert y^{\#}-\tilde{y}^{\#}\Vert_{2\beta;[a,b]}\leqslant \\ &(b-a)^{\gamma -\beta}H[D(X),D(\tilde{X})] [R+\Vert y-\tilde{y}\Vert_{\beta;[a,b]}+\Vert y^{\prime}-\tilde{y}^{\prime}\Vert_{\beta ;[a,b]}+\Vert y^{\#}-\tilde{y}^{\#}\Vert_{2\beta;[a,b]}].
\end{align*}
Now by the same argument as for the linear case, cf. the proof of Theorem \ref{thm:delay_linear}, one sees that
\begin{align}\label{MNNM}
\begin{split}
&\Vert y-\tilde{y}\Vert_{\beta;[0,r]}+\Vert y^{\prime}-\tilde{y}^{\prime}\Vert_{\beta ;[0,r]}+\Vert y^{\#}-\tilde{y}^{\#}\Vert_{2\beta;[0,r]}\leqslant F[D(X)+D(\tilde{X})]R 
\end{split}
\end{align}
holds for an increasing continous function $F$. The claim follows from (\ref{MNNM}) .
 \end{proof}

\section{A pathwise MET}

 
 \begin{proof}[Proof of Proposition \ref{imp}]\label{proof_pathwise_MET}
  For given $n \in \N$, let $E^{1}_{n} := \langle e^1_n,\ldots, e^m_n \rangle$ be an $m$-dimensional subspace of $ V_{0} $ with $ \Vert e^{i}_{n}\Vert =1 $ and
 \begin{align}\label{MKMM}
    \operatorname{Vol}(T^{n}e^{1}_{n},...,T^{n}e^{m}_{n})\geqslant\frac{1}{2}D_{m}(T^{n}).
 \end{align}
 By \cite[Lemma 2.3]{Blu16}, we can find a closed complement subspace $F^{2}_{n}$ to $ E^{2}_{n} := T^{n}E^{1}_{n} $ in $V_{n}$ such that for $ P_{n}^{2}:=\Pi_{E^{2}_{n}||F^{2}_{n}} $,
 \begin{align*}
\Vert P_{n}^{2}\Vert\leqslant \sqrt{m}.
 \end{align*}
 Let $ F^{1}_{n}:= \lbrace v\in V_{0} : T^{n}v\in F^{2}_{n}\rbrace$. One can check
 that $F^1_n$ is a closed complement subspace to $ E^{1}_{n} $. Set $P^{1}_{n} :=\Pi_{E^{1}_{n}||F^{1}_{n}}$. From Lemma \ref{estimate2} and (\ref{MKMM}), it follows that there is a constant $\alpha_m$ such that for any $ v\in E^{1}_{n} $,
 \begin{align}\label{TYT}
 \frac{\Vert T^{n}v\Vert}{\Vert v\Vert}\geqslant \frac{D_{m}(T^{n})}{2\alpha_{m}\Vert T^{n}\Vert^{m-1}}.
 \end{align}
 From $ P^{1}_{n}=(T^{n}|_{E^{1}_{n}})^{-1}P^{2}_{n}T^{n} $, (\ref{TYT}) implies that
 \begin{align}
 \Vert P^{1}_{n}\Vert\leqslant (m+1) \Vert T^{n}\Vert \Vert (T^{n}|_{E^{1}_{n}})^{-1}\Vert\leqslant \frac{2\alpha_{m}\Vert T^{n}\Vert^{m}}{D_{m}(T^{n})}.
 \end{align}
 Let $ v\in F^{1}_{n} $ with $ \Vert v\Vert =1 $. Then 
 \begin{align}\label{CVCV}
  \operatorname{Vol} (T^{n}e^{1}_{n},...,T^{n}e^{m}_{n},T^{n}v) = \operatorname{Vol}(T^{n}e^{1}_{n},...,T^{n}e^{m}_{n})\, d(T^{n}v, \langle T^{n}e^{1}_{n},...,T^{n}e^{m}_{n} \rangle).
 \end{align}
 Since $d(T^{n}v, \langle T^{n}e^{1}_{n},...,T^{n}e^{m}_{n} \rangle) = \inf_{\beta_{j}\in \mathbb{R}}\Vert T^{n}v -\sum_{1\leqslant j\leqslant m}\beta_{j}T^{n}e^{j}_{n} \Vert  $, we see that
\begin{align*}
\frac{\Vert T^{n}v\Vert}{d(T^{n}v, \langle T^{n}e^{1}_{n},...,T^{n}e^{m}_{n} \rangle)}\leqslant \Vert P^{2}_{n} \Vert +1 \leqslant \sqrt{m} + 1.
\end{align*} 
Consequently, from (\ref{MKMM}) and (\ref{CVCV}),
\begin{align}\label{eqn:key_est_blum}
\Vert T^{n}v\Vert\leqslant\frac{2( \sqrt{m} + 1) D_{m+1}(T^{n})}{D_{m}(T^{n})}.
\end{align}
The rest of the proof is almost identical to the original proof of \cite[Proposition 3.4]{Blu16}. First, one can show that the sequence of subspaces $(F^n)$ converge to $F$ in the Hausdorff distance at a sufficiently fast exponential rate, cf. \cite[Claim 3 on page 2396]{Blu16}. Together with \eqref{eqn:key_est_blum}, this implies the bound
\begin{align*}
\limsup_{n\rightarrow\infty}\frac{1}{n}\log \Vert T^{n}|_{F} \Vert \leqslant\underline{l}
\end{align*}
which was announced in Remark \eqref{remark:after_prop_key}. From the convergence, we can also deduce that $F$ is closed and $m$-codimensional. The identities \eqref{eqn:conv_overlinel} and \eqref{uniform} can be proved exactly as in \cite{Blu16}. To see \eqref{det},
%
%
%
let $ H = \langle h_{1},...,h_{m}\rangle $ be a complement subspace to $ F $. Note that, from (\ref{uniform}) and assumption (ii), for any $ \delta >0 $, we can choose $n$ large enough such that
\begin{align*}
\exp\big{(}n(\overline{l}-\delta )\big{)}\leqslant\frac{\Vert T^{n}v\Vert}{\Vert v\Vert}\leqslant\exp\big{(} n(\overline{l}+\delta )\big{)} 
\end{align*}
holds for all $v \in H$. Consequently, 
\begin{align*}
\exp\big{(}n(\overline{l}-\delta )\big{)}\leqslant\frac{d\big{(}T^{n}h_{j}, \langle T^{n} h_{i} \rangle_{1\leqslant i< j}\big{)}}{d\big{(} h_{j}, \langle h_{i} \rangle_{1\leqslant i< j}\big{)}}\leqslant\exp\big{(} n(\overline{l}+\delta )\big{)}
\end{align*}
for all $1 \leq j \leq m$ and (\ref{det}) follows.
 \end{proof}

\subsection*{Acknowledgements}
\label{sec:acknowledgements}

MGV acknowledges a scholarship from the Berlin Mathematical School (BMS). SR and MS acknowledge financial support by the DFG via Research Unit FOR 2402. All authors would like to thank A.~Blumenthal for sending us a corrected version of the proof of \cite[Lemma 3.7]{Blu16} and A.~Schmeding for valuable discussions and comments during the preparation of the manuscript.

\bibliographystyle{alpha}
\bibliography{refs}

\end{document}